\documentclass[12pt,reqno]{article}
\usepackage{amssymb}
\usepackage[mathscr]{eucal}
\usepackage{amsmath,amssymb,latexsym,theorem,enumerate,ifthen}
\usepackage[english]{babel}
\usepackage{color,url}
\usepackage{appendix}
\usepackage{graphicx}
\usepackage{subcaption}

\newcommand{\NN}{\mathbb{N}}
\newcommand{\QQ}{\mathbb{Q}}
\newcommand{\RR}{\mathbb{R}}

\newcommand{\ZZ}{\mathbb{Z}}

\newcommand{\bA}{{\boldsymbol{A}}}

\newcommand{\bG}{{\boldsymbol{G}}}

\newcommand{\bI}{{\boldsymbol{I}}}

\newcommand{\bM}{{\boldsymbol{M}}}

\newcommand{\bV}{{\boldsymbol{V}}}

\newcommand{\bx}{{\boldsymbol{x}}}
\newcommand{\bX}{{\boldsymbol{X}}}

\newcommand{\bGamma}{{\boldsymbol{\Gamma}}}

\newcommand{\blambda}{{\boldsymbol{\lambda}}}

\newcommand{\bbeta}{{\boldsymbol{\beta}}}

\newcommand{\Beta}{{\boldsymbol{\eta}}}

\newcommand{\bone}{{\boldsymbol{1}}}

\newcommand{\cA}{{\mathcal A}}

\newcommand{\cX}{{\mathcal X}}

\newcommand{\ee}{\mathrm{e}}

\DeclareMathOperator*{\conv}{conv}

\newcommand{\EE}{\operatorname{\mathbb{E}}}

\newcommand{\PP}{\operatorname{\mathbb{P}}}

\newcommand{\cov}{\operatorname{Cov}}
\newcommand{\rank}{\operatorname{rank}}
\newcommand{\sign}{\operatorname{sign}}

\DeclareMathOperator*{\argmin}{arg\,min}

\newcommand{\halpha}{\widehat{\alpha}}

\newcommand{\comment}[1]{}

\renewcommand{\mid}{\,|\,}

\renewcommand{\leq}{\leqslant}
\renewcommand{\geq}{\geqslant}

\newcommand{\proofend}{\hfill\mbox{$\Box$}}

\numberwithin{equation}{section}

\theoremstyle{change} \theorembodyfont{\em}
\newtheorem{Lem}{Lemma.}[section]
\newtheorem{Thm}[Lem]{Theorem.}
\newtheorem{Pro}[Lem]{Proposition.}

\newtheorem{Cor}[Lem]{Corollary.}

\newtheorem{Def}[Lem]{Definition.}

\newtheorem{Rem}[Lem]{Remark.}

\theorembodyfont{\rm}
\newtheorem{Ex}[Lem]{Example.}

\def\eq#1{{\rm(\ref{#1})}}
\long\def\Eq#1#2{\ifthenelse{\equal{#1}{*}}
  {\begin{equation*}\begin{aligned}#2\end{aligned}\end{equation*}}
  {\begin{equation}\begin{aligned}\label{#1}#2\end{aligned}\end{equation}}}

\def\OnlyOnArXiv#1#2{\ifthenelse{\equal{#1}{Y}}{#2}{}}

\newenvironment{proof}{\noindent{\bf Proof.}}{\proofend}

\begin{document}

\begin{center}
 {\bfseries\Large Comparison and equality of generalized $\psi$-estimators}

\vspace*{2mm}

{\sc\large
  M\'aty\'as $\text{Barczy}^{*,\diamond}$,
  Zsolt $\text{P\'ales}^{**}$ }

\end{center}

%\vskip5mm

\noindent
 * HUN-REN–SZTE Analysis and Applications Research Group,
   Bolyai Institute, University of Szeged,
   Aradi v\'ertan\'uk tere 1, H--6720 Szeged, Hungary.

\noindent
 ** Institute of Mathematics, University of Debrecen,
    Pf.~400, H--4002 Debrecen, Hungary.

\noindent e-mail: barczy@math.u-szeged.hu (M. Barczy),
                  pales@science.unideb.hu  (Zs. P\'ales).

\noindent $\diamond$ Corresponding author.

\vskip0.1cm

%\centerline{\sl March 10, 2018}

{\renewcommand{\thefootnote}{}
\vspace*{0.2cm}
\footnote{M\'aty\'as Barczy was supported by the project TKP2021-NVA-09.
Project no.\ TKP2021-NVA-09 has been implemented with the support provided by the Ministry of Culture and Innovation of Hungary from the National Research, Development and Innovation Fund,
 financed under the TKP2021-NVA funding scheme.
Zsolt P\'ales is supported by the K-134191 NKFIH Grant.}}

\vspace*{-15mm}

\begin{abstract}
We solve the comparison problem for generalized $\psi$-estimators introduced in Barczy and P\'ales (2022).
Namely, we derive several necessary and sufficient conditions under which a generalized
 $\psi$-estimator less than or equal to another $\psi$-estimator for any sample.
We also solve the corresponding equality problem for generalized $\psi$-estimators.
For applications, we solve the two problems in question for Bajraktarevi\'c-type- and quasi-arithmetic-type estimators.
We also apply our results for some known statistical estimators such as for empirical expectiles and Mathieu-type
 estimators and for solutions of likelihood equations in case of normal, a Beta-type, Gamma, Lomax (Pareto type II),
 lognormal and Laplace distributions.
\end{abstract}

\tableofcontents

\section{Introduction}
\label{section_intro}

In this paper, we solve the comparison and equality problems for $\psi$-estimators (also called $Z$-estimators)
 that have been playing an important role in statistics since the 1960's.
Let $(X,\cX)$ be a measurable space, $\Theta$ be a Borel subset of $\RR$, and $\psi:X\times\Theta\to\RR$
 be a function such that for all $t\in\Theta$, the function $X\ni x\mapsto \psi(x,t)$ is measurable
 with respect to the sigma-algebra $\cX$.
Let $(\xi_n)_{n\geq 1}$ be a sequence of i.i.d.\ random variables with values in $X$ such that the distribution of $\xi_1$ depends on an unknown parameter $\vartheta \in\Theta$.
For each $n\geq 1$, Huber \cite{Hub64, Hub67} among others introduced an important estimator of $\vartheta$ based on the observations
 $\xi_1,\ldots,\xi_n$ as a solution $\widehat\vartheta_{n,\psi}(\xi_1,\ldots,\xi_n)$ of the equation
 (with respect to the unknown parameter):
 \[
    \sum_{i=1}^n \psi(\xi_i,t)=0, \qquad t\in\Theta,
 \]
 provided that such a solution exists.
In the statistical literature, one calls $\widehat\vartheta_{n,\psi}(\xi_1,\ldots,\xi_n)$ a $\psi$-estimator of the unknown parameter $\vartheta\in\Theta$ based on
 the i.i.d.\ observations $\xi_1,\ldots,\xi_n$, while other authors call it a Z-estimator (the letter Z refers to ''zero'').
In fact, $\psi$-estimators are special M-estimators (where the letter M refers to ''maximum likelihood-type'')
 that were also introduced by Huber \cite{Hub64, Hub67}.
For a detailed exposition of M-estimators and $\psi$-estimators, see, e.g.,
 Kosorok \cite[Subsection 2.2.5 and Section 13]{Kos} or van der Vaart \cite[Section 5]{Vaa}.

According to our knowledge, results on the comparison and the equality of $\psi$-estimators or $M$-estimators are not available in the literature.
In other words, given $\psi,\varphi:X\times\Theta\to\RR$ (with the properties described above),
 we are interested in finding necessary as well as sufficient conditions for the inequality
 $\widehat\vartheta_{n,\psi}\leq \widehat\vartheta_{n,\varphi}$ and for the equality $\widehat\vartheta_{n,\psi}=\widehat\vartheta_{n,\varphi}$
 to be valid for all $n\geq1$, respectively.
In this paper, we make the first steps to fill this gap in case of generalized $\psi$-estimators introduced in Barczy and P\'ales \cite{BarPal2}
 (see also Definition \ref{Def_Tn}), which are generalizations of $\psi$ estimators recalled above.

We mention that in general linear models, many authors investigated a related question, namely,  the equality of the ordinary least squares estimator (OLSE) and the best linear unbiased estimator (BLUE) of the regression parameters.
In fact, OLSE is a special $M$-estimator (see, e.g., van der Vaart \cite[Examples 5.27 and 5.28]{Vaa}),
and the solution of the related minimization problem is usually traced back to solve a system of equations, and hence OLSE can be also considered as a $Z$-estimator.
For a detailed discussion on results for the equality of the OLSE and the BLUE of the regression parameters in general linear models, see Section \ref{Section_gen_regression}.

In the rest of this section, we introduce the basic notations and concepts that are used throughout the paper.
Then we present some properties of generalized $\psi$-estimators that do not appear in Barczy and P\'ales \cite{BarPal2} such as a mean-type property
 (see Proposition \ref{Pro_psi_becsles_mean_prop}).
These properties are not just interesting on their own rights, but we also use them in the proofs for solving the comparison and equality problems of generalized $\psi$-estimators.
The comparison problem for generalized $\psi$-estimators is solved in Section \ref{Section_comp_equal_psi_est}
(see Theorems \ref{Lem_psi_est_eq_5} and \ref{Lem_phi_est_eq_5}), namely, we derive several equivalent conditions
 in order that a $\psi$-estimator be less than or equal to another $\psi$-estimator based on any possible realization of samples of any size.
Under an additional differentiability assumption, a further well-useable equivalent condition is derived,
see Theorem \ref{Thm_inequality_diff}.
We point out the fact that these results make one possible to compare two $\psi$-estimators even if these estimators cannot be explicitly computed, only they can be numerically approximated.
Theorem \ref{Thm_equality} is devoted to solve the equality problem for generalized $\psi$-estimators.
In Section \ref{Sec_comp_equal_Bajrak}, we apply our results in Section \ref{Section_comp_equal_psi_est}
in order to solve the comparison and equality problems for Bajraktarevi\'c-type estimators
that were introduced in Barczy and P\'ales \cite{BarPal2} (see also \eqref{help14}).
Propositions \ref{Pro_Baj_type_comparison} and \ref{Pro_Baj_type_comparison_2} are
about the comparison problem, while Theorems \ref{Thm_Baj_type_equality} and \ref{Lem_aux}
are about the equality problem for Bajraktarevi\'c-type estimators.
We note that, surprisingly, in the heart of the proof of the equality problem for Bajraktarevi\'c-type estimators a result about Schwarzian derivative and rational functions come into play, see Lemma \ref{Lem_Sch_deriv}.
We can also characterize the equality of quasiarithmetic-type $\psi$-estimators, see Corollary \ref{Pro_qa_type_equality}.
In Proposition \ref{Pro_Mobius},  we derive a necessary and sufficient condition in order that two
 strictly increasing functions defined on a nondegenerate open interval be the M\"obius transforms of each other.
In Section \ref{Sec_stat_examples}, we apply our results in Section \ref{Section_comp_equal_psi_est} for some known statistical estimators such as for empirical expectiles and Mathieu-type estimators and for solutions of likelihood equations in case of normal, a Beta-type, Gamma, Lomax (Pareto type II), lognormal and Laplace distributions. We mention that our results can also be used for approximating the solutions of likelihood equations in terms of some simpler $\psi$-estimators.
This approach can be useful when one cannot explicitly  calculate the solution of the likelihood equation in question,
but the approximating $\psi$-estimators are explicitly and easily computable.
In Proposition \ref{Pro_Beta_approx}, we present such an approximation of the solution of the likelihood equation for a Beta-type distribution.

The statistical applications of Bajraktarevi\'c-type $\psi$-estimators has not been explored yet, it can be a topic of a future research.
Here we point out an important field in practice, where these types of estimators may be indeed useful.
Mukhopadhyay et al.\ \cite{MukDasBasChaBha} found that in the presence of outliers in the data,
more precisely, when the data are generated by a mixture population involving a major (dominating) component and a minor (outlying) component,
the power mean (also called generalized mean) estimates the mean of the dominating population more accurately compared to the usual maximum likelihood estimator.
Thus the class of power means offers an alternative way for estimating the target mean parameter
without invoking the complications of sophisticated robust techniques.
Power means are special Bajraktarevi\'c means, that can be considered as special Bajraktarevi\'c-type $\psi$-estimators.
This can indicate some potential of Bajraktarevi\'c-type $\psi$-estimators as well in estimation of parameters for
 date coming from a mixture population.

Throughout this paper, we fix the following notations: the symbols $\NN$, $\ZZ_+$, $\QQ$, $\RR$, $\RR_+$, $\RR_{++}$, and $\RR_{--}$ will stand for the sets of positive integers,
non-negative integers, rational numbers, real numbers, non-negative real numbers, positive real numbers, and negative real numbers, respectively. For a subset $S\subseteq \RR$, the convex hull of $S$ (which is the smallest interval containing $S$) is denoted by $\conv(S)$. A real interval will be called nondegenerate if it contains at least two distinct points.
For each $n\in\NN$, let us also introduce the set $\Lambda_n:=\RR_+^n\setminus\{(0,\ldots,0)\}$.
The rank of a matrix $A\in\RR^{n\times n}$ is denoted by $\rank(A)$.
All the random variables are defined on an appropriate probability space $(\Omega,\cA,\PP)$.

\begin{Def}\label{Def_sign_change}
Let $\Theta$ be a nondegenerate open interval of $\RR$. For a function $f:\Theta\to\RR$, consider the following three level sets
\[
  \Theta_{f>0}:=\{t\in \Theta: f(t)>0\},\qquad
  \Theta_{f=0}:=\{t\in \Theta: f(t)=0\},\qquad
  \Theta_{f<0}:=\{t\in \Theta: f(t)<0\}.
\]
We say that $\vartheta\in\Theta$ is a \emph{point of sign change (of decreasing-type) for $f$} if
 \[
 f(t) > 0 \quad \text{for $t<\vartheta$,}
   \qquad \text{and} \qquad
    f(t)< 0 \quad  \text{for $t>\vartheta$.}
 \]
\end{Def}

Note that there can exist at most one element $\vartheta\in\Theta$ which is a point of sign change for $f$.
Further, if $f$ is continuous at a point $\vartheta$ of sign change, then $\vartheta$ is the unique zero of $f$.

Let $X$ be a nonempty set, $\Theta$ be a nondegenerate open interval of $\RR$ and let $\Psi(X,\Theta)$ denote the class of real-valued functions $\psi:X\times\Theta\to\RR$ such that, for all $x\in X$, there exist $t_+,t_-\in\Theta$ such that $t_+<t_-$ and $\psi(x,t_+)>0>\psi(x,t_-)$.
Roughly speaking, a function $\psi\in\Psi(X,\Theta)$ satisfies the following property:
 for all $x\in X$, the function $t\ni\Theta\mapsto \psi(x,t)$ changes sign (from positive to negative)
 on the interval $\Theta$ at least once.

\begin{Def}\label{Def_Tn}
We say that a function $\psi\in\Psi(X,\Theta)$
\vspace{-3mm}
  \begin{enumerate}[(i)]
    \item \emph{possesses the property $[C]$ (briefly, $\psi$ is a $C$-function)} if
           it is continuous in its second variable, i.e., if, for all $x\in X$,
           the mapping $\Theta\ni t\mapsto \psi(x,t)$ is continuous.
    \item \emph{possesses the property $[T_n]$ (briefly, $\psi$ is a $T_n$-function)
           for some $n\in\NN$} if there exists a mapping $\vartheta_{n,\psi}:X^n\to\Theta$ such that,
           for all $\pmb{x}=(x_1,\dots,x_n)\in X^n$ and $t\in\Theta$,
           \begin{align}\label{psi_est_inequality}
             \psi_{\pmb{x}}(t)=\sum_{i=1}^n \psi(x_i,t) \begin{cases}
                 > 0 & \text{if $t<\vartheta_{n,\psi}(\pmb{x})$,}\\
                 < 0 & \text{if $t>\vartheta_{n,\psi}(\pmb{x})$},
            \end{cases}
           \end{align}
          that is, for all $\pmb{x}\in X^n$, the value $\vartheta_{n,\psi}(\pmb{x})$ is a point of sign change for the function $\psi_{\pmb{x}}$. If there is no confusion, instead of
          $\vartheta_{n,\psi}$ we simply write $\vartheta_n$.
          We may call $\vartheta_{n,\psi}(\pmb{x})$ as a generalized $\psi$-estimator for
         some unknown parameter in $\Theta$ based on the realization $\bx=(x_1,\ldots,x_n)\in X^n$. If, for each $n\in\NN$, $\psi$ is a $T_n$-function, then we say that \emph{$\psi$ possesses the
         property $[T]$ (briefly, $\psi$ is a $T$-function)}.
    \item \emph{possesses the property $[Z_n]$ (briefly, $\psi$ is a $Z_n$-function) for some $n\in\NN$} if it is a $T_n$-function and
    \[
   \psi_{\pmb{x}}(\vartheta_{n,\psi}(\pmb{x}))=\sum_{i=1}^n \psi(x_i,\vartheta_{n,\psi}(\pmb{x}))= 0
    \qquad \text{for all}\quad \pmb{x}=(x_1,\ldots,x_n)\in X^n.
    \]
    If, for each $n\in\NN$, $\psi$ is a $Z_n$-function, then we say that \emph{$\psi$ possesses the property $[Z]$ (briefly, $\psi$ is a $Z$-function)}.
    \item \emph{possesses the property $[T_n^{\pmb{\lambda}}]$ for some $n\in\NN$ and $\pmb{\lambda}=(\lambda_1,\ldots,\lambda_n)\in\Lambda_n$ (briefly, $\psi$ is a
        $T_n^{\pmb{\lambda}}$-function)} if there exists a mapping $\vartheta_{n,\psi}^{\pmb{\lambda}}:X^n\to\Theta$ such that, for all $\pmb{x}=(x_1,\dots,x_n)\in X^n$ and $t\in\Theta$,
          \begin{align}\label{psi_est_inequality_weighted}
           \psi_{\pmb{x},\pmb{\lambda}}(t)= \sum_{i=1}^n \lambda_i\psi(x_i,t) \begin{cases}
                 > 0 & \text{if $t<\vartheta_{n,\psi}^{\pmb{\lambda}}(\pmb{x})$,}\\
                 < 0 & \text{if $t>\vartheta_{n,\psi}^{\pmb{\lambda}}(\pmb{x})$},
             \end{cases}
           \end{align}
           that is, for all $\pmb{x}\in X^n$, the value $\vartheta_{n,\psi}^{\pmb{\lambda}}(\pmb{x})$ is
           a point of sign change for the function $\psi_{\pmb{x},\pmb{\lambda}}$.
           If there is no confusion, instead of $\vartheta_{n,\psi}^{\pmb{\lambda}}$ we simply write $\vartheta_n^{\pmb{\lambda}}$.
          We may call $\vartheta_{n,\psi}^{\pmb{\lambda}}(\pmb{x})$
          as a weighted generalized $\psi$-estimator for some unknown parameter in $\Theta$ based
          on the realization $\bx=(x_1,\ldots,x_n)\in X^n$ and weights $(\lambda_1,\ldots,\lambda_n)\in\Lambda_n$.
   \end{enumerate}
\end{Def}

It can be seen that if $\psi$ is continuous in its second variable and, for some $n\in\NN$, it is a $T_n$-function, then it also a $Z_n$-function.

Given $q\in\NN$ and properties
\[
 [P_1], \ldots, [P_q]\in\big\{[C],[T],[Z]\big\}
 \cup\big\{[T_n], [Z_n]\colon n\in\NN\big\}
 \cup\big\{[T_n^{\pmb{\lambda}}]\colon n\in\NN,\, \pmb{\lambda}\in\Lambda_n\big\},
\]
the subclass of $\Psi(X,\Theta)$ consisting of elements possessing the properties $[P_1],\ldots,[P_q]$ will be denoted by $\Psi[P_1,\ldots,P_q](X,\Theta)$, i.e.,
\[
  \Psi[P_1,\ldots,P_q](X,\Theta)
  :=\bigcap_{i=1}^q\Psi[P_i](X,\Theta).
\]

The following statement is a direct consequence of the property $[Z]$.

\begin{Lem}\label{Lem_psi_est_eq_2}
Let $\psi\in\Psi[Z](X,\Theta)$. Then, for each $n\in\NN$, $\pmb{x}=(x_1,\ldots,x_n)\in X^n$ and $t\in\Theta$, the inequality
 \[
  \vartheta_{n,\psi}(\pmb{x})\leq (<)\, t
 \]
 is valid if and only if
 \[
  \sum_{i=1}^n \psi(x_i,t)\leq (<)\, 0
 \]
 is true.
\end{Lem}

The next result establishes a mean-type property of generalized $\psi$-estimators.

\begin{Pro}\label{Pro_psi_becsles_mean_prop}
Let $n\in\NN$ be fixed and $\psi\in\Psi[T_n](X,\Theta)$.
Then, for all $x_1,\ldots,x_n\in X$, we have
 \begin{align}\label{help_mean_prop}
   \min(\vartheta_{1,\psi}(x_1),\dots,\vartheta_{1,\psi}(x_n))
     \leq \vartheta_{n,\psi}(x_1,\ldots,x_n)\leq \max(\vartheta_{1,\psi}(x_1),\dots,\vartheta_{1,\psi}(x_n)).
 \end{align}
Furthermore, if $\psi\in\Psi[Z_n](X,\Theta)$ and $x_1,\ldots,x_n\in X$ are such that not all the values $\vartheta_{1,\psi}(x_1),\ldots,\vartheta_{1,\psi}(x_n)$ are equal, then both inequalities in
\eqref{help_mean_prop} are strict.
\end{Pro}

\begin{proof}
It is easy to see that the properties $[T_n]$ and $[Z_n]$ imply $[T_1]$ and $[Z_1]$, respectively. If $n=1$, then there is nothing to show. We may assume that $n\geq2$. Let $\psi\in\Psi[T_n](X,\Theta)$. To see that $\psi\in\Psi[T_1](X,\Theta)$, let $x_1\in X$ and define $x_2=\cdots =x_n:=x_1$. Then, according to the property $[T_n]$ of $\psi$, the function
\Eq{*}{
  \Theta\ni t\mapsto\sum_{i=1}^n\psi(x_i,t)=n\psi(x_1,t)
}
has a unique point of sign change, hence the map $\Theta\ni t\mapsto\psi(x_1,t)$ has also a unique point of sign change for all $x_1\in X$.
Therefore, $\psi\in\Psi[T_1](X,\Theta)$.
 A similar argument shows that the property $[Z_n]$ implies $[Z_1]$.
In fact, in Barczy and P\'ales \cite[Proposition 2.15]{BarPal2}, we have shown a more general statement, namely, if $m\in\NN$ is a divisor of $n$, then the properties $[T_n]$ and $[Z_n]$ imply $[T_m]$ and $[Z_m]$, respectively.

To verify the first part of the assertion, assume that $\psi\in\Psi[T_n](X,\Theta)$, let $x_1,\ldots,x_n\in X$ and introduce the notations $t_0:=\vartheta_{n,\psi}(x_1,\ldots,x_n)$,
 \begin{align*}
   t':=\min(\vartheta_{1,\psi}(x_1),\dots,\vartheta_{1,\psi}(x_n)) \qquad \text{and}\qquad t'':=\max(\vartheta_{1,\psi}(x_1),\dots,\vartheta_{1,\psi}(x_n)).
 \end{align*}
Since \ $\psi$ \ is a $T_1$-function, we have $\psi(x_i,t)>0$, $i\in\{1,\ldots,n\}$, for all $t\in\Theta$ with $t<t'$,
and $\psi(x_i,t)<0$, $i\in\{1,\ldots,n\}$, for all $t\in\Theta$ with $t>t''$.
This yields that
 \[
    \sum_{i=1}^n \psi(x_i,t) \begin{cases}
                               >0 & \text{if $t\in\Theta$ and $t<t'$,}\\
                               <0 & \text{if $t\in\Theta$ and $t>t''$.}
                              \end{cases}
 \]
Since $\psi$ is a $T_n$-function, it implies that $t'\leq t_0\leq t''$, as desired.

Now let us suppose that $\psi\in\Psi[Z_n](X,\Theta)$ and that $x_1,\ldots,x_n\in X$ are chosen such that not all
 the values $\vartheta_{1,\psi}(x_1),\ldots,\vartheta_{1,\psi}(x_n)$ are equal.
Then $t'<t''$, and hence there exist $i\ne j$, $i,j\in\{1,\ldots,n\}$
 such that $t'=\vartheta_{1,\psi}(x_i)$ and $t''=\vartheta_{1,\psi}(x_j)$.
On the contrary, let us suppose that $t_0=t'$ or $t_0=t''$.
If $t_0=t'$ were true, then, using that $\psi$ is a $Z_n$-function,
 we have $\psi(x_k,t_0)\geq 0$, $k\in\{1,\ldots,n\}$, and $\sum_{k=1}^n \psi(x_k,t_0) = 0$, from which we can conclude that $\psi(x_k,t_0)= 0$, $k\in\{1,\ldots,n\}$.
On the other hand, $t_0<t''=\vartheta_{1,\psi}(x_j)$ implies $\psi(x_j,t_0) > 0$, yielding us to a contradiction.
Similarly, if $t_0=t''$ were true, then, using that $\psi$ is a $Z_n$-function,
 we have $\psi(x_k,t_0)\leq 0$, $k\in\{1,\ldots,n\}$, and $\sum_{k=1}^n \psi(x_k,t_0) = 0$, from which we get that $\psi(x_k,t_0)=0$, $k\in\{1,\ldots,n\}$.
However, $\vartheta_{1,\psi}(x_i)=t'<t_0$ yields $\psi(x_i,t_0) < 0$, which leads to a contradiction as well.
\end{proof}

\begin{Pro}\label{Pro_density}
Let $\psi\in\Psi[T](X,\Theta)$ and assume that, for all $x,y\in X$ with $\vartheta_{1,\psi}(x)<\vartheta_{1,\psi}(y)$, the map
\Eq{umap}{
   (\vartheta_{1,\psi}(x), \vartheta_{1,\psi}(y))\ni u \mapsto -\frac{\psi(x,u)}{\psi(y,u)}
}
is positive and strictly increasing. Then the set
\Eq{*}{
  S_\psi:=\{\vartheta_{n,\psi}(x_1,\ldots,x_n)\mid n\in\NN,\, x_1,\ldots,x_n\in X\}
}
is a dense subset of $\conv(\vartheta_{1,\psi}(X))$.
\end{Pro}

\begin{proof} According to Proposition \ref{Pro_psi_becsles_mean_prop}, it follows that $S_\psi\subseteq \conv(\vartheta_{1,\psi}(X))=:C$.
Therefore, if $C$ is a singleton, then there is nothing to prove.
Thus, we may assume that the convex set $C$ is a nondegenerate subinterval of $\Theta$.
To prove the density of $S_\psi$ in $C$, let $s,t\in C$ with $\inf C<s<t<\sup C$.
Since $\inf C = \inf \vartheta_{1,\psi}(X)$ and $\sup C = \sup \vartheta_{1,\psi}(X)$,
 there exist $x,y\in X$ such that $\vartheta_{1,\psi}(x)<s<t<\vartheta_{1,\psi}(y)$. Consequently,
\[
   0< -\frac{\psi(x,s)}{\psi(y,s)} < -\frac{\psi(x,t)}{\psi(y,t)},
\]
and hence one can choose $n,m\in\NN$ such that
\[
  0 < -\frac{\psi(x,s)}{\psi(y,s)} < \frac{m}{n} < -\frac{\psi(x,t)}{\psi(y,t)}.
\]
Using that $\psi(y,s)>0$ and $\psi(y,t)>0$, we have that
 \[
  n\psi(x,s) + m\psi(y,s) >0 \qquad \text{and}\qquad   n\psi(x,t) + m\psi(y,t) < 0.
 \]
Since $\psi$ has the property $[T_{n+m}]$, we can conclude that
 \[
     s \leq \vartheta_{n+m,\psi}(\underbrace{x,\ldots,x}_{n}, \underbrace{y,\ldots,y}_{m}) \leq t.
 \]
As a consequence, we have that $[s,t]\cap S_\psi$ is nonempty.
Since $[s,t]$ was an arbitrary nondegenerate subinterval in the interior of $C$, it follows that $S_\psi$ is dense in $C$.
\end{proof}

\begin{Rem}\rm In our recent paper Barczy and P\'ales \cite{BarPal2}, several implications between the property $[T]$ of a function $\psi\in\Psi(X,\Theta)$ and the monotonicity properties of the map \eqref{umap} have been established.
Among others, we proved that if $\psi$ possesses the property $[T]$, then for all $x,y\in X$ with $\vartheta_{1,\psi}(x)<\vartheta_{1,\psi}(y)$, the map \eqref{umap} is positive and (not necessarily strictly) increasing.
On the other hand, if $\psi$ has the properties $[T]$ and $[Z_1]$, and, for all $x,y\in X$ with $\vartheta_{1,\psi}(x)<\vartheta_{1,\psi}(y)$, the map \eqref{umap} is strictly increasing, then $\psi$ possesses the property $[T_n^{\pmb{\lambda}}]$ for all $n\in\NN$ and $\pmb{\lambda}=(\lambda_1,\ldots,\lambda_n)\in\Lambda_n$.
\proofend
\end{Rem}

\begin{Lem}\label{Lem_psi_est_cont}
Let $\psi\in\Psi[T](X,\Theta)$. Then, for all $k\in\NN$ and $x,y_1,\dots,y_k\in X$, we have
 \[
    \lim_{n\to\infty} \vartheta_{n+k,\psi}({\underbrace{x,\ldots,x}_{n}},y_1,\dots,y_k) = \vartheta_{1,\psi}(x).
 \]
The same statement holds if $\psi$ has the properties $[Z_1]$ and $[T_2^\blambda]$ for all $\blambda\in\Lambda_2$.
\end{Lem}

\begin{proof}
Let $k\in\NN$ and let $x,y_1,\dots,y_k\in X$ be fixed arbitrarily.
For each $n\in\NN$, denote
 \[
   t_n:=\vartheta_{n+k,\psi}({\underbrace{x,\ldots,x}_{n}},y_1,\dots,y_k).
 \]
We need to show that $t_n$ converges to $\vartheta_{1,\psi}(x)$ as $n\to\infty$.
Let $t',t''\in\Theta$ be arbitrary such that $t'<\vartheta_{1,\psi}(x)<t''$ (since $\Theta$ is open such $t'$ and $t''$ exist).
Then
 \[
  n\psi(x,t')+\sum_{i=1}^k\psi(y_i,t')
  =n\bigg(\psi(x,t')+\frac1n\sum_{i=1}^k\psi(y_i,t')\bigg)>0
 \]
 if $n$ is large enough, because $t'<\vartheta_{1,\psi}(x)$ and
 \[
   \lim_{n\to\infty}\bigg(\psi(x,t')+\frac1n\sum_{i=1}^k\psi(y_i,t')\bigg)=\psi(x,t')>0.
 \]
Similarly, we get that the inequality
 \[
  n\psi(x,t'')+\sum_{i=1}^k\psi(y_i,t'')<0
 \]
 is valid if $n$ is large enough.
Therefore, the point of sign change $t_n$ of the function
 \[
  \Theta\ni t\mapsto n\psi(x,t)+\sum_{i=1}^k\psi(y_i,t)
 \]
is in the interval $[t',t'']$ for $n$ large enough, that is, there exists $n_0\in\NN$ such that $t_n\in[t',t'']$ for each $n\geq n_0$, $n\in\NN$.
Since $t',t''\in\Theta$ were arbitrary with $t'< \vartheta_{1,\psi}(x)<t''$, this implies that $t_n$ converges to $ \vartheta_{1,\psi}(x)$ as $n\to\infty$.

If $\psi$ has the properties $[Z_1]$ and $[T_2^\blambda]$ for all $\blambda\in\Lambda_2$, then, by Corollary 2.11 in Barczy and P\'ales \cite{BarPal2}, we have that $\psi$ is also a $T$-function,
consequently, the same conclusion holds.
\end{proof}

\section{Historical notes about equality of OLSE and BLUE in general linear models}\label{Section_gen_regression}

Consider a general linear model
 \begin{align}\label{gen_lin_mod}
  \EE(\Beta)= \bX \bbeta \qquad \text{with}\qquad \cov(\Beta)=\bV,
 \end{align}
 where $\Beta$ is an $\RR^n$-valued random vector,
 $\bX\in\RR^{n\times q}$ with $q<n$, $q,n\in\NN$, $\bbeta\in\RR^q$ is a vector of unknown (regression) parameters,
 $\bV\in\RR^{n\times n}$ is a known, symmetric and positive semidefinite matrix.
Note that $\bV$ is positive definite if and only if $\bV$ has full rank $n$.
Indeed, $\bV$ is positive (semi)definite if and only if all its eigenvalues are (non-negative) positive.
Hence, using that the determinant of $\bV$ equals the product of its eigenvalues, we have that $\bV$ is positive definite
 if and only if $\det(\bV)\ne 0$, or equivalently, $\rank(\bV)=n$, as desired.

The ordinary least squares estimator (OLSE) of $\bbeta$ is defined to be
 \[
  \widehat\bbeta:=\argmin_{\bbeta\in\RR^q}(\Beta - \bX\bbeta)^\top(\Beta - \bX\bbeta),
 \]
 and the OLSE of $\bX\bbeta$ is defined to be $\bX\widehat\bbeta$.
The best linear unbiased estimator (BLUE) $\bbeta^*$ of $\bbeta$ is a linear estimator $\bG\Beta$ such that $\EE(\bG\Beta)=\bbeta$
 (i.e., $\bG\Beta$ is an unbiased estimator of $\bbeta$)
 and for any other linear unbiased estimator $\bM\Beta$ of $\bbeta$, we have that $\cov(\bM\Beta) - \cov(\bG\Beta)\in\RR^{q\times q}$
 is nonnegative definite, where $\bG,\bM\in\RR^{q\times n}$.
Similarly, the BLUE estimator $\bbeta^*$ of $\bX\bbeta$ is a linear estimator $\bG\Beta$ such that $\EE(\bG\Beta)=\bX\bbeta$
 (i.e., $\bG\Beta$ is an unbiased estimator of $\bX\bbeta$)
 and for any other linear unbiased estimator $\bM\Beta$ of $\bX\bbeta$, we have that $\cov(\bM\Beta) - \cov(\bG\Beta)\in\RR^{n\times n}$ is nonnegative definite,
 where $\bG,\bM\in\RR^{n\times n}$.
In the special case when $\bX$ and $\bV$ are both of full rank (i.e., when $\rank(X)=q$ and $\rank(V)=n$),
 the OLSE and BLUE of $\bbeta$ take the following forms
 \[
   \widehat\bbeta=(\bX^\top \bX)^{-1}\bX^\top \Beta,
 \]
 and
 \[
  \bbeta^*=(\bX^\top \bV^{-1}\bX)^{-1}\bX^\top \bV^{-1}\Beta.
 \]
In the general case, there are also explicit formulae for $\widehat\bbeta$ and $\bbeta^*$ involving Moore-Penrose inverses, see, e.g.,
 Lemma 1 in Tian and Wiens \cite{TiaWie}.

Puntanen and Styan \cite{PunSty} gave a good historical overview on the (almost sure) equality of the OLSE
 and the BLUE of $\bbeta$ and those of $\bX\bbeta$, respectively.
It is well-known that OLSE and BLUE of $\bbeta$ can coincide even if $\bV$ is not a multiple of the $n\times n$ identity matrix $\bI_n$.
In particular, if $q=1$ and $\bX=(1,\ldots,1)^\top \in\RR^{n\times 1}$, then the OLSE of $\bbeta$ is the arithmetic mean
 $\widehat\bbeta = (\eta_1+\cdots+\eta_n)/n$ with the notation $\Beta=:(\eta_1,\ldots,\eta_n)$,
 and one can prove that if the covariance matrix $\bV$ has all of its row totals equal to each other,
 then the OLSE of $\bbeta$ coincides with the BLUE of $\bbeta$.
For several necessary and sufficient conditions on the equality of the OLSE and BLUE of $\bX\bbeta$ have been recalled
 in Section 2 in Puntanen and Styan \cite{PunSty}.
In the special case when $\bX$ and $\bV$ are both of full rank (i.e., when $\rank(X)=q$ and $\rank(V)=n$),
 one can introduce the so-called relative goodness (also called D-relative efficiency) of the OLSE of $\bbeta$ defined by
 \[
   \kappa:=\frac{\det(\cov(\bbeta^*))}{\det(\cov(\widehat\bbeta))}
          = \frac{(\det(\bX^\top\bX))^2}{\det(\bX^\top\bV\bX) \det(\bX^\top\bV^{-1}\bX)} ,
 \]
 see, e.g., Puntanen and Styan \cite[Section 3.1]{PunSty} or Tian and Wiens \cite[page 1267]{TiaWie}.
One can prove that $\kappa\in(0,1]$, and $\kappa=1$ holds if and only if the OLSE and BLUE of $\bbeta$ coincide.

Lee \cite{Lee} considered the general linear model \eqref{gen_lin_mod} such that $\bX$ and $\bV$ are both of full rank.
He derived necessary and sufficient conditions for the equality of the BLUEs $\bbeta^*_{\bV_1}$ and $\bbeta^*_{\bV_2}$
 in two general linear models with positive definite covariance matrices $\bV_1$ and $\bV_2$ and with a common matrix $\bX$ having full rank $q$.
Among others, it was proved that $\bbeta^*_{\bV_1} = \bbeta^*_{\bV_2}$ holds if and only if there exists a non-singular matrix $\bA\in\RR^{q\times q}$
 such that $\bV_1^{-1}\bX = \bV_2^{-1}\bX\bA$.
For further necessary and sufficient conditions, see also Lu and Schmidt \cite[Theorem 2]{LuSch}.
If one chooses  $\bV_2:=\bI_n$, then the problem considered above by Lee \cite{Lee}
 is equivalent to finding necessary and sufficient conditions for the equality of the BLUE $\bbeta^*_{\bV_1}$
 and the OLSE $\bbeta^*_{\bI_n}=\widehat\bbeta$.
Lee \cite[Section 4]{Lee} derived several other sufficient, but not necessary conditions for the equality of two BLUEs as well.
For example, if $\bX$ has full rank $q$, $\bV_1$ is positive semidefinite, $\bGamma\in\RR^{q\times q}$ is symmetric and
 $\bV_2:=\bV_1+\bX\bGamma \bX^\top$, then  $\bbeta^*_{\bV_1} = \bbeta^*_{\bV_2}$.

In the special case when $\bX$ and $\bV$ are both of full rank, Kr\"amer et al.\ \cite{KraBarFie}
 derived a necessary and sufficient condition for the equality of the OLSE and BLUE of a subset of the regression parameters
 $\beta_1,\ldots,\beta_q$ in a general linear model \eqref{gen_lin_mod}, where $\bbeta^\top=:(\beta_1,\ldots,\beta_q)$.
More precisely, rewriting a general linear model \eqref{gen_lin_mod} in the form $\EE(\Beta)=\bX_1\bbeta_1+\bX_2\bbeta_2$
 and denoting by $\widehat\bbeta_2$ and $\bbeta^*_2$ the respective subvectors of $\widehat\bbeta$ and $\bbeta^*$,
 a necessary and sufficient condition has been derived for the equality of  $\widehat\bbeta_2$ and $\bbeta^*_2$,
 where $\bbeta^\top = (\bbeta_1^\top,\bbeta_2^\top)\in\RR^{q_1}\times\RR^{q_2}$ and $\bX=(\bX_1,\bX_2)\in\RR^{n\times q_1}\times\RR^{n\times q_2}$.

Tian and Wiens \cite[Theorem 5]{TiaWie} derived necessary and sufficient conditions for the equality and proportionality of the OLSE and BLUE of $\bX\bbeta$, respectively.
In particular, it turned out that if the OLSE and BLUE of $\bX\bbeta$ are proportional to each other
 (i.e., there exists a $\lambda\in\RR$ such that $\mathrm{OLSE}(\bX\bbeta)=\lambda\cdot \mathrm{BLUE}(\bX\bbeta)$ with probability one),
 then the two estimators in question are equal with probability 1.

Tian and Puntanen \cite{TiaPun} derived some necessary and sufficient conditions for
 the equality of OLSE and BLUE of the unknown (regression) parameters under a
 general linear model and its linearly transformed linear model.

Lu and Schmidt \cite[Theorem 3]{LuSch} derived necessary and sufficient conditions for the equality of BLUE
 and so-called Amemiya-Cragg estimator of $\bbeta$.

Recently, Gong \cite{Gon} has reconsidered the equality problem of OLSE and BLUE for so-called seemingly unrelated regression models.
Jiang and Sun \cite{JiaSun} have investigated necessary and sufficient conditions for the equality of OLSE and BLUE of the whole (or partial)
 set of regression parameters in a general linear model with linear parameter restrictions.

\section{Comparison and equality of generalized $\psi$ estimators}
\label{Section_comp_equal_psi_est}

Given $\psi,\varphi\in \Psi(X,\Theta)$, we are going to establish necessary and sufficient conditions in order that
 the comparison inequality $\vartheta_{n,\psi}\leq \vartheta_{n,\varphi}$ be valid on $X^n$ for all $n\in\NN$.

As a first result, provided that $\psi,\varphi\in\Psi(X,\Theta)$, $\varphi$ is a $Z$-function and $\psi$ possesses the properties $[T]$ and $[Z_1]$, we give necessary and sufficient conditions in order that $\vartheta_{n,\psi}\leq \vartheta_{n,\varphi}$ be valid on $X^n$ for all $n\in\NN$.

For a function $\psi\in\Psi[T_1](X,\Theta)$, we introduce the notation
 \begin{align}\label{theta_psi}
 \Theta_\psi:=\big\{t\in\Theta\mid \exists\, x,y\in X: \vartheta_{1,\psi}(x)<t<\vartheta_{1,\psi}(y)\big\}.
 \end{align}
Observe that $\Theta_\psi$ is open. Indeed,
if it is the empty set, then it is open trivially. Otherwise, $\Theta_\psi$ is the union of all open intervals $(\vartheta_{1,\psi}(x),\vartheta_{1,\psi}(y))$,
 where $x,y\in X$ are such that $\vartheta_{1,\psi}(x)<\vartheta_{1,\psi}(y)$.
In fact, $\Theta_\psi$ is nothing else but the interior of the convex hull of $\vartheta_{1,\psi}(X)$.
Indeed, any interior point $s$ of the convex hull of a subset $S$ of $\RR^d$
 is an interior point of the convex hull of some subset (possibly depending on $s$) of $S$ containing at most $2d$ points, see Gustin \cite{Gus}.
Consequently, $\Theta_\psi$ is an open (possibly degenerate) interval.

\begin{Thm}\label{Lem_psi_est_eq_5}
Let $\psi\in\Psi[T,Z_1](X,\Theta)$ and $\varphi\in\Psi[Z](X,\Theta)$. Then the following assertions are equivalent to each other:
\vspace{-3mm}
\begin{enumerate}[(i)]
 \item The inequality
 \begin{align}\label{psi_est_inequality_2}
   \vartheta_{n,\psi}(x_1,\ldots,x_n)\leq \vartheta_{n,\varphi}(x_1,\ldots,x_n)
\end{align}
holds for each $n\in\NN$ and $x_1,\dots,x_n\in X$.
 \item The inequality
 \begin{align}\label{psi_est_inequality_2.5}
  \vartheta_{k+m,\psi}(\underbrace{x,\ldots,x}_{k}, \underbrace{y,\ldots,y}_{m})
  \leq \vartheta_{k+m,\varphi}(\underbrace{x,\ldots,x}_{k}, \underbrace{y,\ldots,y}_{m})
\end{align}
holds for each $k,m\in\NN$ and $x,y\in X$.
 \item For all $x\in X$, we have $\vartheta_{1,\psi}(x)\leq \vartheta_{1,\varphi}(x)$, and the inequality
 \begin{align}\label{psi_est_inequality_3}
    \psi(x,t) \varphi(y,t) \leq \psi(y,t) \varphi(x,t)
 \end{align}
 is valid for all $t\in\Theta$ and for all $x,y\in X$ with $\vartheta_{1,\varphi}(x)<t<\vartheta_{1,\varphi}(y)$.
 \item For all $x\in X$, we have $\vartheta_{1,\psi}(x)\leq \vartheta_{1,\varphi}(x)$, and there exists a nonnegative function $p:\Theta_\varphi\to\RR_+$ such that
 \begin{align}\label{psi_est_inequality_5}
    \psi(z,t)\leq p(t)\varphi(z,t), \qquad z\in X,\, t\in\Theta_\varphi.
 \end{align}
\end{enumerate}
\end{Thm}

\begin{proof}
The implication (i)$\Rightarrow$(ii) is obvious by taking $n=k+m$ in assertion (i).

(ii)$\Rightarrow$(iii).
Assume that (ii) holds.
Then the inequality  \eqref{psi_est_inequality_2.5} for $k:=m:=1$ and $y:=x$ together with $\vartheta_{2,\psi}(x,x) = \vartheta_{1,\psi}(x)$ and $\vartheta_{2,\varphi}(x,x) =
\vartheta_{1,\varphi}(x)$, imply that $\vartheta_{1,\psi}(x)\leq\vartheta_{1,\varphi}(x)$ for all $x\in X$.

We prove the inequality \eqref{psi_est_inequality_3} by contradiction.
Assume that, for some $t\in\Theta$ and $x,y\in X$ with $\vartheta_{1,\varphi}(x)<t<\vartheta_{1,\varphi}(y)$, we have
 \begin{align}\label{help_6}
    \psi(x,t) \varphi(y,t) > \psi(y,t) \varphi(x,t).
 \end{align}
In view of the inequalities $\vartheta_{1,\psi}(x)\leq\vartheta_{1,\varphi}(x)<t<\vartheta_{1,\varphi}(y)$, we get that $\psi(x,t)<0$, $\varphi(x,t)<0$ and $\varphi(y,t)>0$, whence the inequality
\eqref{help_6}
 yields that $\psi(y,t)>0$. Thus, we obtain
 \begin{align}\label{help_10}
  0<-\frac{\psi(x,t)}{\psi(y,t)}
  <-\frac{\varphi(x,t)}{ \varphi(y,t)}.
 \end{align}
Therefore, there exist $k,m\in\NN$ such that
$$
  0<-\frac{\psi(x,t)}{\psi(y,t)}
  <\frac{m}{k}<-\frac{\varphi(x,t)}{\varphi(y,t)}.
$$
Rearranging these inequalities, it follows that
$$
  k\varphi(x,t)+m\varphi(y,t)
  < 0 < k\psi(x,t)+m\psi(y,t).
$$
Since $\varphi$ possesses the property $[Z]$, by Lemma \ref{Lem_psi_est_eq_2}, we get the inequalities
$$
  \vartheta_{k+m,\varphi}(\underbrace{x,\ldots,x}_{k}, \underbrace{y,\ldots,y}_{m})
  < t \leq \vartheta_{k+m,\psi}(\underbrace{x,\ldots,x}_{k}, \underbrace{y,\ldots,y}_{m}).
$$
which contradicts \eqref{psi_est_inequality_2.5}.

(iii)$\Rightarrow$(iv). Assume that (iii) holds. If $\Theta_\varphi$ is empty, then there is nothing to prove. Thus, we may assume that $\Theta_\varphi\neq\emptyset$. Rearranging the inequality
\eqref{psi_est_inequality_3}, for all $t\in\Theta_\varphi$ and for all $x,y\in X$ with $\vartheta_{1,\varphi}(x)<t<\vartheta_{1,\varphi}(y)$, we get
 \begin{align}\label{help_7}
  \frac{\psi(y,t)}{\varphi(y,t)}
  \leq\frac{\psi(x,t)}{\varphi(x,t)},
\end{align}
where we used that $\varphi(x,t)<0$ and $\varphi(y,t)>0$.
Now we define the function $p:\Theta_\varphi\to\RR$ by
$$
  p(t):=\inf \bigg\{\frac{\psi(x,t)}{\varphi(x,t)}\,\bigg|\,x\in X,\,\vartheta_{1,\varphi}(x)<t\bigg\},
  \qquad t\in\Theta_\varphi.
$$
Due to the inclusion $t\in\Theta_\varphi$, the function $p$ is finite valued, and $p(t)\geq 0$ for all $t\in \Theta_\varphi$.
Indeed, if $t\in\Theta_\varphi$, then there exists $x\in X$ such that $\vartheta_{1,\varphi}(x)<t$ yielding that
 $\vartheta_{1,\psi}(x)\leq\vartheta_{1,\varphi}(x)<t$.
Therefore $\psi(x,t)<0$ and $\varphi(x,t)<0$,
 and hence $p(t)$ is defined as the infimum of certain positive real numbers.
Further, by the definition of infimum and \eqref{help_7}, for all $t\in\Theta_\varphi$ and for all $x,y\in X$ with $\vartheta_{1,\varphi}(x)<t<\vartheta_{1,\varphi}(y)$, we obtain that
\begin{equation}\label{help_6.5}
  \frac{\psi(y,t)}{\varphi(y,t)}
  \leq p(t)\leq\frac{\psi(x,t)}{\varphi(x,t)}.
\end{equation}
Let $z\in X$ be fixed arbitrarily.
The first inequality in \eqref{help_6.5} with $y:=z$ implies that \eqref{psi_est_inequality_5} holds for all $t\in\Theta_\varphi$ with $t<\vartheta_{1,\varphi}(z)$.
Similarly, the second inequality in \eqref{help_6.5} with $x:=z$ yields that \eqref{psi_est_inequality_5} is also valid for all $t\in\Theta_\varphi$ with $\vartheta_{1,\varphi}(z)<t$.
Finally, if $t=\vartheta_{1,\varphi}(z)$, then by the property $[Z]$ of $\varphi$, we have that $\varphi(z,t)=0$.
Furthermore, by the assumption in part (iii), the inequality $\vartheta_{1,\psi}(z)\leq \vartheta_{1,\varphi}(z)=t$ holds.
Using that $\psi$ is a $Z_1$-function, it implies that $\psi(z,t)\leq 0$.
This proves that \eqref{psi_est_inequality_5} holds for $t=\vartheta_{1,\varphi}(z)$ as well.

(iv)$\Rightarrow$(i). Assume that (iv) is valid.
To prove (i), let $n\in\NN$ and $x_1,\dots,x_n\in X$.

In the proof of the inequality \eqref{psi_est_inequality_2}, we distinguish two cases.
First, consider the case when $\vartheta_{1,\varphi}(x_1)=\dots=\vartheta_{1,\varphi}(x_n)=:t_0$.
Then $\varphi(x_1,t_0)=\dots= \varphi(x_n,t_0)=0$, since $\varphi$ possesses the property $[Z]$.
Therefore,
$$
  \sum_{i=1}^n\varphi(x_i,t_0)=0,
$$
which (using that $\varphi$ is also a $T$-function or Lemma \ref{Lem_psi_est_eq_2}) implies that
$$
  \vartheta_{n,\varphi}(x_1,\dots,x_n)=t_0.
$$
On the other hand, the inequality $\vartheta_{1,\psi}\leq  \vartheta_{1,\varphi}$ yields that
$$
  \vartheta_{1,\psi}(x_1)\leq \vartheta_{1,\varphi}(x_1)=t_0,\qquad\dots,\qquad \vartheta_{1,\psi}(x_n)\leq \vartheta_{1,\varphi}(x_n)=t_0,
$$
and hence, using that $\psi$ is a $Z_1$-function, we have
$$
  \psi(x_1,t_0)\leq0,\qquad\dots,\qquad
  \psi(x_n,t_0)\leq0.
$$
Therefore
$$
  \sum_{i=1}^n\psi(x_i,t_0)\leq0,
$$
which (using that $\psi$ is a $T$-function) implies that
$$
  \vartheta_{n,\psi}(x_1,\dots,x_n)\leq t_0
  =\vartheta_{n,\varphi}(x_1,\dots,x_n).
$$
Thus \eqref{psi_est_inequality_2} is proved, when $\vartheta_{1,\varphi}(x_1)=\dots=\vartheta_{1,\varphi}(x_n)$.

Now consider the case when
$$
  t':=\min(\vartheta_{1,\varphi}(x_1),\dots,\vartheta_{1,\varphi}(x_n))
  <\max(\vartheta_{1,\varphi}(x_1),\dots,\vartheta_{1,\varphi}(x_n))=:t''.
$$
Let $t_0:=\vartheta_{n,\varphi}(x_1,\dots,x_n)$.
Then there exist $i,j\in\{1,\dots,n\}$ such that $t'=\vartheta_{1,\varphi}(x_i)$ and $t''=\vartheta_{1,\varphi}(x_j)$.
According to Proposition \ref{Pro_psi_becsles_mean_prop}  (which can be applied, since $\varphi$ has the property $[Z]$),
we have that $t'<t_0<t''$, which shows that $t_0\in\Theta_\varphi$.
Hence we can apply the inequality \eqref{psi_est_inequality_5} with $z:=x_i$, $i\in\{1,\dots,n\}$ and $t:=t_0$. Then, adding up the inequalities so obtained side by side, we arrive at
 $$
  \sum_{i=1}^n \psi(x_i,t_0)
  \leq p(t_0)\sum_{i=1}^n \varphi(x_i,t_0)=0,
 $$
where the equality follows from the property $[Z]$ of $\varphi$ and the definition if $t_0$.
This, according to the property $[T_n]$ of $\psi$, implies that
$$
  \vartheta_{n,\psi}(x_1,\dots,x_n)\leq t_0=\vartheta_{n,\varphi}(x_1,\dots,x_n),
$$
and proves (i) in the considered second case as well.
\end{proof}

In the next remark, we highlight the role of the assumption
$\vartheta_{1,\psi}\leq\vartheta_{1,\varphi}$ in part (iii) of Theorem \ref{Lem_psi_est_eq_5}.

\begin{Rem}\label{Rem_Thm_main}
\rm
(i).
Let $\psi,\varphi\in\Psi[Z_1](X,\Theta)$, and suppose that $\vartheta_{1,\psi}(z)\leq \vartheta_{1,\varphi}(z)$ for all $z\in X$.
Let $x,y\in X$ with $\vartheta_{1,\varphi}(x)<\vartheta_{1,\varphi}(y)$.
Then the inequality \eqref{psi_est_inequality_3} is automatically valid for $t=\vartheta_{1,\varphi}(x)$ and $t=\vartheta_{1,\varphi}(y)$.
Indeed, if $t=\vartheta_{1,\varphi}(y)$, then we have $\vartheta_{1,\psi}(y)\leq t$ and
 \[
    \vartheta_{1,\psi}(x)\leq \vartheta_{1,\varphi}(x) <  \vartheta_{1,\varphi}(y) = t.
 \]
Using that $\psi$ and $\varphi$ have the property $[Z_1]$, it implies that $\psi(x,t)<0$, $\varphi(y,t)=0$, $\psi(y,t)\leq 0$ and $\varphi(x,t)<0$,
 and consequently,
 \[
  \psi(x,t)\varphi(y,t) = 0 \leq \psi(y,t)\varphi(x,t),
 \]
 as desired.
Similarly, if $t=\vartheta_{1,\varphi}(x)$, then we have
 \[
    \vartheta_{1,\psi}(x)\leq t = \vartheta_{1,\varphi}(x) < \vartheta_{1,\varphi}(y).
 \]
Using that $\psi$ and $\varphi$ have the property $[Z_1]$, it implies that $\psi(x,t)\leq 0$, $\varphi(y,t)>0$ and $\varphi(x,t)=0$,
 and consequently,
 \[
  \psi(x,t)\varphi(y,t) \leq 0 = \psi(y,t)\varphi(x,t),
 \]
 as desired.

(ii).
Let $\psi,\varphi\in\Psi[Z_1](X,\Theta)$.
If $\vartheta_{1,\varphi}(X)\subseteq \Theta_\varphi$,
 and there exists a nonnegative function $p:\Theta_\varphi\to\RR_+$ such that
 the inequality \eqref{psi_est_inequality_5} holds, then
the inequality $\vartheta_{1,\psi}(x)\leq \vartheta_{1,\varphi}(x)$ holds for all $x\in X$.
Indeed, using \eqref{psi_est_inequality_5} with $z:=x$ and $t:=\vartheta_{1,\varphi}(x)$ (where $x\in X$)
 and that $\varphi\in\Psi[Z_1](X,\Theta)$, we get that
 \[
    \psi(x,\vartheta_{1,\varphi}(x)) \leq p(\vartheta_{1,\varphi}(x))\varphi(x,\vartheta_{1,\varphi}(x)) = p(\vartheta_{1,\varphi}(x)) \cdot 0 = 0,
      \qquad x\in X.
 \]
Since $\psi\in\Psi[Z_1](X,\Theta)$, this inequality implies that $\vartheta_{1,\psi}(x)\leq \vartheta_{1,\varphi}(x)$ for all $x\in X$, as desired.
A similar argument shows that if $\psi,\varphi\in\Psi[Z_1](X,\Theta)$ and $\psi(z,t)\leq\varphi(z,t)$, $z\in X$, $t\in\Theta$, then $\vartheta_{1,\psi}(x)\leq \vartheta_{1,\varphi}(x)$, $x\in X$.

(iii). Let us suppose that $\psi\in\Psi[T,Z_1](X,\Theta)$ and $\varphi\in\Psi[Z](X,\Theta)$.
If the inequality \eqref{psi_est_inequality_3} holds for all $t\in\Theta$ and for all $x,y\in X$ with $\vartheta_{1,\varphi}(x)\leq t\leq \vartheta_{1,\varphi}(y)$, then, in general,
\eqref{psi_est_inequality_2} does not hold, not even for $n=1$.
We give a counterexample.
Let $X:=\{x_1,x_2\}$, $\Theta:=\RR$, and
 \begin{align*}
  &\psi(x_j,t):=-jt,\qquad t\in\RR,\qquad j\in\{1,2\},\\
  &\varphi(x_j,t):=-j(t+1),\qquad t\in\RR,\qquad j\in\{1,2\}.
 \end{align*}
Then $\psi$ and $\varphi$ are $Z$-functions with $\vartheta_{n,\psi}(\pmb{x})=0$
  and $\vartheta_{n,\varphi}(\pmb{x})=-1$ for all $n\in\NN$ and $\pmb{x}\in X^n$.
In particular, $\Theta_\varphi = \emptyset$.
Moreover, if $t\in\Theta$ and $x,y\in X$ are such that $\vartheta_{1,\varphi}(x)\leq t \leq \vartheta_{1,\varphi}(y)$, then $t=-1$, and, since $\varphi(x,-1) = \varphi(y,-1) = 0$, $x,y\in X$, we get
that both sides of the inequality \eqref{psi_est_inequality_3} are $0$ for $t=-1$.
Consequently, the inequality \eqref{psi_est_inequality_3} holds for all $t\in\Theta$ and for all $x,y\in X$ with $\vartheta_{1,\varphi}(x)\leq t\leq \vartheta_{1,\varphi}(y)$.
However, $\vartheta_{1,\psi}(x)=0>\vartheta_{1,\varphi}(x)=-1$, $x\in X$, and hence
 \eqref{psi_est_inequality_2} does not hold for $n=1$.
This example also points out the fact that, in part (iii) of Theorem \ref{Lem_psi_est_eq_5}, the assumption that the inequality $\vartheta_{1,\psi}(x)\leq \vartheta_{1,\varphi}(x)$ should hold for
all $x\in X$ is not a redundant one.
 \proofend
\end{Rem}

The following result is parallel to Theorem \ref{Lem_psi_est_eq_5}.

\begin{Thm}\label{Lem_phi_est_eq_5}
Let $\psi\in\Psi[Z](X,\Theta)$ and $\varphi\in\Psi[T,Z_1](X,\Theta)$.
Then both of the assertions (i) and (ii) of Theorem~\ref{Lem_psi_est_eq_5} and following statements are equivalent to each other.
\vspace{-3mm}
\begin{enumerate}
 \item[(iii)] For all $x\in X$, we have $\vartheta_{1,\psi}(x)\leq \vartheta_{1,\varphi}(x)$, and the inequality \eqref{psi_est_inequality_3}
 is valid for all $t\in \Theta$ and for all $x,y\in X$ with $\vartheta_{1,\psi}(x)<t<\vartheta_{1,\psi}(y)$.
 \item[(iv)] For all $x\in X$, we have $\vartheta_{1,\psi}(x)\leq \vartheta_{1,\varphi}(x)$, and there exists a nonnegative function $q:\Theta_\psi\to\RR_+$ such that
 \begin{align}\label{phi_est_inequality_5}
    q(t)\psi(z,t)\leq \varphi(z,t), \qquad z\in X,\, t\in\Theta_\psi.
 \end{align}
\end{enumerate}
\end{Thm}

The proof of Theorem \ref{Lem_phi_est_eq_5} is completely analogous to that of Theorem~\ref{Lem_psi_est_eq_5}, for the sake of completeness, we verify it in details.

\begin{proof}
(i) of Theorem \ref{Lem_psi_est_eq_5} $\Rightarrow$ (ii) of Theorem \ref{Lem_psi_est_eq_5}.
It was already proved.

(ii) of Theorem \ref{Lem_psi_est_eq_5} $\Rightarrow$  (iii).
Assume that (ii) of Theorem \ref{Lem_psi_est_eq_5} holds.
Exactly as in the proof of Theorem~\ref{Lem_psi_est_eq_5}, we have $\vartheta_{1,\psi}(x)\leq\vartheta_{1,\varphi}(x)$ for all $x\in X$.
We prove the inequality \eqref{psi_est_inequality_3} by contradiction.
Assume that for some $t\in\Theta$ and $x,y\in X$ with $\vartheta_{1,\psi}(x)<t<\vartheta_{1,\psi}(y)$, we have
 \begin{align*}
    \psi(x,t) \varphi(y,t) > \psi(y,t) \varphi(x,t).
 \end{align*}
In view of the inequalities $\vartheta_{1,\psi}(x)<t<\vartheta_{1,\psi}(y)\leq \vartheta_{1,\varphi}(y)$,
 we get that $\psi(x,t)<0$, $\psi(y,t)>0$ and $\varphi(y,t)>0$.
Thus, we obtain \eqref{help_10}, and, as in the corresponding part of the proof of Theorem~\ref{Lem_psi_est_eq_5},
 we have that there exist $k,m\in\NN$ such that
 \[
  k\varphi(x,t) + m \varphi(y,t) < 0 < k\psi(x,t) + m \psi(y,t).
 \]
Since $\psi$ possesses the property $[Z]$, in view of Lemma \ref{Lem_psi_est_eq_2}, we get the inequalities
 $$
  \vartheta_{k+m,\varphi}(\underbrace{x,\ldots,x}_{k}, \underbrace{y,\ldots,y}_{m})
  \leq  t < \vartheta_{k+m,\psi}(\underbrace{x,\ldots,x}_{k}, \underbrace{y,\ldots,y}_{m}),
 $$
 which contradicts \eqref{psi_est_inequality_2.5}.

(iii)$\Rightarrow$(iv).
Assume that (iii) holds. If $\Theta_\psi$ is empty, then there is nothing to prove.
Thus, we may assume that $\Theta_\psi\neq\emptyset$.
Rearranging the inequality \eqref{psi_est_inequality_3},
 for all $t\in\Theta_\psi$ and for all $x,y\in X$ with $\vartheta_{1,\psi}(x)<t<\vartheta_{1,\psi}(y)$, we get
 \begin{align}\label{help_7_uj}
  \frac{\varphi(x,t)}{\psi(x,t)}
  \leq\frac{\varphi(y,t)}{\psi(y,t)},
\end{align}
 where we used that $\psi(x,t)<0$ and $\psi(y,t)>0$.
Now we define the function $q:\Theta_\psi\to\RR$ by
$$
  q(t):=\inf \bigg\{\frac{\varphi(y,t)}{\psi(y,t)}\,\bigg|\,y\in X,\,t<\vartheta_{1,\psi}(y)\bigg\},
  \qquad t\in\Theta_\psi.
$$
Due to the inclusion $t\in\Theta_\psi$, the function $q$ is finite valued, and $q(t)\geq 0$ for all $t\in \Theta_\psi$.
Indeed, if $t\in\Theta_\psi$, then there exists $y\in X$ such that $t<\vartheta_{1,\psi}(y)$ yielding that
 $t<\vartheta_{1,\psi}(y)\leq\vartheta_{1,\varphi}(y)$, therefore $\psi(y,t)>0$ and $\varphi(y,t)>0$,
 and hence $q(t)$ is defined as the infimum of certain positive real numbers.
Further, by the definition of infimum and \eqref{help_7_uj}, for all $t\in\Theta_\psi$ and for all $x,y\in X$
 with $\vartheta_{1,\psi}(x)<t<\vartheta_{1,\psi}(y)$, we obtain that
\begin{equation}\label{help_6.5_uj}
  \frac{\varphi(x,t)}{\psi(x,t)}
  \leq q(t)\leq\frac{\varphi(y,t)}{\psi(y,t)}.
\end{equation}
Let $z\in X$ be fixed arbitrarily.
The first inequality in \eqref{help_6.5_uj} with $x:=z$ implies that \eqref{phi_est_inequality_5} holds for all $t\in\Theta_\psi$ with $\vartheta_{1,\psi}(z)<t$.
Similarly, the second inequality in \eqref{help_6.5_uj} with $y:=z$ yields that \eqref{phi_est_inequality_5} is also valid
 for all $t\in\Theta_\psi$ with $t<\vartheta_{1,\psi}(z)$.
Finally, if $t=\vartheta_{1,\psi}(z)$, then by the property $[Z]$ of $\psi$, we have that $\psi(z,t)=0$.
Furthermore, by the assumption in part (iii), the inequality $t=\vartheta_{1,\psi}(z)\leq \vartheta_{1,\varphi}(z)$ holds.
Using that $\varphi$ is a $Z_1$-function, it implies that $\varphi(z,t)\geq 0$.
This proves that \eqref{phi_est_inequality_5} holds for $t=\vartheta_{1,\psi}(z)$ as well.

(iv) $\Rightarrow$ (i) of Theorem \ref{Lem_psi_est_eq_5}.
Assume that (iv) is valid.
To prove (i) of Theorem \ref{Lem_psi_est_eq_5}, let $n\in\NN$ and $x_1,\dots,x_n\in X$.

In the proof of the inequality \eqref{psi_est_inequality_2}, we distinguish two cases.
First, consider the case when $\vartheta_{1,\psi}(x_1)=\dots=\vartheta_{1,\psi}(x_n)=:t_0$.
Then $\psi(x_1,t_0)=\dots= \psi(x_n,t_0)=0$, since $\psi$ possesses the property $[Z]$.
Therefore,
$$
  \sum_{i=1}^n\psi(x_i,t_0)=0,
$$
 which (using that $\psi$ is also a $T$-function or Lemma \ref{Lem_psi_est_eq_2}) implies that
$$
  \vartheta_{n,\psi}(x_1,\dots,x_n)=t_0.
$$
On the other hand, the inequality $\vartheta_{1,\psi}\leq  \vartheta_{1,\varphi}$ yields that
$$
  t_0=\vartheta_{1,\psi}(x_1)\leq \vartheta_{1,\varphi}(x_1),\qquad\dots,\qquad t_0=\vartheta_{1,\psi}(x_n)\leq \vartheta_{1,\varphi}(x_n),
$$
and hence, using that $\varphi$ is a $Z_1$-function, we have
$$
  \varphi(x_1,t_0)\geq0,\qquad\dots,\qquad
  \varphi(x_n,t_0)\geq0.
$$
Therefore
$$
  \sum_{i=1}^n\varphi(x_i,t_0)\geq0,
$$
which (using that $\varphi$ is a $T$-function) implies that
$$
  \vartheta_{n,\psi}(x_1,\dots,x_n)= t_0
   \leq \vartheta_{n,\varphi}(x_1,\dots,x_n).
$$
Thus \eqref{psi_est_inequality_2} is proved when $\vartheta_{1,\psi}(x_1)=\dots=\vartheta_{1,\psi}(x_n)$ holds.

Now consider the case, when
$$
  t':=\min(\vartheta_{1,\psi}(x_1),\dots,\vartheta_{1,\psi}(x_n))
  <\max(\vartheta_{1,\psi}(x_1),\dots,\vartheta_{1,\psi}(x_n))=:t''.
$$
Let $t_0:=\vartheta_{n,\psi}(x_1,\dots,x_n)$.
Then there exist $i,j\in\{1,\dots,n\}$ such that $t'=\vartheta_{1,\psi}(x_i)$ and $t''=\vartheta_{1,\psi}(x_j)$.
According to Proposition \ref{Pro_psi_becsles_mean_prop} (which can be applied, since $\psi$ has the property $[Z]$), we have that $t'<t_0<t''$, which shows that $t_0\in\Theta_\psi$.
Hence we can apply the inequality \eqref{phi_est_inequality_5} with $z:=x_i$, $i\in\{1,\dots,n\}$ and $t:=t_0$. Then, adding up the inequalities so obtained side by side, we arrive at
 $$
  0=q(t_0)\sum_{i=1}^n \psi(x_i,t_0)
    \leq \sum_{i=1}^n \varphi(x_i,t_0),
 $$
where the equality follows from the property $[Z]$ of $\psi$ and from the definition of $t_0$.
This, according to the property $[T_n]$ of $\varphi$, implies that
$$
  \vartheta_{n,\psi}(x_1,\dots,x_n) = t_0 \leq \vartheta_{n,\varphi}(x_1,\dots,x_n),
$$
and proves (i) in the considered second case as well.
\end{proof}

Note that if $\psi,\varphi\in\Psi[Z](X,\Theta)$, then the assertions (i), (ii), (iii) and (iv) of Theorem~\ref{Lem_psi_est_eq_5} and the assertions (iii) and (iv) of Theorem \ref{Lem_phi_est_eq_5}
are equivalent to each other.

In the next result, under some additional regularity assumptions on $\psi$ and $\varphi$, we derive another
 set of conditions that is equivalent to \eqref{psi_est_inequality_2}.

\begin{Thm}\label{Thm_inequality_diff}
Let $\psi,\varphi\in\Psi[C,Z](X,\Theta)$.
Assume that $\vartheta_{1,\psi}=\vartheta_{1,\varphi}=:\vartheta_1$ on $X$, $\vartheta_1(X)=\Theta$, and, for all $x\in X$, the maps
$$
  \Theta\ni t\mapsto\psi(x,t) \qquad\mbox{and}\qquad
  \Theta\ni t\mapsto\varphi(x,t)
$$
are differentiable at $\vartheta_1(x)$ with a non-vanishing derivative.
Then any of the equivalent assertions (i), (ii), (iii) and (iv) of Theorem~\ref{Lem_psi_est_eq_5} is equivalent to the following one:
\vspace{-3mm}
\begin{enumerate}
 \item[(v)] For all $x,y\in X$, we have
 \begin{align}\label{phi_est_inequality_6}
    -\frac{\psi(y,\vartheta_1(x))}{\partial_2\psi(x,\vartheta_1(x))}
    \leq -\frac{\varphi(y,\vartheta_1(x))}{\partial_2\varphi(x,\vartheta_1(x))}.
 \end{align}
\end{enumerate}
\end{Thm}

\begin{proof}
First, we check that $\Theta_\psi=\Theta_\varphi=\Theta$.
If $t\in\Theta$, then, using that $\Theta$ is open, there exist $t_1,t_2\in\Theta$ such that $t_1<t<t_2$.
Since $\vartheta_1(X)=\Theta$, there exist $x_1,x_2\in X$
 such that $\vartheta_1(x_1)=t_1$ and $\vartheta_1(x_2)=t_2$, yielding that $t\in\Theta_\psi$ and $t\in\Theta_\varphi$,
 and hence $\Theta_\psi=\Theta_\varphi=\Theta$.

Next, we verify that $\partial_2\psi(z,\vartheta_1(z))<0$ and $\partial_2\varphi(z,\vartheta_1(z)) < 0$ for all $z\in X$.
Using that $\psi$ is a $Z$-function, for all $z\in X$, we have $\psi(z,\vartheta_1(z))=0$, and hence
 \begin{align}\label{help_8}
   \partial_2\psi(z,\vartheta_1(z))
     = \lim_{t\to\vartheta_1(z)} \frac{\psi(x,t) - \psi(z,\vartheta_1(z))}{t-\vartheta_1(z)}
     = \lim_{t\to\vartheta_1(z)} \frac{\psi(x,t)}{t-\vartheta_1(z)}
     \leq 0.
 \end{align}
Indeed, if $t\in\Theta$ and $t<\vartheta_1(z)$, then $\psi(x,t)>0$ and $t-\vartheta_1(z)<0$;
 and if $t\in\Theta$ and $t>\vartheta_1(z)$, then $\psi(x,t)<0$ and $t-\vartheta_1(z)>0$.
\ Using the assumption $\partial_2\psi(z,\vartheta_1(z))\ne 0$, $z\in X$,
 \eqref{help_8} yields $\partial_2\psi(z,\vartheta_1(z))<0$, $z\in X$, as desired.
In the same way, we can get $\partial_2\varphi(z,\vartheta_1(z))<0$, $z\in X$.

Assume that any of the equivalent assertions (i)-(iv) of Theorem~\ref{Lem_psi_est_eq_5} holds.
We give two alternative proofs for \eqref{phi_est_inequality_6}.

{\sl First proof for \eqref{phi_est_inequality_6}:}
Let $x,y\in X$ be arbitrary.
We may suppose that $\vartheta_{1}(x) \ne \vartheta_{1}(y)$, since otherwise, using that both $\psi$ and $\varphi$ have the property $[Z]$, we get
 \[
   \psi(y,\vartheta_1(x)) = \varphi(y,\vartheta_1(x)) = 0,
 \]
and consequently, the inequality \eqref{phi_est_inequality_6} holds trivially.

For each $n\in\NN$, let $t_n:=\vartheta_{n+1,\psi}(x,\dots,x,y)$ and $\widetilde t_n:=\vartheta_{n+1,\varphi}(x,\dots,x,y)$,
 where $x$ is repeated $n$-times both for $t_n$ and $\widetilde t_n$.
Assertion (i) of Theorem~\ref{Lem_psi_est_eq_5} yields that, for each $n\in\NN$, we have
 \begin{align}\label{help_13}
 n(t_n-\vartheta_{1}(x)) \leq n(\widetilde t_n-\vartheta_{1}(x)).
 \end{align}
Then $t_n\to \vartheta_{1}(x)$ as $n\to\infty$ (see Lemma \ref{Lem_psi_est_cont})
 and, since $\psi$ has the property $[Z]$, we get
 \[
   n\psi(x,t_n)+\psi(y,t_n)=0, \qquad n\in\NN.
 \]
If $\psi(x,t_n) = 0$ for some $n\in\NN$, then, we also have $\psi(y,t_n)=0$, and hence $t_n = \vartheta_{1}(x) = \vartheta_{1}(y)$.
This leads us to a contradiction.
Consequently, $\psi(x,t_n)\ne 0$, and then $n=-\frac{\psi(y,t_n)}{\psi(x,t_n)}$ for any $n\in\NN$.
Therefore, using also that $\psi$ is a $C$-function, we get
\begin{align*}
  n(t_n-\vartheta_{1}(x))
  &=-\frac{\psi(y,t_n)}{\psi(x,t_n)}(t_n-\vartheta_{1}(x))\\
  &=-\psi(y,t_n)\frac{t_n-\vartheta_{1}(x)}{\psi(x,t_n)-\psi(x,\vartheta_{1}(x))}
  \to -\frac{\psi(y,\vartheta_{1}(x))}{\partial_2\psi(x,\vartheta_{1}(x))}
 \qquad \text{as $n\to\infty$.}
\end{align*}
Similarly, using that $\varphi\in\Psi[C,Z](X,\Theta)$, one can check that
 \begin{align*}
  n(\widetilde t_n-\vartheta_{1}(x))
  \to -\frac{\varphi(y,\vartheta_{1}(x))}{\partial_2\varphi(x,\vartheta_{1}(x))}
  \qquad \text{as \ $n\to\infty$.}
 \end{align*}
Now, upon taking the limit $n\to\infty$ in \eqref{help_13},
we can see that \eqref{phi_est_inequality_6} is valid.

{\sl Second proof for \eqref{phi_est_inequality_6}:}
Assertion (iii) of Theorem~\ref{Lem_psi_est_eq_5} states that \eqref{psi_est_inequality_3} holds for all $t\in\Theta$ and for all $x,y\in X$ with $\vartheta_1(x)<t<\vartheta_1(y)$.
Let $x,y\in X$ be arbitrarily fixed.
If $\vartheta_1(x) = \vartheta_1(y)$, then, using that $\psi$ and $\varphi$ are $Z$-functions, we have
 $\psi(y,\vartheta_1(x)) = \varphi(y,\vartheta_1(x))=0$,
 and hence \eqref{phi_est_inequality_6} readily holds.
If $\vartheta_1(x) < \vartheta_1(y)$, then, for all $t\in\Theta$ with $\vartheta_1(x)<t<\vartheta_1(y)$, the inequality \eqref{psi_est_inequality_3} is valid, and using that $\psi$ and $\varphi$ are
$Z$-functions, we get
 \[
  \frac{\psi(x,t)-\psi(x,\vartheta_1(x))}
    {t-\vartheta_1(x)} \varphi(y,t)
  \leq \psi(y,t) \frac{\varphi(x,t)-\varphi(x,\vartheta_1(x))}{t-\vartheta_1(x)},
   \qquad t\in(\vartheta_1(x), \vartheta_1(y)).
 \]
Upon taking the limit $t\downarrow\vartheta_1(x)$, using the differentiability assumption and that $\psi$ and $\varphi$ are $C$-functions, the previous inequality follows that
 \[
  \partial_2\psi(x,\vartheta_1(x))\cdot\varphi(y,\vartheta_1(x))
  \leq \psi(y,\vartheta_1(x))\cdot\partial_2\varphi(x,\vartheta_1(x)).
 \]
Using that $\partial_2\psi(z,\vartheta_1(z))<0$, $z\in X$, and $\partial_2\varphi(z,\vartheta_1(z)) < 0$,  $z\in X$,
the previous inequality implies \eqref{phi_est_inequality_6} in case of $\vartheta_1(x) < \vartheta_1(y)$.
If $\vartheta_1(x) > \vartheta_1(y)$, by a similar argument, we get that the inequality \eqref{psi_est_inequality_3} is valid
 for all $t\in\Theta$ with $\vartheta_1(y)<t< \vartheta_1(x)$.
Using that $\psi$ and $\varphi$ are $Z$-functions, we get
 \[
  \psi(y,t)
    \frac{\varphi(x,t) - \varphi(x,\vartheta_1(x))}{t-\vartheta_1(x)}
  \geq \frac{\psi(x,t)-\psi(x,\vartheta_1(x))}{t-\vartheta_1(x)}  \varphi(y,t),
   \qquad t\in(\vartheta_1(y), \vartheta_1(x)).
 \]
Upon taking the limit $t\uparrow\vartheta_1(x)$, using the differentiability assumption and that $\psi$ and $\varphi$ are $C$-functions, the previous inequality implies that
 \[
  \psi(y,\vartheta_1(x))\cdot\partial_2 \varphi(x,\vartheta_1(x))
  \geq \partial_2\psi(x,\vartheta_1(x))\cdot \varphi(x,\vartheta_1(x)).
 \]
Using that $\partial_2\psi(z,\vartheta_1(z))<0$ and $\partial_2\varphi(z,\vartheta_1(z)) < 0$,  $z\in X$,
the previous inequality implies \eqref{phi_est_inequality_6} also in the case of $\vartheta_1(x) > \vartheta_1(y)$.

Assume that \eqref{phi_est_inequality_6} holds.
We give two alternative proofs for one of the equivalent assertions (i)-(iv) of Theorem~\ref{Lem_psi_est_eq_5}.

\textit{First proof for the sufficiency (using the Axiom of Choice)}. We are going to show that our assertion (v) implies condition (iv) of Theorem~\ref{Lem_psi_est_eq_5}.

In view of the condition $\vartheta_1(X)=\Theta$ and applying the axiom of choice,
 there exists a right inverse of $\vartheta_1$, i.e., there exists a function $\rho:\Theta\to X$ such that $\vartheta_1(\rho(t))=t$
 is valid for all $t\in\Theta$. Then, substituting $x:=\rho(t)$ in the inequality \eqref{phi_est_inequality_6},
 for all $t\in\Theta$ and $y\in X$, we get that
 \begin{align}\label{phi_est_inequality_66}
    -\frac{\psi(y,t)}{\partial_2\psi(\rho(t),t)}
    \leq -\frac{\varphi(y,t)}{\partial_2\varphi(\rho(t),t)}.
 \end{align}
Using that $\partial_2\psi(\rho(t),t) = \partial_2\psi(\rho(t),\vartheta_1(\rho(t))) < 0$ and
   $\partial_2\varphi(\rho(t),t) = \partial_2\varphi(\rho(t),\vartheta_1(\rho(t))) < 0$ for all $t\in\Theta$,
 \eqref{phi_est_inequality_66}  implies that condition (iv) of Theorem~\ref{Lem_psi_est_eq_5} holds
 with the positive function $p:\Theta\to\RR$ defined by
 \[
  p(t):=\frac{\partial_2\psi(\rho(t),t)}{\partial_2\varphi(\rho(t),t)}, \qquad t\in\Theta.
 \]

\textit{Second proof for the sufficiency (without using the axiom of choice)}. We are going to show that our assertion (v) implies condition (iii) of Theorem~\ref{Lem_psi_est_eq_5}.

Assume that assertion (v) of the present theorem holds.
Let $t\in\Theta$ and $x,y\in X$ be such that $\vartheta_1(x) < t < \vartheta_1(y)$.
Since $\vartheta_1(X)=\Theta$, there exists $z\in X$ such that $t=\vartheta_1(z)$, yielding that $\vartheta_1(x)< \vartheta_1(z) < \vartheta_1(y)$.
By applying \eqref{phi_est_inequality_6} for the pairs $x,z$ and $y,z$, we get
 \begin{align}\label{help_9}
   -\frac{\psi(x,\vartheta_1(z))}{\partial_2\psi(z,\vartheta_1(z))}
    \leq -\frac{\varphi(x,\vartheta_1(z))}{\partial_2\varphi(z,\vartheta_1(z))}
    \qquad \text{and} \qquad
    -\frac{\psi(y,\vartheta_1(z))}{\partial_2\psi(z,\vartheta_1(z))}
    \leq -\frac{\varphi(y,\vartheta_1(z))}{\partial_2\varphi(z,\vartheta_1(z))}.
 \end{align}
Since $\varphi(x,\vartheta_1(z))<0$ and $\partial_2\varphi(z,\vartheta_1(z))<0$, the first inequality in \eqref{help_9} implies that
  \[
    0 < \frac{\varphi(x,\vartheta_1(z))}{\partial_2\varphi(z,\vartheta_1(z))}
      \leq \frac{\psi(x,\vartheta_1(z))}{\partial_2\psi(z,\vartheta_1(z))}.
  \]
Since $\psi(y,\vartheta_1(z))>0$ and $\partial_2\psi(z,\vartheta_1(z))<0$, the second inequality in \eqref{help_9} yields that
  \[
    0 <  -\frac{\psi(y,\vartheta_1(z))}{\partial_2\psi(z,\vartheta_1(z))}
    \leq -\frac{\varphi(y,\vartheta_1(z))}{\partial_2\varphi(z,\vartheta_1(z))}.
  \]
Multiplying the previous two inequalities, we have
 \[
   -\frac{\varphi(x,\vartheta_1(z))}{\partial_2\varphi(z,\vartheta_1(z))}
     \cdot\frac{\psi(y,\vartheta_1(z))}{\partial_2\psi(z,\vartheta_1(z))}
     \leq
    -\frac{\psi(x,\vartheta_1(z))}{\partial_2\psi(z,\vartheta_1(z))}
     \cdot\frac{\varphi(y,\vartheta_1(z))}{\partial_2\varphi(z,\vartheta_1(z))}.
 \]
Since $\partial_2\psi(z,\vartheta_1(z))\cdot \partial_2\varphi(z,\vartheta_1(z)) >0$, we have
 \[
    \psi(x,\vartheta_1(z)) \varphi(y,\vartheta_1(z)) \leq  \varphi(x,\vartheta_1(z)) \psi(y,\vartheta_1(z)).
 \]
Consequently, using that $\vartheta_1(z)=t$, we get
 \[
    \psi(x,t) \varphi(y,t) \leq \psi(y,t)\varphi(x,t),
 \]
i.e., we have \eqref{psi_est_inequality_3}, as desired.
This proves that assertion (iii) of Theorem~\ref{Lem_psi_est_eq_5} is valid.
\end{proof}

Theorem \ref{Thm_inequality_diff} gives back the comparison theorem for deviation means proved by Dar\'oczy \cite[part 1, Section 3]{Dar71b} (see also Dar\'oczy and P\'ales \cite[Theorem
2]{DarPal82}), since in the special case in question one can choose $X:=\Theta$ and we have that $\vartheta_1(x)=x$, $x\in X$.

Next, we can solve the equality problem for generalized $\psi$-estimators.

\begin{Thm}\label{Thm_equality}
Let $\psi\in\Psi[T,Z_1](X,\Theta)$ and $\varphi\in\Psi[Z](X,\Theta)$.
Assume that $\vartheta_{1,\psi}=\vartheta_{1,\varphi}$ on $X$.
Then $\Theta_\psi=\Theta_\varphi$ and the following assertions are equivalent:
\vspace{-3mm}
 \begin{enumerate}[(i)]
 \item The equality
 \begin{align}\label{psi_est_inequality_2_eq}
   \vartheta_{n,\psi}(x_1,\ldots,x_n) = \vartheta_{n,\varphi}(x_1,\ldots,x_n)
 \end{align}
 holds for each $n\in\NN$ and $x_1,\dots,x_n\in X$.
 \item The equality
 \Eq{psi_est_inequality_2.5_eq}{
  \vartheta_{k+m,\psi}(\underbrace{x,\ldots,x}_{k}, \underbrace{y,\ldots,y}_{m})
  = \vartheta_{k+m,\varphi}(\underbrace{x,\ldots,x}_{k}, \underbrace{y,\ldots,y}_{m})
 }
 holds for each $k,m\in\NN$ and $x,y\in X$.
 \item There exists a positive function $p:\Theta_\varphi\to(0,\infty)$ such that
 \begin{align}\label{psi_est_inequality_5_eq}
    \psi(z,t) = p(t)\varphi(z,t), \qquad z\in X,\, t\in\Theta_\varphi.
 \end{align}
\end{enumerate}
\end{Thm}

\begin{proof}
The equality $\Theta_\psi=\Theta_\varphi$ follows from $\vartheta_{1,\psi}=\vartheta_{1,\varphi}=: \vartheta_1$.

(i)$\Rightarrow$(ii). This implication becomes obvious by taking $n=k+m$ in assertion (i).

(ii)$\Rightarrow$(iii). Assume that (ii) holds.
If $\Theta_\psi=\Theta_\varphi =\emptyset$, then there is nothing to prove. Thus, we may assume that $\Theta_\psi=\Theta_\varphi \neq\emptyset$.

The equality \eqref{psi_est_inequality_2.5_eq} implies that the inequalities
\Eq{*}{
  \vartheta_{k+m,\psi}(\underbrace{x,\ldots,x}_{k}, \underbrace{y,\ldots,y}_{m})
  &\leq \vartheta_{k+m,\varphi}(\underbrace{x,\ldots,x}_{k}, \underbrace{y,\ldots,y}_{m}), \\
  \vartheta_{k+m,\varphi}(\underbrace{x,\ldots,x}_{k}, \underbrace{y,\ldots,y}_{m})
  &\leq \vartheta_{k+m,\psi}(\underbrace{x,\ldots,x}_{k}, \underbrace{y,\ldots,y}_{m})
}
hold for each $k,m\in\NN$ and $x,y\in X$. Then the equivalence of parts (ii) and (iv) of Theorem \ref{Lem_psi_est_eq_5} and that of parts (ii) and (iv) of Theorem \ref{Lem_phi_est_eq_5} yield that
there exist non-negative functions
 $p,q:\Theta_\varphi\to\RR_+$ such that
 \Eq{*}{
  \psi(z,t) &\leq p(t) \varphi(z,t), &\qquad z&\in X, \;\; t\in\Theta_\varphi,\\[2mm]
  q(t)\varphi(z,t) &\leq \psi(z,t), &\qquad z&\in X, \;\; t\in\Theta_\varphi,
 }
respectively. Consequently, we have
\begin{align}\label{help_12}
   q(t)\varphi(z,t)
   \leq p(t)\varphi(z,t), \qquad z\in X, \;\; t\in\Theta_\varphi.
 \end{align}
For all $t\in\Theta_\varphi$, there exist $x,y\in X$ such that $\vartheta_1(x)<t<\vartheta_1(y)$.
Taking the substitutions $z:=x$ and $z:=y$ in the inequality \eqref{help_12}, and using that $\varphi(x,t)<0$ and $\varphi(y,t)>0$, it follows that $p(t)=q(t)$ for all $t\in\Theta_\varphi$.
Consequently, $\psi(z,t)=p(t) \varphi(z,t)$ for all $z\in X$ and $t\in\Theta_\varphi$.
This equality shows that $p$ cannot vanish anywhere.
Indeed, if $p(t)=0$ for some $t\in\Theta_\varphi$, then $\psi(z,t)=0$ for all $z\in X$.
However, $t\in\Theta_\varphi=\Theta_\psi$ implies the existence of $x\in X$ such that $\vartheta_1(x)<t$,
 yielding that $\psi(x,t)<0$, which leads us to a contradiction.
Thus, $p$ has to be positive everywhere and the assertion (iii) holds.

(iii)$\Rightarrow$(i). Let $n\in\NN$ and $(x_1,\dots,x_n)\in X^n$ be fixed. If
\Eq{*}{
  \vartheta_1(x_1)=\dots=\vartheta_1(x_n):=t_0,
}
then, according to Proposition~\ref{Pro_psi_becsles_mean_prop}, it follows that
\Eq{*}{
  \vartheta_{n,\psi}(x_1,\dots,x_n)=t_0
  =\vartheta_{n,\varphi}(x_1,\dots,x_n).
}
Now assume that $\min(\vartheta_1(x_1),\dots,\vartheta_1(x_n))<\max(\vartheta_1(x_1),\dots,\vartheta_1(x_n))$. Then $\Theta_\psi=\Theta_\varphi$ is not the emptyset, and for each
$t\in\Theta_\varphi$, we have
 \[
     \sign\left( \sum_{i=1}^n \psi(x_i,t) \right)
         = \sign\left( \sum_{i=1}^n p(t)\varphi(x_i,t) \right)
         = \sign\left( \sum_{i=1}^n\varphi(x_i,t) \right),
 \]
since $p(t)>0$ and $t\in\Theta_\varphi$.
Recall that the nonempty set $\Theta_\varphi$ is the interior of the convex hull of $\vartheta_1(X)$ (see the discussion before
 Theorem \ref{Lem_psi_est_eq_5}).
Consequently, for $t\in\Theta$ with $t\leq\inf\Theta_\varphi$, we have that $t\leq\min(\vartheta_1(x_1),\dots,\vartheta_1(x_n))$, thus, for all $i\in\{1,\dots,n\}$, the property $[Z_1]$ of $\varphi$
and $\psi$ yield that $\psi(x_i,t)\geq0$ and $\varphi(x_i,t)\geq0$.
Furthermore, for some  $j\in\{1,\dots,n\}$, we have that $t<\vartheta_1(x_j)$, which implies the strict inequalities
 $\psi(x_j,t)>0$ and $\varphi(x_j,t)>0$ (in fact, one can choose $j$ as an index for which
 $\vartheta_1(x_j) = \max(\vartheta_1(x_1),\dots,\vartheta_1(x_n))$).
Therefore $\sum_{i=1}^n \psi(x_i,t)>0$ and $\sum_{i=1}^n\varphi(x_i,t)>0$ for $t\in\Theta$ with $t\leq\inf\Theta_\varphi$.
On the other hand, for $t\in\Theta$ with $t\geq\sup\Theta_\varphi$, we can similarly see that the reversed inequalities
  $\sum_{i=1}^n \psi(x_i,t)<0$ and $\sum_{i=1}^n\varphi(x_i,t)<0$ hold.
Thus, we have proved that
 \Eq{sign-eq}{
     \sign\left( \sum_{i=1}^n \psi(x_i,t) \right)
         = \sign\left( \sum_{i=1}^n\varphi(x_i,t) \right)
 }
 is valid for all $t\in\Theta$.
This, taking into account that $\psi$ and $\varphi$ have the property $[T_n]$, shows that the maps
 $$
  \Theta\ni t\mapsto\sum_{i=1}^n \psi(x_i,t)
  \qquad\mbox{and}\qquad
  \Theta\ni t\mapsto\sum_{i=1}^n \varphi(x_i,t)
 $$
must have the same point of sign change, whence the assertion (i) follows.
\end{proof}

\begin{Rem}\label{Rem_Thm_eq}
\rm
(i).\
Roughly speaking, the equivalence of parts (i) and (ii) of Theorem \ref{Thm_equality} means that if the $\psi$- and $\varphi$-estimators (satisfying the conditions of Theorem \ref{Thm_equality}) coincide on
any realization of samples of arbitrary length but having only two different observations, then the two estimators in question coincide on any realization of samples of arbitrary length without any restriction on the number of different observations.

(ii).\
Under the assumptions of Theorem \ref{Thm_equality}, if any of its equivalent assertions (i), (ii) and (iii) hold and $\Theta_\varphi$ is not empty, then $\psi$ has the property $[Z]$ as well.
Indeed, in the second case of the proof of (iii)$\Rightarrow$(i), we have checked that for all $n\in\NN$, $(x_1,\ldots,x_n)\in X^n$ and $t\in\Theta$, the sums $\sum_{i=1}^n \psi(x_i,t)$ and
$\sum_{i=1}^n \varphi(x_i,t)$ have the same sign (see \eqref{sign-eq}).
Consequently, using that $\varphi$ has the property $[Z]$, it follows that $\psi$ has the property $[Z]$ as well.

(iii).\
The statement in Theorem \ref{Thm_equality} remains true under the assumptions $\psi\in\Psi[Z](X,\Theta)$, $\varphi\in\Psi[T,Z_1](X,\Theta)$,
 and $\vartheta_{1,\psi}=\vartheta_{1,\varphi}$ on $X$ as well.
This can be readily seen from the proof of Theorem \ref{Thm_equality}.

(iv).\
The proof of Theorem \ref{Thm_equality} shows that its assumptions can be weakened in the following sense.
For the implication (i)$\Rightarrow$(ii), it is enough to assume that $\psi$ and $\varphi$ have the property $[T]$.
For the implication (iii)$\Rightarrow$(i), it is enough to assume that both $\psi$ and $\varphi$ possess the properties $[T]$ and $[Z_1]$.

(v).\
It can happen that $\vartheta_{n,\psi}(\bx) = \vartheta_{n,\varphi}(\bx)$ holds for all $n\in\NN$ and $\bx\in X^n$
 with $\psi,\varphi\in \Psi(X,\Theta)$, but Theorem \ref{Thm_equality} cannot be applied.
For example, let \ $X:=\{x_1,x_2\}$, \ $\Theta:=\RR$, \ and
 \begin{align*}
  &\psi(x_j,t):=j\big(\bone_{(-\infty,0]}(t) - \bone_{(0,\infty)}(t) \big),\qquad t\in\RR,\quad j\in\{1,2\},\\
  &\varphi(x_j,t):=-jt,\qquad t\in\RR,\quad j\in\{1,2\}.
 \end{align*}
Then $\psi$ is a $T$-function with $\vartheta_{n,\psi}(\bx)=0$ for all $n\in\NN$ and $\bx\in X^n$,
 and $\varphi$ is a $Z$-function with $\vartheta_{n,\varphi}(\bx)=0$ for all $n\in\NN$ and $\bx\in X^n$.
Consequently, $\Theta_\psi = \Theta_\varphi = \emptyset$, and $\vartheta_{n,\psi}(\bx) = \vartheta_{n,\varphi}(\bx)$
 for all $n\in\NN$ and $\bx\in X^n$.
Note also that $\psi$ does not satisfy the property $[Z_1]$, since $\psi(x_j,0)=j\ne 0$, $j\in\{1,2\}$,
 and hence we cannot apply Theorem \ref{Thm_equality}.
\proofend
\end{Rem}

In the next result, under some additional regularity assumptions on $\psi$ and $\varphi$, we derive another
 set of conditions that are equivalent to any of the equivalent conditions (i), (ii) and (iii) of Theorem \ref{Thm_equality}.
This result can be considered as a counterpart of the second part of Theorem 8 in P\'ales \cite{Pal88a} for quasi-deviation means.

\begin{Cor}\label{Cor_equality_diff}
Let $\psi,\varphi\in\Psi[C,Z](X,\Theta)$.
Assume that $\vartheta_{1,\psi}=\vartheta_{1,\varphi}=:\vartheta_1$ on $X$, $\vartheta_1(X)=\Theta$, and, for all $x\in X$, the maps
$$
  \Theta\ni t\mapsto\psi(x,t) \qquad\mbox{and}\qquad
  \Theta\ni t\mapsto\varphi(x,t)
$$
are differentiable at $\vartheta_1(x)$ with a non-vanishing derivative.
Then any of the equivalent assertions  (i), (ii) and (iii) of Theorem \ref{Thm_equality} is equivalent to the following one:
\vspace{-3mm}
\begin{enumerate}
 \item[(iv)] For all $x,y\in X$, we have
 \Eq{phi_est_equality_6}{
    \frac{\psi(y,\vartheta_1(x))}{\partial_2\psi(x,\vartheta_1(x))}
     = \frac{\varphi(y,\vartheta_1(x))}{\partial_2\varphi(x,\vartheta_1(x))}.
}
\end{enumerate}
\end{Cor}

\begin{proof}
It readily follows from Theorem \ref{Thm_inequality_diff}.
\end{proof}

\section{Comparison and equality of Bajraktarevi\'c-type estimators}
\label{Sec_comp_equal_Bajrak}

In this section we apply our results in Section \ref{Section_comp_equal_psi_est} for solving the comparison and equality problems
 of Bajraktarevi\'c-type estimators that are special generalized $\psi$-estimators.

First, we recall the notions of Bajraktarevi\'c-type functions and then Bajraktarevi\'c-type estimators.
Let $X$ be a nonempty set, $\Theta$ be a nondegenerate open interval of $\RR$, $f:\Theta\to\RR$ be strictly increasing, $p:X\to\RR_{++}$ and $F:X\to\conv(f(\Theta))$.
In terms of these functions, define $\psi:X\times\Theta\to\RR$ by
\Eq{help16}{
   \psi(x,t):=p(x)(F(x)-f(t)), \qquad x\in X, \;
   t\in\Theta.
}
By Lemma 1 in Gr\"unwald and P\'ales \cite{GruPal},
 there exists a uniquely determined monotone function $g:\conv(f(\Theta))\to \Theta$ such that $g$ is the left inverse of $f$, i.e.,
\Eq{*}{
   (g\circ f)(t)&=t, \qquad t\in \Theta.
}
Furthermore, $g$ is monotone in the same sense as $f$  (i.e, $f$ is monotone increasing),
 is continuous, and the following relation holds:
\Eq{*}{
  (f\circ g)(y)=y, \qquad y\in f(\Theta).
}
The function $g:\conv(f(\Theta))\to \Theta$ is called the \emph{generalized left inverse of the strictly increasing function $f:\Theta\to\RR$} and is denoted by $f^{(-1)}$.
It is clear that the restriction of $f^{(-1)}$ to $f(\Theta)$ is the inverse of $f$ in the standard sense, which is also  strictly increasing.
Therefore, $f^{(-1)}$ is the continuous and monotone extension of the inverse of $f$ to the smallest interval containing the range of $f$, that is, to the convex hull of $f(\Theta)$.

Recall also that, by Proposition 2.19 in Barczy and P\'ales \cite{BarPal2}, under the above assumptions, $\psi$ is a $T_n^{\blambda}$-function for each $n\in\NN$ and $\blambda\in\Lambda_n$, and
 \begin{align}\label{help14}
  \vartheta_{n,\psi}^{\blambda}(\bx)
  =f^{(-1)}\bigg(\frac{\lambda_1p(x_1)F(x_1)+\dots+\lambda_np(x_n)F(x_n)}{\lambda_1p(x_1)+\dots+\lambda_np(x_n)}\bigg)
 \end{align}
 for each $n\in\NN$, $\bx=(x_1,\dots,x_n)\in X^n$ and $\blambda=(\lambda_1,\dots,\lambda_n)\in\Lambda_n$.
In particular, $\vartheta_{1,\psi}=f^{(-1)}\circ F$ holds.
The value $\vartheta_{n,\psi}^{\blambda}(\bx)$ given by \eqref{help14} can be called as a Bajraktarevi\'c-type $\psi$-estimator
 of some unknown parameter in $\Theta$ based on the realization $\bx=(x_1,\ldots,x_n)\in X^n$ and weights $(\lambda_1,\ldots,\lambda_n)\in\Lambda_n$
 corresponding to the Bajraktarevi\'c-type function given by \eqref{help16}.
In particular, if $p=1$ is a constant function in \eqref{help16}, then we speak about a quasi-arithmetic-type $\psi$-estimator.

As a first result, we give a necessary and sufficient condition in order that
 $\Theta_\psi=\emptyset$ hold, where $\Theta_\psi$ is given by \eqref{theta_psi}.

\begin{Lem}\label{Lem_Theta_psi_empty}
Let $X$ be a nonempty set, $\Theta$ be a nondegenerate open interval of $\RR$, $f:\Theta\to\RR$ be a strictly increasing function,
 $p:X\to\RR_{++}$, $F:X\to \conv(f(\Theta))$, and define $\psi:X\times\Theta\to\RR$ by \eqref{help16}.
Then $\Theta_\psi = \emptyset$ holds if and only if there exists $t_0\in \Theta$ such that
 the range $F(X)$ of $F$ is contained in $[f(t_0-0),f(t_0+0)]$, where $f(t_0-0)$ and $f(t_0+0)$ denote the
 left and right hand limits of $f$ at $t_0$, respectively.
\end{Lem}

\begin{proof}
First, let us suppose that there exists $t_0\in \Theta$ such that $F(X)\subseteq J_f(t_0):=[f(t_0-0),f(t_0+0)]$.
Then, using that $f$ is strictly increasing, for all $x\in X$ and $t',t''\in\Theta$ with $t'<t_0<t''$,
 we have that $f(t')<f(t_0-0)\leq F(x)\leq f(t_0+0)<f(t'')$, and therefore, taking into account that $p$ is positive,
 for all $x\in X$, we get
 \[
   p(x)(F(x) - f(t))
            \begin{cases}
                 >0   & \text{if $t<t_0$, $t\in\Theta$,}\\
                 <0 & \text{if $t>t_0$, $t\in\Theta$.}
               \end{cases}
 \]
Hence, $\vartheta_{1,\psi}(x)=t_0$ for all $x\in X$.
This yields that $\Theta_\psi = \emptyset$.

To prove the converse statement, we check that if there does not exist $t_0\in\Theta$ such that $F(X)\subseteq J_f(t_0)$,
 then $\Theta_\psi\ne \emptyset$.
Since $f$ is strictly increasing, we have
 \[
    \conv(f(\Theta)) = \bigcup_{t\in\Theta} J_f(t),
 \]
 and $J_f(t')\cap J_f(t'') = \emptyset$ for all $t',t''\in\Theta$ with $t'\neq t''$.
Using also that $F(X)\subseteq \conv(f(\Theta))$, there exist $t_1,t_2\in\Theta$ with $t_1<t_2$ such that
 $F(X)\cap J_f(t_i)\ne \emptyset$, $i=1,2$.
Hence there exist $x_1,x_2\in X$ such that $F(x_i)\in J_f(t_i)$, $i=1,2$.
Consequently, similarly as before, we have
  \[
   p(x_i)(F(x_i) - f(t))
            \begin{cases}
                 >0   & \text{if $t<t_i$, $t\in\Theta$,}\\
                 <0 & \text{if $t>t_i$, $t\in\Theta$,}
               \end{cases}
               \qquad i=1,2,
 \]
yielding that $\vartheta_{1,\psi}(x_i)=t_i$, $i=1,2$.
Therefore, $(t_1,t_2)\subseteq \Theta_\psi$, showing that $\Theta_\psi$ is not empty, as expected.
\end{proof}

\begin{Lem}\label{Lem_p_negativ_ellentmod}
Let $X$ be a nonempty set, $\Theta$ be a nondegenerate open interval of $\RR$, $f:\Theta\to\RR$ be a strictly increasing function,
$p:X\to\RR_{++}$, $F:X\to \conv(f(\Theta))$, and define $\psi:X\times\Theta\to\RR$ by \eqref{help16}. Then, for all $x,y\in X$ with $\vartheta_{1,\psi}(x)<\vartheta_{1,\psi}(y)$, the map \eqref{umap} is positive and strictly increasing.
\end{Lem}

\begin{proof} Let $x,y\in X$ with $\vartheta_{1,\psi}(x)< \vartheta_{1,\psi}(y)$ and let $u\in (\vartheta_{1,\psi}(x), \vartheta_{1,\psi}(y))$ be arbitrary. Then $\psi(x,u)<0<\psi(y,u)$, which proves that the map \eq{umap} is positive valued. To see the strict monotonicity property, let additionally $v\in (\vartheta_{1,\psi}(x), \vartheta_{1,\psi}(y))$ be arbitrary with $u<v$. Then $\psi(x,v)<0<\psi(y,v)$, which implies that $F(x)<f(v)<F(y)$. Thus $F(x)<F(y)$ and, by the strict monotonicity of $f$, we also have $f(u)<f(v)$. Therefore
\Eq{*}{
  (F(y)-F(x))(f(v)-f(u))>0,
}
which is equivalent to the inequality
\Eq{*}{
  (F(x)-f(u))(F(y)-f(v))>(F(y)-f(u))(F(x)-f(v)).
}
Multiplying this inequality by $\frac{-p(x)}{p(y)(F(y)-f(u))(F(y)-f(v))}<0$ side by side, it follows that
\Eq{*}{
-\frac{\psi(x,u)}{\psi(y,u)}
  = -\frac{p(x)(F(x)-f(u))}{p(y)(F(y)-f(u))}
  <-\frac{p(x)(F(x)-f(v))}{p(y)(F(y)-f(v))}
  =-\frac{\psi(x,v)}{\psi(y,v)}.
}
This completes the proof of the strict increasingness of the map \eqref{umap}.

We note that the statement also follows from the proof of Proposition 2.19 in Barczy and P\'ales \cite{BarPal2}.
\end{proof}

In the following statement we describe sufficient conditions which imply that the function $\psi$ defined by \eqref{help16} possesses the property $[Z_1]$ and $[Z]$, respectively.

\begin{Lem}\label{Lem_Baj_psi_prop}
Let $X$ be a nonempty set, $\Theta$ be a nondegenerate open interval of $\RR$, $f:\Theta\to\RR$ be a strictly increasing function, $p:X\to\RR_{++}$, $F:X\to f(\Theta)$
and define $\psi:X\times\Theta\to\RR$ by \eqref{help16}.
Then $\psi$ has the property $[Z_1]$ and
\Eq{help_23A}{
    \Theta_\psi=\big\{t\in\Theta\mid \exists\, x,y\in X: F(x)<f(t)<F(y)\big\}.
}
If, in addition, $\conv(F(X))\subseteq f(\Theta)$, then $\psi$ has the property $[Z]$ as well.
\end{Lem}

\begin{proof}
Since $F(X)\subseteq f(\Theta)$, the restriction of $f^{(-1)}$ onto $F(X)$ is the strictly increasing inverse of $f$ in the standard sense restricted to $F(X)$. Thus, for all $x\in X$, we have
 \begin{align*}
    \psi(x,\vartheta_{1,\psi}(x))
    & = p(x) (F(x) - f(\vartheta_{1,\psi}(x)))
      = p(x) \big(F(x) - f(f^{(-1)}(F(x))) \big) \\
    & = p(x) (F(x) - F(x))=0,
 \end{align*}
yielding that $\psi$ has the property $[Z_1]$.
Furthermore, using also that $(f^{(-1)}\circ f)(t)=t$, $t\in\Theta$, we get that
\Eq{*}{
 \Theta_\psi&=\big\{t\in\Theta\mid \exists\, x,y\in X: \vartheta_{1,\psi}(x)<t<\vartheta_{1,\psi}(y)\big\} \\
    &= \big\{t\in\Theta\mid \exists\, x,y\in X: f^{(-1)}(F(x))<t<f^{(-1)}(F(y))\big\}\\
    &= \big\{t\in\Theta\mid \exists\, x,y\in X: (f\circ f^{(-1)})(F(x))<f(t)<(f\circ f^{(-1)})(F(y))\big\}\\
    &= \big\{t\in\Theta\mid \exists\, x,y\in X: F(x)<f(t)<F(y)\big\},
}
as desired.

To prove the last assertion, let us assume that $\conv(F(X))\subseteq f(\Theta)$.
Then, by \eqref{help14}, for each $n\in\NN$ and $\bx=(x_1,\ldots,x_n)\in X^n$, we have
 \begin{align*}
    \sum_{i=1}^n \psi(x_i,\vartheta_{n,\psi}(\bx))
       & = \sum_{i=1}^n p(x_i) \big( F(x_i) - f(\vartheta_{n,\psi}(\bx)) \big) \\
       & = \sum_{i=1}^n p(x_i) F(x_i) - \sum_{i=1}^n p(x_i) f\left(f^{(-1)}\left( \sum_{j=1}^n \frac{p(x_j)}{p(x_1)+\cdots +p(x_n)} F(x_j) \right) \right) \\
       & = \sum_{i=1}^n p(x_i) F(x_i) - \sum_{i=1}^n p(x_i) \sum_{j=1}^n \frac{p(x_j)}{p(x_1)+\cdots +p(x_n)} F(x_j)
        =0,
 \end{align*}
 where we used that $\sum_{j=1}^n \frac{p(x_j)}{p(x_1)+\cdots +p(x_n)} F(x_j)\in \conv(F(X)) \subseteq f(\Theta)$.
\end{proof}

In the next remark, we point out the fact that \eqref{help_23A} does not hold in general,
 showing that the assumption $F(X)\subseteq f(\Theta)$ in Lemma \ref{Lem_Baj_psi_prop} is indispensable.

\begin{Rem}\label{Rem_exl1}
\rm
Let $X:=\{x_1,x_2\}$, $\Theta:=\RR$, $p:X\to\RR_{++}$, $f:\RR\to\RR$,
 \[
    f(t):=\begin{cases}
            t & \text{if $t\leq 1$,}\\
            t+1 & \text{if $1<t\leq 2$,}\\
            t+2 & \text{if $t>2$,}
          \end{cases}
 \]
and $F:X\to \RR$ be such that $F(x_1):=1$ and $F(x_2):=3,5$.
Then $\conv(f(\Theta))=\RR$, and
  \[
   p(x_1)(F(x_1) - f(t))
            \begin{cases}
                 >0   & \text{if $t<1$,}\\
                 =0 & \text{if $t=1$,}\\
                 <0 & \text{if $t>1$,}
               \end{cases}
 \]
 and
 \[
   p(x_2)(F(x_2) - f(t))
            \begin{cases}
                 >0   & \text{if $t\leq 2$,}\\
                 <0 & \text{if $t>2$,}
               \end{cases}
 \]
 yielding that $\vartheta_{1,\psi}(x_j)=j$, $j=1,2$.
Hence $\Theta_\psi = (1,2)$.
However,
 \[
  \big\{ t\in\Theta \mid \exists\; x,y\in X: F(x) < f(t) < F(y)\big\}
     = (1,2],
 \]
 which does not coincide with $\Theta_\psi=(1,2)$.
Note that $F(X)\subseteq f(\Theta)$ is also not valid,
 and $\psi$ does not have the property $[Z_1]$, since $\psi(x_2,\vartheta_{1,\psi}(x_2)) = \psi(x_2,2)
 = p(x_2)(F(x_2) - f(2)) = \frac{p(x_2)}{2}>0$.
\proofend
\end{Rem}

Below, we solve the comparison problem for Bajraktarevi\'c-type estimators.

\begin{Pro}\label{Pro_Baj_type_comparison}
Let $X$ be a nonempty set, $\Theta$ be a nondegenerate open interval of $\RR$.
Let $f,g:\Theta\to\RR$ be strictly increasing functions, $p,q:X\to\RR_{++}$, $F:X\to f(\Theta)$, $G:X\to g(\Theta)$,
 and suppose that $\conv(G(X)) \subseteq g(\Theta)$.
Let $\psi:X\times \Theta\to\RR$ and $\varphi: X\times \Theta\to\RR$ be given by
 \begin{align}\label{help_Baj_psi_phi}
 \begin{split}
   &\psi(x,t):=p(x)(F(x)-f(t)), \qquad x\in X, \; t\in\Theta,\\
   &\varphi(x,t):=q(x)(G(x)-g(t)), \qquad x\in X, \; t\in\Theta.
  \end{split}
 \end{align}
Assume that $(f^{(-1)}\circ F)(x)\leq (g^{(-1)}\circ G)(x)$, $x\in X$.
Then $\vartheta_{n,\psi}(\bx)\leq \vartheta_{n,\varphi}(\bx)$ holds for each $n\in\NN$ and $\bx\in X^n$ if and only if
 the inequality
 \begin{align}\label{ineq_Bajrak_type_1}
    \frac{q(y)}{q(x)}\cdot \frac{G(y) - g(t)}{G(x) - g(t)}
       \leq  \frac{p(y)}{p(x)}\cdot \frac{F(y) - f(t)}{F(x) - f(t)}
 \end{align}
 is valid for all $t\in\Theta$ and for all $x,y\in X$ with $G(x)<g(t)<G(y)$.
\end{Pro}

\begin{proof}
By Proposition 2.19 in Barczy and P\'ales \cite{BarPal2} and Lemma \ref{Lem_Baj_psi_prop},
 $\psi$ has the properties $[T]$ and $[Z_1]$, and $\varphi$ has the property $[Z]$.
Using Theorem \ref{Lem_psi_est_eq_5}, we have that $\vartheta_{n,\psi}(\bx)\leq \vartheta_{n,\varphi}(\bx)$ holds
 for each $n\in\NN$ and $\bx\in X^n$ if and only if the inequality
 \begin{align*}
  \psi(x,t) \varphi(y,t) \leq \psi(y,t) \varphi(x,t)
 \end{align*}
 is valid for all $t\in\Theta$ and for all $x,y\in X$ with $\vartheta_{1,\varphi}(x)<t<\vartheta_{1,\varphi}(y)$.

Using $G(X)\subseteq g(\Theta)$, Lemma~\ref{Lem_Baj_psi_prop} yields that
\Eq{help_23}{
 \Theta_\varphi
 = \big\{t\in\Theta\mid \exists\, x,y\in X: G(x)<g(t)<G(y)\big\},
 }
and for all $t\in\Theta$ and $x,y\in X$, the inequality $\vartheta_{1,\varphi}(x)<t<\vartheta_{1,\varphi}(y)$
 holds if and only if $G(x) < g(t) < G(y)$.
Consequently, $\vartheta_{n,\psi}(\bx)\leq \vartheta_{n,\varphi}(\bx)$ holds for each $n\in\NN$ and $\bx\in X^n$ if and only if
 \begin{align}\label{help17}
   p(x)(F(x) - f(t)) q(y)(G(y) - g(t)) \leq p(y)(F(y) - f(t)) q(x)(G(x) - g(t))
 \end{align}
 is valid for all $t\in\Theta$ and for all $x,y\in X$ with $G(x)<g(t)<G(y)$.
Using that $f$ is strictly increasing, $F(X)\subseteq f(\Theta)$, and $g^{(-1)}$ restricted to $g(\Theta)$ is strictly increasing,
 for all $t\in\Theta$ and for all $x\in X$ with $G(x)<g(t)$, we have that
 \[
   F(x) = (f\circ f^{(-1)})(F(x)) = f(f^{(-1)}(F(x))) \leq f(g^{(-1)}(G(x)))
         < f( g^{(-1)}(g(t)) ) = f(t).
 \]
Consequently, $G(x) -g(t)<0$ and $F(x)-f(t)<0$ in the inequality \eqref{help17}, and hence by rearranging it, the assertion follows.
\end{proof}

In the next result, under some additional regularity assumptions on $f$ and $g$, we derive another set of conditions that is equivalent to \eqref{ineq_Bajrak_type_1}.

\begin{Pro}\label{Pro_Baj_type_comparison_2}
Let $X$ be a nonempty set, $\Theta$ be a nondegenerate open interval of $\RR$.
Let $f,g:\Theta\to\RR$ be strictly increasing functions, $p,q:X\to\RR_{++}$, $F:X\to f(\Theta)$, $G:X\to g(\Theta)$,
 and suppose that $F(X)=f(\Theta)$ and $\conv(G(X)) \subseteq g(\Theta)$.
Let $\psi:X\times \Theta\to\RR$ and $\varphi: X\times \Theta\to\RR$ be given by \eqref{help_Baj_psi_phi}.
 Assume that $(f^{(-1)}\circ F)(x) = (g^{(-1)}\circ G)(x)=:\vartheta_1(x)$, $x\in X$, and that
 $f$ and $g$ are differentiable at $\vartheta_1(x)$ for all $x\in X$ with non-vanishing (and hence positive) derivatives.
Then $\vartheta_{n,\psi}(\bx)\leq \vartheta_{n,\varphi}(\bx)$ holds for each $n\in\NN$ and $\bx\in X^n$ if and only if
 the inequality
 \begin{align}\label{ineq_Bajrak_type_2}
    \frac{p(y)}{p(x)} \cdot \frac{F(y) - F(x)}{f'\big(f^{(-1)}(F(x))\big)}
       \leq  \frac{q(y)}{q(x)} \cdot \frac{G(y) - G(x)}{g'\big(g^{(-1)}(G(x))\big)}
 \end{align}
 is valid for all $x,y\in X$.
\end{Pro}

\begin{proof}
Since $F(X)=f(\Theta)$, we have that
 \begin{align}\label{help_15}
 \vartheta_1(X)=(f^{(-1)}\circ F)(X)=f^{(-1)}(F(X))=f^{(-1)}(f(\Theta))=\Theta.
 \end{align}
Further, for all $x\in X$, we have
 \begin{align*}
  &\partial_2\psi(x,\vartheta_1(x)) = -p(x) f'(\vartheta_1(x)) = -p(x) f'\big(f^{(-1)}(F(x))\big),\\
  &\partial_2\varphi(x,\vartheta_1(x)) = -q(x) g'(\vartheta_1(x)) = -q(x) g'\big(g^{(-1)}(G(x))\big).
 \end{align*}
Consequently, using Theorem \ref{Thm_inequality_diff}, we have that
 $\vartheta_{n,\psi}(\bx)\leq \vartheta_{n,\varphi}(\bx)$ holds for each $n\in\NN$ and $\bx\in X^n$ if and only if
 \begin{align*}
    \frac{p(y)}{p(x)} \cdot \frac{F(y) - f(\vartheta_1(x))}{f'(\vartheta_1(x))}
       \leq  \frac{q(y)}{q(x)} \cdot \frac{G(y) - g(\vartheta_1(x))}{g'(\vartheta_1(x))}
 \end{align*}
 is valid for all $x,y\in X$, which yields the statement, since
 \[
   f(\vartheta_1(x)) = f(f^{(-1)}(F(x))) = (f\circ f^{(-1)})(F(x))
                     = F(x),\qquad x\in X,
 \]
 due to the condition $F(X)\subseteq f(\Theta)$, and similarly, we also have $g(\vartheta_1(x)) = G(x)$, $x\in X$.
\end{proof}

In the next result, among others, we point out that, under the assumptions of Proposition \ref{Pro_Baj_type_comparison_2},
the function $g$ is continuous.

\begin{Lem}\label{Rem_psi_varphi}
Let $X$ be a nonempty set and $\Theta$ be a nondegenerate open interval of $\RR$.
Let $f,g:\Theta\to\RR$ be strictly increasing functions, $F:X\to f(\Theta)$, $G:X\to g(\Theta)$,
 and suppose that $F(X)=f(\Theta)$ and $\conv(G(X)) \subseteq g(\Theta)$.
Assume that $(f^{(-1)}\circ F)(x) = (g^{(-1)}\circ G)(x)$, $x\in X$.
Then
 \begin{enumerate}
  \item[(i)] $G(X)=g(\Theta)$ and this set is convex.
  \item[(ii)] $g$ is continuous, and $G(x) = (g\circ f^{(-1)})(F(x))$, $x\in X$.
 \end{enumerate}
\end{Lem}

\begin{proof}
(i). Since $F(X)=f(\Theta)$, we have that
 \[
  (f^{(-1)}\circ F)(X)=f^{(-1)}(F(X))=f^{(-1)}(f(\Theta))=\Theta,
 \]
 which implies that $(g^{(-1)}\circ G)(X) = (f^{(-1)}\circ F)(X) = \Theta$ holds as well.
Hence, using that $G(X)\subseteq g(\Theta)$, we have that $g(\Theta) = g((g^{(-1)}\circ G)(X))=(g\circ g^{(-1)})(G(X)) = G(X)$.
Consequently, $\conv(g(\Theta)) = \conv(G(X))$, and, since $\conv(G(X)) \subseteq g(\Theta)$, we obtain that $g(\Theta)\subseteq \conv(g(\Theta)) = \conv(G(X)) \subseteq g(\Theta)$, yielding that $g(\Theta)=\conv(G(X))$ is a convex set.

(ii). Since $g$ is strictly increasing and its range $g(\Theta)$ is convex (see part (i)), we can see that $g$ is continuous as well.
Finally, note that the equality $(f^{(-1)}\circ F)(x) = (g^{(-1)}\circ G)(x)$, $x\in X$, implies that $G(x) = (g\circ f^{(-1)})(F(x))$, $x\in X$,
 since $G(X)\subseteq g(\Theta)$.
\end{proof}

Next, we solve the equality problem for Bajraktarevi\'c-type estimators.
In the proof, we will use a result on the Schwarzian derivative of a function. Given a nondegenerate open interval $I\subseteq \RR$, for a three times differentiable function $h:I \to \RR$ with a
nonvanishing first derivative, its Schwarzian derivative $S_h: I \to \RR$ is defined by
 \[
   S_h(x):=\frac{h'''(x)}{h'(x)} - \frac{3}{2} \left(\frac{h''(x)}{h'(x)}\right)^2,\qquad x\in I.
 \]
The following result is well-known, see, e.g., Gr\"unwald and P\'ales \cite[Corollary 3]{GruPal2}.

\begin{Lem}\label{Lem_Sch_deriv}
Let $I\subseteq \RR$ be a nondegenerate open interval, and $h:I\to\RR$ be a three times differentiable
 function such that $h'$ does not vanish on $I$.
Then $S_h(x)=0$, $x\in I$, holds if and only if there exist four constants $a,b,c,d\in\RR$ with $ad\ne bc$
 and $0\notin cI +d$ such that
  \[
    h(x) = \frac{ax+b}{cx+d}, \qquad x\in I.
  \]
\end{Lem}

In the proof of the subsequent theorems, the following auxiliary result plays an important role.

\begin{Lem}\label{Lem_aux}
Let $I$ be a nondegenerate open interval of $\RR$.
Let $f,g:I\to\RR$ be strictly increasing functions such that there exist four constants $a,b,c,d\in\RR$ with $0\notin c f(I)+d$ and
 \[
    g(t)=\frac{af(t)+b}{cf(t)+d},\qquad t\in I.
 \]
Then $ad>bc$ and $cf+d$ is either everywhere positive or everywhere negative on $I$.
\end{Lem}

\begin{proof} The condition $0\notin c f(I)+d$ yields that $(c,d)\neq(0,0)$. If $ad=bc$, then there exists $\lambda\in\RR$ such that $(a,b)=\lambda(c,d)$.
In this case, we get that $g(t)=\lambda$ for all $t\in I$, which contradicts the strict monotonicity of $g$. Thus $ad\neq bc$ must hold.

Next, we check that $cf+d$ is either everywhere positive or everywhere negative on $I$.
On the contrary, if $cf+d$ changes sign in $I$, then $c$ cannot be zero and hence $cf+d$ is also strictly monotone.
Therefore, using also that $I$ is a nondegenerate open interval, there exists a unique point $\tau\in I$ such that $cf(t)+d>\!(<)\,0$  for all $t<\tau$, $t\in I$, and $cf(t)+d<\!(>)\,0$ for all $t>\tau$, $t\in I$. Let $t<\tau<r<s$ be arbitrarily fixed elements of $I$. Then $(cf(t)+d)(cf(r)+d)<0$ and $(cf(r)+d)(cf(s)+d)>0$. Consequently, using the strict increasingness of $g$, the inequalities
 \[
   \frac{af(t)+b}{cf(t)+d}=g(t)
   <g(r)=\frac{af(r)+b}{cf(r)+d}
   \qquad\mbox{and}\qquad
   \frac{af(r)+b}{cf(r)+d}=g(r)
   <g(s)=\frac{af(s)+b}{cf(s)+d}
 \]
 imply that
 \[
  (af(t)+b)(cf(r)+d)>(af(r)+b)(cf(t)+d),\qquad
  (af(r)+b)(cf(s)+d)<(af(s)+b)(cf(r)+d),
 \]
 or equivalently,
 \begin{align}\label{help_25}
   0>(ad-bc)(f(r)-f(t))
   \qquad\mbox{and}\qquad
   0<(ad-bc)(f(s)-f(r)).
 \end{align}
Since $f$ is also strictly increasing, we have that $f(t)<f(r)<f(s)$.
Therefore, $(ad-bc)(f(r)-f(t))$ and $(ad-bc)(f(s)-f(r))$ should have the same signs.
This together with the inequalities \eqref{help_25} lead us to a contradiction.

Finally, to show that $ad>bc$, let $r,s\in I$ with $r<s$ be fixed.
Then, using that $cf+d$ does not change sign in $I$, we have $(cf(r)+d)(cf(s)+d)>0$,
 and therefore the inequality $g(r)<g(s)$, in the same way as above, implies $0<(ad-bc)(f(s)-f(r))$.
This, in view of the strict increasingness of $f$ yields that $ad-bc>0$.
\end{proof}

\begin{Thm}\label{Thm_Baj_type_equality}
Let $X$ be a nonempty set and $\Theta$ be a nondegenerate open interval of $\RR$.
Let $f,g:\Theta\to\RR$ be strictly increasing functions such that $f$ is continuous,
 $p,q:X\to\RR_{++}$, $F:X\to f(\Theta)$, $G:X\to g(\Theta)$, and suppose that $F(X)=f(\Theta)$ and $\conv(G(X)) \subseteq g(\Theta)$.
Let $\psi:X\times \Theta\to\RR$ and $\varphi: X\times \Theta\to\RR$ be given by \eqref{help_Baj_psi_phi}.
If $\vartheta_{n,\psi}(\bx) = \vartheta_{n,\varphi}(\bx)$ holds for each $n\in\NN$ and $\bx\in X^n$,
 then there exist four constants $a,b,c,d\in\RR$ with $ad\ne bc$ and $0\notin c f(\Theta)+d$ such that
 \begin{align}\label{gGq}
    g(t)=\frac{af(t)+b}{cf(t)+d},\,\,\, t\in\Theta\;\;
    \mbox{and}\;\;
    G(x)=\frac{aF(x)+b}{cF(x)+d},\;\;\;
    q(x)=(cF(x)+d)p(x),\,\,\, x\in X.
 \end{align}
\end{Thm}

\begin{proof}
Since $f$ is strictly increasing and continuous and $\Theta$ is a non-degenerate open interval,
 we have $f(\Theta)$ is a non-degenerate open interval.
Hence $\conv(f(\Theta)) = f(\Theta) = F(X)$, and then, as a consequence of Lemma \ref{Lem_Baj_psi_prop},
 we get $\psi\in \Psi[Z](X,\Theta)$.
Further, since $\conv(G(X))\subseteq g(\Theta)$, Lemma \ref{Lem_Baj_psi_prop} also yields that $\varphi\in \Psi[Z](X,\Theta)$.

We first verify that $\Theta_\psi=\Theta$. The inclusion $\Theta_\psi\subseteq\Theta$ is trivial.
To prove the reversed one, let $t\in\Theta$ be arbitrary. Now choose $r,s\in\Theta$ such that $r<t<s$.
Using that $F(X)=f(\Theta)$, we can find $x,y\in X$ such that $f(r)=F(x)$ and $f(s)=F(y)$.
Since $f$ is strictly increasing, it follows that $F(x)=f(r)<f(t)<f(s)=F(y)$,
 showing that $t$ belongs to the set
 \[
   \big\{t\in\Theta\mid \exists\, x,y\in X: F(x)<f(t)<F(y)\big\},
 \]
 which, according to \eqref{help_23A}, equals $\Theta_\psi$.

Assume that  $\vartheta_{n,\psi}(\bx) = \vartheta_{n,\varphi}(\bx)$ holds for each $n\in\NN$ and $\bx\in X^n$.
Then, in the case $n=1$, this equality and \eqref{help14} imply that $(f^{(-1)}\circ F)(x) = \vartheta_{1,\psi}(x)
 =\vartheta_{1,\varphi}(x) = (g^{(-1)}\circ G)(x)$, $x\in X$.
Hence, according to Lemma \ref{Rem_psi_varphi}, we get that $G(X)=\conv(G(X))=g(\Theta)$ is a convex set and $g$ is continuous.
Then, similarly as we derived $\Theta_\psi = \Theta$, we have that $\Theta_\varphi = \Theta$ holds as well.

For all $x,y\in X$ with $G(x)<G(y)$, let us introduce the notation
 \[
    \Theta_{x,y}:=\{t\in\Theta\mid G(x)<g(t)<G(y)\}.
 \]
Using that $F(X)\subseteq f(\Theta)$, $G(X)\subseteq g(\Theta)$, and that the restrictions of $f^{(-1)}$ and $g^{(-1)}$ to
 $f(\Theta)$ and $g(\Theta)$ are the strictly increasing inverses of $f$ and $g$ in the standard sense, respectively,
 for all $x,y\in X$ with $G(x)<G(y)$, we have
 \begin{align}\label{help_17}
   \begin{split}
    \Theta_{x,y}& =\{t\in\Theta\mid g^{(-1)}(G(x))<t<g^{(-1)}(G(y))\}\\
                & = \{t\in\Theta\mid f^{(-1)}(F(x))<t<f^{(-1)}(F(y))\} \\
                & = \{t\in\Theta\mid F(x)<f(t)<F(y)\}.
   \end{split}
 \end{align}
The previous argument also shows that for all $x,y\in X$, we get $G(x)<G(y)$ if and only if $F(x)<F(y)$,
 and the set $\Theta_{x,y}$ is a nonempty open interval for all $x,y\in X$ with $G(x)<G(y)$,
 since it is the intersection of the open intervals $(g^{(-1)}(G(x)), g^{(-1)}(G(y)))$ and $\Theta$.

In view of \eqref{help_17}, Lemma \ref{Lem_Baj_psi_prop} and the equality $\Theta_\psi=\Theta$ we can see that
\Eq{help_16.5}{
   \Theta_\varphi
     &= \bigcup_{\{ x,y\in X: G(x)<G(y)\} }\Theta_{x,y} \\
     &=\bigcup_{\{ x,y\in X: F(x)<F(y)\} }
        \{t\in\Theta\mid F(x)<f(t)<F(y)\}
      =\Theta_\psi=\Theta.
}
Using that the image under $f$ of a union of subsets is the union of the images under $f$ of the given subsets, \eqref{help_16.5} immediately yields that
\begin{align}\label{help_16}
    \bigcup_{\{ x,y\in X: G(x)<G(y)\} }f(\Theta_{x,y}) = f(\Theta),
\end{align}
which is an open interval, since $\Theta$ is a nonempty open interval and $f$ is strictly increasing and continuous.

Using Proposition \ref{Pro_Baj_type_comparison}, we have that
\begin{align}\label{help_14}
    \frac{q(y)}{q(x)}\cdot \frac{G(y) - g(t)}{G(x) - g(t)}
       = \frac{p(y)}{p(x)}\cdot \frac{F(y) - f(t)}{F(x) - f(t)}
\end{align}
holds for all $t\in\Theta$ and for all $x,y\in X$ with $G(x)<g(t)<G(y)$, or equivalently, \eqref{help_14} holds for all $x,y\in X$ with $G(x)<G(y)$ and for all $t\in\Theta_{x,y}$.

One can readily check that for all $x,y\in X$ with $G(x)<G(y)$ and for all $t\in\Theta_{x,y}$,
the equality \eqref{help_14} is equivalent to any of the following three equalities:
 \begin{align}\label{help18}
  \begin{split}
   &p(x)(F(x) - f(t)) q(y)(G(y) - g(t)) =p(y)(F(y) - f(t)) q(x)(G(x) - g(t)),\\[1mm]
   &\Big(p(y)(F(y) - f(t)) q(x)-p(x)(F(x) - f(t)) q(y)\Big)g(t)\\
   &\qquad  =p(y)(F(y) - f(t)) q(x)G(x)-p(x)(F(x) - f(t)) q(y)G(y),\\[1mm]
   &(c_{x,y}f(t)+d_{x,y})g(t)=a_{x,y}f(t)+b_{x,y},
  \end{split}
 \end{align}
where
 \begin{align*}
  a_{x,y}&:=p(x)q(y)G(y)-p(y)q(x)G(x), &\qquad b_{x,y}&:= p(y)q(x)F(y)G(x)-p(x)q(y)F(x)G(y),\\
  c_{x,y}&:=p(x)q(y)-p(y)q(x),  &\qquad  d_{x,y}&:= p(y)q(x)F(y)-p(x)q(y)F(x).
 \end{align*}
Here, due to $G(x)\neq G(y)$, we have that $(a_{x,y},c_{x,y})\neq(0,0)$ and $(b_{x,y},d_{x,y})\neq (0,0)$ hold.
Substituting $s:=f(t)$ (i.e., $t=f^{(-1)}(s)$) in the third equality in \eqref{help18}, it follows that
 \begin{align}\label{help_24}
     (c_{x,y}s+d_{x,y})(g\circ f^{(-1)})(s)=a_{x,y}s+b_{x,y}
 \end{align}
for all $x,y\in X$ with $G(x)<G(y)$ and for all $s\in f(\Theta_{x,y})$.

Next, we check that $c_{x,y}s+d_{x,y}\ne 0$ for all $x,y\in X$ with $G(x)<G(y)$ and for all $s\in f(\Theta_{x,y})$.
If $c_{x,y}s+d_{x,y}=0$ and $c_{x,y}=0$ were true, then $d_{x,y}=0$, $a_{x,y}\ne 0$, $b_{x,y}\ne 0$ and $a_{x,y}s+b_{x,y}=0$ (following from \eqref{help_24}).
This leads us to a contradiction, since $c_{x,y}=d_{x,y}=0$ implies that $p(x)q(y) = p(y)q(x)$ and $F(x)=F(y)$, which cannot happen due to $F(x)<F(y)$.
If $c_{x,y}s+d_{x,y}=0$ and $c_{x,y}\ne 0$ were true, then $s=-\frac{d_{x,y}}{c_{x,y}}$ and $a_{x,y}s+b_{x,y}=0$, yielding that $c_{x,y}b_{x,y}-d_{x,y}a_{x,y}=0$.
This leads us to a contradiction, since an easy calculation shows that
 \begin{align}\label{help_19}
 c_{x,y}b_{x,y}-d_{x,y}a_{x,y}
   = p(x)p(y)q(x)q(y)(F(x)-F(y))(G(y)-G(x)),
 \end{align}
 which cannot be $0$ for any $x,y\in X$ with $G(x)<G(y)$.

Consequently,
 \begin{align}\label{help_18}
     (g\circ f^{(-1)})(s)=\frac{a_{x,y}s+b_{x,y}}{c_{x,y}s+d_{x,y}}
 \end{align}
 for all $x,y\in X$ with $G(x)<G(y)$ and for all $s\in f(\Theta_{x,y})$.

We can apply Lemma \ref{Lem_Sch_deriv} to the function $h:=g\circ f^{(-1)}$ defined on the nondegenerate open interval $I:=f(\Theta)$, since \eqref{help_18} implies that $h$ is three times
differentiable on $I$ such that $h'$ does not vanish on $I$.
Indeed, using \eqref{help_18}, we have that
 \[
   h'(s) = \frac{a_{x,y}d_{x,y} - b_{x,y}c_{x,y}}{(c_{x,y}s+d_{x,y})^2} \ne 0, \qquad s\in f(\Theta_{x,y})
 \]
 for all $x,y\in X$ with $G(x)<G(y)$, where we used \eqref{help_19}.
Hence, as a consequence of \eqref{help_16}, we have $h'(s)\ne 0$, $s\in f(\Theta)$.
Taking into account that $f(\Theta_{x,y})$ is a nondegenerate open interval for all $x,y\in X$ with $G(x)<G(y)$,
 Lemma \ref{Lem_Sch_deriv} and \eqref{help_18} imply that $S_h(s)=0$, $s\in f(\Theta)$.
Consequently, using again Lemma \ref{Lem_Sch_deriv}, there exist four constants $a^*,b^*,c^*,d^*\in\RR$ with $a^*d^*\ne b^*c^*$
 and $0\notin c^*f(\Theta) + d^*$ such that
 \begin{align}\label{help_20}
    h(s) = (g\circ f^{(-1)})(s) = \frac{a^*s+b^*}{c^*s+d^*}, \qquad s\in f(\Theta).
 \end{align}
By substituting $s:=f(t)$, where $t\in\Theta$, it follows that
 \begin{align}\label{eq_g_form}
  g(t) = \frac{a^*f(t)+b^*}{c^*f(t)+d^*}, \qquad t\in \Theta,
 \end{align}
 as desired.
Using \eqref{help_20}, the assumptions $f^{(-1)}\circ F=g^{(-1)}\circ G$ and $G(X)\subseteq g(\Theta)$, we get that
 \begin{align}\label{eq_G_form}
   G(x)=g( (f^{(-1)}\circ F)(x) ) = (g\circ f^{(-1)})(F(x))=\frac{a^*F(x)+b^*}{c^*F(x)+d^*},
   \qquad x\in X,
 \end{align}
 where $0\notin c^*F(X) +d^*$, since $F(X)=f(\Theta)$ and $0\notin c^* f(\Theta) + d^*$.

By \eqref{help_14} and taking into account the forms of $G$ and $g$, we get
 \begin{align*}
    \frac{q(y)}{q(x)} \cdot \frac{ \dfrac{a^*F(y) + b^*}{c^*F(y) + d^*} - \dfrac{a^*f(t) + b^*}{c^*f(t) + d^*} }
                                 {\dfrac{a^*F(x) + b^*}{c^*F(x) + d^*} - \dfrac{a^*f(t) + b^*}{c^*f(t) + d^*}}
                    =  \frac{p(y)}{p(x)} \cdot \frac{F(y) - f(t)}{F(x) - f(t)}
 \end{align*}
 holds for all $x,y\in X$ with $G(x)<G(y)$ and for all $t\in \Theta_{x,y}$.
Using that $a^*d^*-b^*c^*\ne0$, after some algebraic calculations, we obtain that
 \[
    \frac{q(y)}{p(y)} = \frac{c^*F(y)+d^*}{c^*F(x)+d^*} \cdot \frac{q(x)}{p(x)}
 \]
 holds for all $x,y\in X$ with $G(x)<G(y)$, or equivalently,
 \[
     \frac{\Big(\dfrac{q}{p}\Big)(y)}{c^*F(y)+d^*}  =  \frac{\Big(\dfrac{q}{p}\Big)(x)}{c^*F(x)+d^*}
 \]
 holds for all $x,y\in X$ with $G(x)<G(y)$.
Since $q/p$ is positive, it follows that there exists a constant $k\in\RR\setminus\{0\}$ such that
 \[
   q(x) = k(c^*F(x) + d^*)p(x), \qquad x\in X.
 \]

The statement of the proposition now holds with the choices $a:=ka^*$, $b:=kb^*$, $c:=kc^*$ and $d:=kd^*$.
\end{proof}

We note that in the proof of Theorem \ref{Thm_Baj_type_equality}, the assumption that $f$ is continuous is used
 for deriving that $f(\Theta)$ is an open interval, which is essential when we apply Lemma \ref{Lem_Sch_deriv}.
Note also that in the proof of Theorem \ref{Thm_Baj_type_equality} it turned out that $g$ is continuous as well, however, we did not utilize this property in the proof.
Nonetheless, the continuity of $g$ also follows from the result itself, since $f$ is continuous and $g(t)=(af(t)+b)/(cf(t)+d)$, $t\in\Theta$.

Next, we will provide a set of sufficient conditions on $f$, $g$, $F$ and $G$ in order that $\vartheta_{n,\psi}(\bx) = \vartheta_{n,\varphi}(\bx)$ hold for each $n\in\NN$ and $\bx\in X^n$.

\begin{Thm}\label{Thm_Baj_type_equality_2}
Let $X$ be a nonempty set and $\Theta$ be a nondegenerate open interval of $\RR$.
Let $f,g:\Theta\to\RR$ be strictly increasing functions, $p,q:X\to\RR_{++}$, $F:X\to \conv(f(\Theta))$, and $G:X\to\conv(g(\Theta))$.
Let $\psi:X\times \Theta\to\RR$ and $\varphi: X\times \Theta\to\RR$ be given by \eqref{help_Baj_psi_phi}.
If there exist four constants $a,b,c,d\in\RR$ with $0\notin c f(\Theta)+d$ such that \eqref{gGq} holds, then $\vartheta_{n,\psi}(\bx) = \vartheta_{n,\varphi}(\bx)$ holds for each $n\in\NN$ and $\bx\in X^n$.
\end{Thm}

\begin{proof}
Since $p$ and $q$ are strictly positive functions, as a consequence of the equality $q=(cF+d)p$, we get that $cF+d$ is positive on $X$.
Further, for all $x\in X$ and $t\in \Theta$, we obtain
 \begin{align}\label{help_Thm_Baj_type_equality_2}
  \begin{split}
 \varphi(x,t)&=q(x)(G(x)-g(t))
   =(cF(x)+d)p(x)\bigg(\frac{aF(x)+b}{cF(x)+d}-\frac{af(t)+b}{cf(t)+d}\bigg)\\
  &=p(x)\frac{(ad-bc)(F(x)-f(t))}{cf(t)+d}
   =\frac{ad-bc}{cf(t)+d}\psi(x,t).
 \end{split}
 \end{align}

Using Lemma \ref{Lem_aux}, we have that $ad>bc$, and $cf+d$ is either everywhere positive or everywhere negative on $\Theta$.
We show that the latter property cannot hold.
To the contrary, assume that $cf+d$ is everywhere negative on $\Theta$, i.e., $cf(\Theta)+d\subseteq\RR_{--}$.
Then $c\cdot\conv(f(\Theta))+d=\conv(cf(\Theta)+d)\subseteq\RR_{--}$. Using that $F(X)\subseteq \conv(f(\Theta))$, this implies that
 \Eq{*}{
  c\cdot F(X)+d\subseteq c\cdot\conv(f(\Theta))+d\subseteq\RR_{--},
 }
 which contradicts the positivity of $cF+d$ on $X$.
Consequently, $cf+d$ must be everywhere positive on $\Theta$.

To prove the equality $\vartheta_{n,\psi} = \vartheta_{n,\varphi}$ on $X^n$, let $n\in\NN$ and $\bx=(x_1,\dots,x_n)\in X^n$ be arbitrary.
Then, by \eqref{help_Thm_Baj_type_equality_2}, we have
\Eq{*}{
  \sum_{i=1}^n\varphi(x_i,t)
  =\frac{ad-bc}{cf(t)+d}\sum_{i=1}^n\psi(x_i,t),\qquad t\in\Theta.
}
Since $(ad-bc)/(cf+d)$ is positive everywhere on $\Theta$.
This implies that
\Eq{*}{
  \sign\bigg(\sum_{i=1}^n\varphi(x_i,t)\bigg)
  =\sign\bigg(\sum_{i=1}^n\psi(x_i,t)\bigg),
  \qquad t\in\Theta.
}
Hence the unique points of sign change of the functions
 \[
  \Theta\ni t\mapsto \sum_{i=1}^n\varphi(x_i,t)
   \qquad \text{and} \qquad
   \Theta\ni t\mapsto \sum_{i=1}^n\psi(x_i,t)
 \]
 are equal to each other, which implies the equality $\vartheta_{n,\psi}(\bx) = \vartheta_{n,\varphi}(\bx)$, as desired.
\end{proof}

Next, we give an equivalent form of the first equality in \eqref{gGq}.
Roughly speaking, we derive a necessary and sufficient condition in order that two strictly increasing functions defined
 on a nondegenerate open interval be the M\"obius transforms of each other.

\begin{Pro}\label{Pro_Mobius}
Let $I$ be a nondegenerate open interval of $\RR$.
Let $f,g:I\to\RR$ be strictly increasing functions.
The following two statements are equivalent:
 \begin{itemize}
   \item[(i)] There exist four constants $a,b,c,d\in\RR$ with $0\notin c f(I)+d$ and
         \begin{align}\label{help_27}
           g(t)=\frac{af(t)+b}{cf(t)+d},\qquad t\in I.
         \end{align}
   \item[(ii)] For all $t_1,t_2,t_3,t_4\in I$, we have
          \Eq{help_28}{
             \begin{vmatrix}
                1 & 1 & 1 & 1\\
                f(t_1) & f(t_2) & f(t_3) & f(t_4) \\
                g(t_1) & g(t_2) & g(t_3) & g(t_4) \\
                f(t_1)g(t_1) & f(t_2)g(t_2) & f(t_3)g(t_3) & f(t_4)g(t_4) \\
             \end{vmatrix}=0.}
 \end{itemize}
\end{Pro}

\begin{proof}
(i)$\Longrightarrow$(ii).
Let us suppose that there exist four constants $a,b,c,d\in\RR$ such that $0\notin c f(I)+d$ and \eqref{help_27} hold.
By Lemma \ref{Lem_aux}, we have $ad>bc$.
Further, $(cf(t)+d)g(t) = af(t)+b$, $t\in I$, yielding that
 \[
   1\cdot b + f(t)\cdot a + g(t)\cdot (-d) + f(t)g(t)\cdot (-c) =0,\qquad t\in I.
 \]
In particular, for all $t_1,t_2,t_3,t_4\in I$, we have that
 \[
 \begin{bmatrix}
   1 & 1 & 1 & 1 \\
   f(t_1) & f(t_2) & f(t_3) & f(t_4) \\
   g(t_1) & g(t_2) & g(t_3) & g(t_4) \\
   f(t_1)g(t_1) & f(t_2)g(t_2) & f(t_3)g(t_3) & f(t_4)g(t_4) \\
 \end{bmatrix}^\top
 \cdot\begin{bmatrix}
   b \\
   a \\
   -d \\
   -c \\
 \end{bmatrix}
 =
 \begin{bmatrix}
   0 \\
   0 \\
   0 \\
   0 \\
 \end{bmatrix}.
 \]
As a consequence of the inequality $ad>bc$, we have that $(b,a,-d,-c)\ne(0,0,0,0)$, which shows that $(b,a,-d,-c)$ is a nontrivial solution to the above homogeneous system of linear equations. Hence we obtain that \eqref{help_28} must hold for all $t_1,t_2,t_3,t_4\in I$.

(ii)$\Longrightarrow$(i). Let $t_3<t_4$ be fixed elements of $I$. By the strict monotonicity of $f$, the vectors $(1,f(t_3))$ and $(1,f(t_4))$ are linearly independent. Assume first that, for all $t\in I$,
\Eq{help_28.5}{
   \begin{vmatrix}
    1 & 1 & 1\\
    f(t) & f(t_3) & f(t_4) \\
    g(t) & g(t_3) & g(t_4)
    \end{vmatrix}=0
}
holds.
Expanding the determinant along its first column, for all $t\in I$, we get
\Eq{*}{
   b+af(t)-dg(t)=0
}
where
\Eq{*}{
     b:=\begin{vmatrix}
    f(t_3) & f(t_4) \\
    g(t_3) & g(t_4)
    \end{vmatrix},\qquad
    a:=-\begin{vmatrix}
    1 & 1\\
    g(t_3) & g(t_4)
    \end{vmatrix},\qquad
    d:=-\begin{vmatrix}
    1 & 1\\
    f(t_3) & f(t_4)
    \end{vmatrix}\neq0.
}
Therefore, \eqref{help_27} holds with $c=0$ and we also have that $0\not\in cf(I)+d=\{d\}$.

Now we consider the case when \eqref{help_28.5} is not valid for all $t\in I$, that is, there exists $t_2\in I$ such that \eqref{help_28.5} does not hold for $t=t_2$.
Then, by \eqref{help_28}, for all $t\in I$,
\Eq{}{
   \begin{vmatrix}
    1 & 1 & 1 & 1\\
    f(t) & f(t_2) & f(t_3) & f(t_4) \\
    g(t) & g(t_2) & g(t_3) & g(t_4) \\
    f(t)g(t) & f(t_2)g(t_2) & f(t_3)g(t_3) & f(t_4)g(t_4) \\
    \end{vmatrix}=0.
}
Expanding the determinant along its first column, for all $t\in I$, we get
\Eq{abcd}{
   b+af(t)-dg(t)-cf(t)g(t)=0,
}
where
\Eq{*}{
  b&:=\begin{vmatrix}
    f(t_2) & f(t_3) & f(t_4) \\
    g(t_2) & g(t_3) & g(t_4) \\
    f(t_2)g(t_2) & f(t_3)g(t_3) & f(t_4)g(t_4)
    \end{vmatrix},\qquad
  &a&:=-\begin{vmatrix}
    1 & 1 & 1\\
    g(t_2) & g(t_3) & g(t_4) \\
    f(t_2)g(t_2) & f(t_3)g(t_3) & f(t_4)g(t_4)
    \end{vmatrix},\\
  d&:=-\begin{vmatrix}
    1 & 1 & 1\\
    f(t_2) & f(t_3) & f(t_4) \\
    f(t_2)g(t_2) & f(t_3)g(t_3) & f(t_4)g(t_4)
    \end{vmatrix} ,\qquad
  &c&:=\begin{vmatrix}
    1 & 1 & 1\\
    f(t_2) & f(t_3) & f(t_4) \\
    g(t_2) & g(t_3) & g(t_4)
    \end{vmatrix}\neq0.
}
Since $c\neq0$, we have that $cf+d$ is strictly monotone.
We now prove that $cf+d$ does not vanish on $I$. Assume, on the contrary, that for some $t_1\in I$, we have that $cf(t_1)+d=0$.
Then, $cf(t)+d\neq0$ for $t\in I\setminus\{t_0\}$, and, by \eq{abcd}, we get that $af(t_1)+b=0$.
This implies that $ad=bc$.
Therefore, applying \eqref{abcd} for $t\in I\setminus\{t_1\}$, we obtain
\Eq{*}{
  g(t)=\frac{af(t)+b}{cf(t)+d}
  =\frac{af(t)+\frac{ad}{c}}{cf(t)+d}=\frac{a}{c},
}
which contradicts the strict monotonicity of $g$.
\end{proof}

As a consequence of Theorem \ref{Thm_Baj_type_equality}, we can characterize the equality of quasiarithmetic-type $\psi$-estimators.

\begin{Cor}\label{Pro_qa_type_equality}
Let $X$ be a nonempty set, $\Theta$ be a nondegenerate open interval of $\RR$.
Let $f,g:\Theta\to\RR$ be strictly increasing functions, $F:X\to \conv(f(\Theta))$, and $G:X\to \conv(g(\Theta))$.
Let $\psi:X\times \Theta\to\RR$ and $\varphi: X\times \Theta\to\RR$ be given by
\[
  \psi(x,t):=F(x)-f(t),\qquad \varphi(x,t):=G(x)-g(t),
  \qquad x\in X,\,t\in\Theta.
\]
The following two assertions hold:
\begin{itemize}
 \item[(i)] If there exist two constants $a,b\in\RR$ with $a\ne 0$  such that
        \begin{align}\label{help_26}
          g(t)=af(t)+b,\quad t\in\Theta,\qquad
             \mbox{and}\qquad
         G(x)=aF(x)+b,\quad x\in X,
       \end{align}
       then $\vartheta_{n,\psi}(\bx) = \vartheta_{n,\varphi}(\bx)$ holds for each $n\in\NN$ and $\bx\in X^n$.
 \item[(ii)] In addition, suppose that $f$ is continuous, $F(X)=f(\Theta)$ and $\conv(G(X)) \subseteq g(\Theta)$.
             If $\vartheta_{n,\psi}(\bx) = \vartheta_{n,\varphi}(\bx)$ holds for each $n\in\NN$ and $\bx\in X^n$,
             then there exist two constants $a,b\in\RR$ with $a\ne 0$  such that \eqref{help_26} holds.
\end{itemize}
\end{Cor}

\begin{proof}
(i). Let us assume that there exist two constants $a,b\in\RR$ with $a\ne 0$ such that \eqref{help_26} holds.
By choosing $c:=0$, $d:=1$ and $p(x):=q(x):=1$, $x\in X$, we have
 $cf(\Theta)+d = 1$, and hence $0\notin cf(\Theta)+d$.
Further, \eqref{gGq} is satisfied as well.
Consequently, Theorem \ref{Thm_Baj_type_equality_2} yields that
 $\vartheta_{n,\psi}(\bx) = \vartheta_{n,\varphi}(\bx)$ for each $n\in\NN$ and $\bx\in X^n$.

(ii).
One can apply Theorem \ref{Thm_Baj_type_equality} with the given functions $f,g,F$ and $G$ and by choosing $p(x):=q(x):=1$, $x\in X$.
Then we obtain that there exist four constants $a,b,c,d\in\RR$ with $ad\ne bc$ and $0\notin c f(\Theta)+d$ such that
 \eqref{gGq} holds, i.e.,
 \begin{align*}
    g(t)=\frac{af(t)+b}{cf(t)+d},\,\,\, t\in\Theta,\quad
    \mbox{and}\quad
    G(x)=\frac{aF(x)+b}{cF(x)+d},\quad
    q(x)=(cF(x)+d)p(x),\,\,\, x\in X.
 \end{align*}
Since $p=q=1$, we get $cF(x)+d=1$, $x\in X$, and hence $G(x) = aF(x)+b$, $x\in X$.
Consequently, in order to prove the statement, it is enough to verify that $cf(t)+d=1$, $t\in\Theta$.
We check that $c=0$ and hence $d=1$.
Since $\Theta$ is a nondegenerate open interval of $\RR$ and $f$ is strictly increasing and continuous,
 we have that $f(\Theta)$ is a nondegenerate interval of $\RR$.
Hence, using that $F(X)=f(\Theta)$, the range $F(X)$ of $F$ contains at least two distinct elements,
 and consequently, there exist $x_1,x_2\in X$ such that $F(x_1)\ne F(x_2)$.
Since $cF(x_1)+d=1$ and $cF(x_2)+d=1$, we have $c(F(x_1) - F(x_2))=0$,
 yielding that $c=0$, as desired.
\end{proof}

\section{Examples}
\label{Sec_stat_examples}

In this section, we give some applications of Theorem \ref{Lem_psi_est_eq_5} for solving the comparison problem
 for some statistical estimators that are special cases of generalized $\psi$-estimators.
We emphasize that Theorem \ref{Lem_psi_est_eq_5} makes one possible to compare two $\psi$-estimators even if these estimators cannot be computed explicitly,
 only they can be numerically approximated.
We consider empirical expectiles, Mathieu-type estimators and solutions of likelihood equations in case of normal, a Beta-type, Gamma, Lomax (Pareto type II), lognormal and Laplace distributions.

\begin{Ex}[Empirical expectiles]
Let $\alpha, \beta\in(0,1)$, $X:=\Theta:=\RR$, and $\psi, \varphi:\RR\times\RR\to\RR$ given by
 \[
   \psi(x,t):=\begin{cases}
                \alpha(x-t) & \text{if $x>t$,}\\
                0 & \text{if $x=t$,}\\
                (1-\alpha)(x-t) & \text{if $x<t$,}
               \end{cases}
  \qquad\text{and}\qquad
   \varphi(x,t):=\begin{cases}
                 \beta(x-t) & \text{if $x>t$,}\\
                 0 & \text{if $x=t$,}\\
                 (1-\beta)(x-t) & \text{if $x<t$.}
                \end{cases}
 \]
By Example 4.3 in Barczy and P\'ales \cite{BarPal2}, $\psi$ and $\varphi$ are $Z$-functions,
 $\vartheta_{1,\psi}(x) = \vartheta_{1,\varphi}(x) = x$, $x\in\RR$, and
 hence we readily get that $\Theta_\psi = \Theta_\varphi = \Theta=\RR$.
For each $n\in\NN$ and $\bx\in\RR^n$, $\vartheta_{n,\psi}(\bx)$ and $\vartheta_{n,\varphi}(\bx)$ can be considered as empirical $\alpha$- and $\beta$-expectiles based on $\bx$, respectively, for more details, see Barczy and P\'ales \cite[Example 4.3]{BarPal2}.
Using Theorem \ref{Lem_psi_est_eq_5}, we have that $\vartheta_{n,\psi}(\bx)\leq \vartheta_{n,\varphi}(\bx)$ holds
 for each $n\in\NN$ and $\bx\in \RR^n$ if and only if
 \begin{align*}
  \psi(x,t) \varphi(y,t) \leq \psi(y,t) \varphi(x,t)
 \end{align*}
 is valid for all $t\in\RR$ and for all $x,y\in\RR$ with $\vartheta_{1,\varphi}(x)<t<\vartheta_{1,\varphi}(y)$, or equivalently,
 \[
   (1-\alpha)(x-t)\beta(y-t) \leq \alpha(y-t)(1-\beta)(x-t)
 \]
 is valid for $x,y,t\in\RR$ with $x<t<y$.
Hence, $\vartheta_{n,\psi}(\bx)\leq \vartheta_{n,\varphi}(\bx)$ holds for all $n\in\NN$ and $\bx\in \RR^n$ if and only if $\alpha \leq \beta$.
\proofend
\end{Ex}

\begin{Ex}[Mathieu-type estimators]
Let $X:=\RR$, $\Theta:=\RR$ and $\psi, \varphi:\RR\times\RR\to\RR$ be given by
 \begin{align}\label{help_Mathieu}
   \psi(x,t):=\sign(x-t)f(\vert x-t\vert)
    \quad \text{and}\quad
   \varphi(x,t):=\sign(x-t)g(\vert x-t\vert),\qquad x,t\in\RR,
 \end{align}
 where $f,g:\RR_+\to\RR$ are continuous and strictly increasing functions with $f(0)=g(0)=0$.
This particular form of $\psi$  has been recently investigated by Mathieu \cite{Mat}.
By Proposition 4.4 in Barczy and P\'ales \cite{BarPal2}, $\psi$ and $\varphi$ have the property $[Z]$
 and $\vartheta_{1,\psi}(x) = \vartheta_{1,\varphi}(x) = x$, $x\in\RR$, and
 hence we readily get that $\Theta_\psi = \Theta_\varphi = \Theta=\RR$.
Using Theorem \ref{Lem_psi_est_eq_5}, we have that $\vartheta_{n,\psi}(\bx)\leq \vartheta_{n,\varphi}(\bx)$ holds
 for each $n\in\NN$ and $\bx\in \RR^n$ if and only if
 \begin{align*}
  \psi(x,t) \varphi(y,t) \leq \psi(y,t) \varphi(x,t)
 \end{align*}
 is valid for all $t\in\RR$ and for all $x,y\in\RR$ with $\vartheta_{1,\varphi}(x)<t<\vartheta_{1,\varphi}(y)$,
 or equivalently,
 \[
  \sign(x-t)f(\vert x-t\vert) \sign(y-t)g(\vert y-t\vert)
       \leq \sign(y-t)f(\vert y-t\vert) \sign(x-t)g(\vert x-t\vert)
 \]
 is valid for all $x,y,t\in\RR$ with $x<t<y$.
It is equivalent to
 \[
    f(y-t) g(t-x) \leq f(t-x) g(y-t) \qquad \text{ for all $x,y,t\in\RR$ with $x<t<y$.}
 \]
Using that $f$ and $g$ are strictly increasing with $f(0)=g(0)=0$, this is equivalent to
 \begin{align}\label{help_21}
   \frac{g(t-x)}{f(t-x)}\leq \frac{g(y-t)}{f(y-t)} \qquad \text{ for all $x,y,t\in\RR$ with $x<t<y$.}
 \end{align}
Given $u,v\in\RR_{++}$, by choosing $x:=t-u$ and $y:=t+v$ with $t\in\RR$,
 we have that \eqref{help_21} is equivalent to
 \[
    \frac{g(u)}{f(u)} \leq \frac{g(v)}{f(v)}\qquad \text{for all $u,v\in\RR_{++}$,}
 \]
 which is equivalent to the existence of a constant $C\in\RR_{++}$ such that $\frac{g(u)}{f(u)}=C$, $u\in\RR_{++}$.
Consequently, using $f(0)=g(0)=0$, we have that $\vartheta_{n,\psi}(\bx)\leq \vartheta_{n,\varphi}(\bx)$ holds
 for each $n\in\NN$ and $\bx\in \RR^n$ if and only if there exists a constant $C\in\RR_{++}$ such that $g(t)=C f(t)$, $t\in\RR_+$.
Based on what we proved, we can formulate the following result:

{\bf Proposition.}
{\it Let $f,g:\RR_+\to\RR$ be continuous and strictly increasing functions with $f(0)=g(0)=0$,
 and let $\psi, \varphi:\RR\times\RR\to\RR$ be given by \eqref{help_Mathieu}.
Then the following assertions are equivalent:
\vspace{-5mm}
 \begin{itemize}
   \item[(i)] $\vartheta_{n,\psi}(\bx)\leq \vartheta_{n,\varphi}(\bx)$ holds for each $n\in\NN$ and $\bx\in \RR^n$,
   \item[(ii)] $\vartheta_{n,\psi}(\bx)=\vartheta_{n,\varphi}(\bx)$ holds for each $n\in\NN$ and $\bx\in \RR^n$,
   \item[(iii)] there exists a constant $C\in\RR_{++}$ such that $g(t)=C f(t)$, $t\in\RR_+$.
 \end{itemize}
\vspace{-6mm}\proofend}
\end{Ex}

In what follows, given a random variable $\xi$, by a sample of size $n$ (where $n\in\NN$), we mean independent and identically distributed random variables \ $\xi_1,\ldots,\xi_n$ \ with common distribution as that of $\xi$.

\begin{Ex}[Normal distribution]
Let $\xi$ be a normally distributed random variable with mean $m\in\RR$ and with variance $\sigma^2$, where $\sigma>0$.
Let $n\in\NN$ and $x_1,\ldots,x_n\in\RR$ be a realization of a sample of size $n$ for $\xi$.
Supposing that $m$ is known, there exists a unique MLE of $\sigma^2$ based on $x_1,\ldots,x_n\in\RR$,
 and it takes the form \ $\widehat{\sigma^2_n}:=\frac{1}{n}\sum_{i=1}^n (x_i-m)^2$.
In this case, let $\Theta:=(0,\infty)$, $X:=\RR\setminus\{m\}$, and $\psi:(\RR\setminus\{m\})\times(0,\infty)\to\RR$,
 \begin{align}\label{help_norm1}
    \psi(x,\sigma^2)
     := \frac{1}{2(\sigma^2)^2}\left( (x-m)^2 - \sigma^2\right) ,
         \qquad x\in\RR\setminus\{m\},\; \sigma^2>0.
 \end{align}
Hence $\psi$ is a $T_1$-function with $\vartheta_{1,\psi}(x):=(x-m)^2$, $x\in\RR\setminus\{m\}$, and $\psi(x,\vartheta_{1,\psi}(x))=0$, $x\in\RR\setminus\{m\}$.
Further, we readily get that $\Theta_\psi = \Theta=(0,\infty)$.
By Example 4.5 in Barczy and P\'ales \cite{BarPal2}, we have that for each $n\in\NN$ and
 $x_1,\ldots,x_n\in\RR\setminus\{m\}$, the (likelihood) equation $\sum_{i=1}^n \psi(x_i,t)=0$, $t>0$, has a unique solution,
 which is equal to $\vartheta_{n,\psi}(x_1,\dots,x_n)=\frac{1}{n}\sum_{i=1}^n (x_i-m)^2=\widehat{\sigma^2_n}$.
Consequently, $\psi$ is a $Z$-function.

Further, let $m^*\in\RR$, and $\varphi:(\RR\setminus\{m^*\})\times(0,\infty)\to\RR$,
 \begin{align}\label{help_norm2}
    \varphi(x,\sigma^2)
     := \frac{1}{2(\sigma^2)^2}\left( (x-m^*)^2 - \sigma^2\right) ,
         \qquad x\in\RR\setminus\{m^*\},\; \sigma^2>0.
 \end{align}
Then $\varphi$ is a $Z$-function with $\vartheta_{n,\varphi}(x_1,\dots,x_n)=\frac{1}{n}\sum_{i=1}^n (x_i-m^*)^2$ for $n\in\NN$ and $x_1,\ldots,x_n\in\RR\setminus\{m^*\}$,
  and $\Theta_\varphi = (0,\infty)$.

Next, we compare the $\psi$-estimators corresponding to the functions given in \eqref{help_norm1} and \eqref{help_norm2}.
Then $\vartheta_{1,\psi}(x_1)\leq \vartheta_{1,\varphi}(x_1)$ holds for all $x_1\in\RR\setminus\{m,m^*\}$
if and only if $(x_1-m)^2 \leq (x_1-m^*)^2$ holds for all $x_1\in\RR\setminus\{m,m^*\}$, which is equivalent to $m=m^*$
(indeed, the limit of the right hand side as $x_1$ tends to $m^*$ is zero implying that $(m^*-m)^2=0$ holds).
Consequently, we readily have that $\vartheta_{n,\psi}(\bx)\leq \vartheta_{n,\varphi}(\bx)$ holds for each $n\in\NN$ and $\bx\in (\RR\setminus\{m,m^*\})^n$ if and only if $m=m^*$, and in this case $\vartheta_{n,\psi}(\bx) = \vartheta_{n,\varphi}(\bx)$ holds for each $n\in\NN$ and $\bx\in (\RR\setminus\{m,m^*\})^n$.
Note also that $\psi(x,\sigma^2)\leq \varphi(x,\sigma^2)$ holds for all $x\in\RR\setminus\{m,m^*\}$
 and $\sigma^2>0$ if and only if $(x-m)^2 \leq (x-m^*)^2$ holds for all $x\in\RR\setminus\{m,m^*\}$,
 which is equivalent to $m=m^*$, as we have already seen.
This is in accordance with part (iv) of Theorem \ref{Lem_psi_est_eq_5} (applied with $X:=\RR\setminus\{m,m^*\}$).
\proofend
\end{Ex}

\begin{Ex}[Beta-type distribution]\label{Ex_Bet_type}
Let $\alpha,\beta\in\RR_{++}$ and let $\xi$ be an absolutely continuous random variable with a density function
 \[
     f_\xi(x) := \begin{cases}
                  \alpha \beta x^{\beta-1} (1-x^\beta)^{\alpha-1} & \text{if $x\in(0,1)$,}\\
                  0 & \text{if $x\notin(0,1)$.}
               \end{cases}
 \]

Supposing that $\beta\in\RR_{++}$ is known, one can check that, given $n\in\NN$ and a realization $x_1,\ldots,x_n\in(0,1)$
 of a sample of size $n$ for $\xi$, there exists a unique MLE of $\alpha$ and it takes the form
 \[
  \halpha_n := -\frac{n}{\sum_{i=1}^n \ln(1-x_i^\beta)}.
 \]
As an application of our results in Barczy and P\'ales \cite{BarPal2}, we also establish the existence and uniqueness of a solution of the corresponding likelihood equation for $\alpha$ using Theorem 2.10 and Proposition 2.12 in Barczy and P\'ales \cite{BarPal2}.
In the considered case, using the setup given before Example 4.6 in Barczy and P\'ales \cite{BarPal2}, we have
 \ $\Theta := (0,\infty)$ \ and \ $f:\RR\times (0,\infty)\to\RR$,
 \[
   f(x,\alpha):=\begin{cases}
                  \alpha \beta x^{\beta-1} (1-x^\beta)^{\alpha-1} & \text{if $x\in(0,1)$, \ $\alpha>0$,}\\
                  0 & \text{otherwise,}
               \end{cases}
 \]
 and consequently, \ $\cX_f=(0,1)$.
\ Then \ $\psi:(0,1)\times (0,\infty)\to\RR$,
 \begin{align}\label{beta_1}
   \psi(x,\alpha) = \frac{\partial_2 f(x,\alpha)}{f(x,\alpha)}
                  = \frac{\beta x^{\beta-1} \big( (1-x^\beta)^{\alpha-1} + \alpha (1-x^\beta)^{\alpha-1}\ln(1-x^\beta) \big)}
                         {\alpha \beta x^{\beta-1} (1-x^\beta)^{\alpha-1}}
                  = \frac{1}{\alpha} + \ln(1-x^\beta)
 \end{align}
for $x\in(0,1)$ and $\alpha>0$.
We have that $\psi$ is a $Z_1$-function with $\vartheta_{1,\psi}(x) = - \frac{1}{\ln(1-x^\beta)}$, $x\in(0,1)$, $\Theta_\psi=\Theta=(0,\infty)$, and $\psi$ is strictly decreasing and continuous in its second variable.
Further, using Proposition 2.12 in Barczy and P\'ales \cite{BarPal2} (with $X:=\cX_f = (0,1)$),
 we can conclude that $\psi$ is a $Z$-function, and, for each $n\in\NN$ and $x_1,\ldots,x_n\in(0,1)$, the (likelihood) equation
 $\sum_{i=1}^n \psi(x_i,\alpha)=0$, $\alpha>0$, has a unique solution, which takes the form
 $\vartheta_{n,\psi}(x_1,\ldots,x_n) = -\frac{n}{\sum_{i=1}^n \ln\left(1-x_i^\beta\right)} = \halpha_n$, as desired.

\noindent Moreover, let $\beta^*\in\RR_{++}$ be also given and let $\varphi:(0,1)\times (0,\infty)\to\RR$ be defined by
 \begin{align}\label{beta_2}
   \varphi(x,\alpha) := \frac{1}{\alpha} + \ln(1-x^{\beta^*}), \qquad  x\in(0,1),\quad \alpha>0.
 \end{align}
Then $\varphi$ is a $Z$-function with $\vartheta_{n,\varphi}(x_1,\ldots,x_n) = -\frac{n}{\sum_{i=1}^n \ln\left(1-x_i^{\beta^*}\right)}$ for $n\in\NN$ and $x_1,\ldots,x_n\in(0,1)$, and we have $\Theta_\varphi=(0,\infty)$.

\noindent Next, we compare the $\psi$-estimators corresponding to the functions given in \eqref{beta_1} and \eqref{beta_2}.
By a direct calculation, one check that $\vartheta_{1,\psi}(x)\leq \vartheta_{1,\varphi}(x)$ holds for all $x\in(0,1)$
 if and only if $\beta\leq \beta^*$; and, more generally,
 $\vartheta_{n,\psi}(x_1,\ldots,x_n) \leq \vartheta_{n,\varphi}(x_1,\ldots,x_n)$ holds for each $n\in\NN$ and $x_1,\ldots,x_n\in(0,1)$ if and only if $\beta\leq \beta^*$.
As a consequence, the equivalent assertions of (i)-(iv) of Theorem \ref{Lem_psi_est_eq_5} hold if and only if $\beta\leq \beta^*$.
In particular, note that if $\beta\leq \beta^*$, then
 \[
   \psi(x,\alpha) = \frac{1}{\alpha} + \ln(1-x^\beta) \leq \frac{1}{\alpha} + \ln(1-x^{\beta^*}) = \varphi(x,\alpha),
   \qquad x\in(0,1),\quad \alpha>0,
 \]
 i.e.,  \eqref{psi_est_inequality_5} holds with $p(\alpha):=1$, $\alpha>0$.

Now, suppose that $\alpha\in\RR_{++}$ is known.
Similarly as above, we have \ $\Theta := (0,\infty)$ \ and \ $f:\RR\times (0,\infty)\to\RR$,
 \[
   f(x,\beta):=\begin{cases}
                  \alpha \beta x^{\beta-1} (1-x^\beta)^{\alpha-1} & \text{if $x\in(0,1)$, \ $\beta>0$,}\\
                  0 & \text{otherwise,}
               \end{cases}
 \]
 and consequently, \ $\cX_f=(0,1)$.
\ Then \ $\psi:(0,1)\times (0,\infty)\to\RR$,
 \begin{align*}
   \psi(x,\beta) &= \frac{\partial_2 f(x,\beta)}{f(x,\beta)} = \partial_2 (\ln f)(x,\beta)
                  = \partial_2\big(\ln(\alpha) + \ln(\beta) + (\beta-1)\ln(x) + (\alpha-1)\ln(1-x^\beta)\big)\\
                 &= \frac{1}{\beta} + \ln(x) + (1-\alpha)\frac{x^\beta \ln(x)}{1-x^\beta},
                  \qquad  x\in(0,1),\quad \beta>0.
   \end{align*}
Note that $\psi$ can be rewritten in the form
 \begin{align}\label{help_29}
   \psi(x,\beta) = \frac{1}{\beta}\left(1+ \frac{1-\alpha x^\beta}{1-x^\beta} \ln(x^\beta)\right),
                   \qquad  x\in(0,1),\quad \beta>0.
 \end{align}
Define \ $\widetilde{\psi}:(0,1)\times (0,\infty)\to\RR$,
 $\widetilde{\psi}(x,\beta):=\beta\psi(x,\beta)$, $x\in(0,1)$, $\beta>0$.
Next, we check that $\widetilde{\psi}$ is strictly decreasing in its second variable.
Since for all $x\in(0,1)$, the range of the strictly decreasing function $(0,\infty)\ni\beta\mapsto x^\beta$ is the interval $(0,1)$,
 in order to prove that $\widetilde{\psi}$ is strictly decreasing in its second variable, it is enough to verify that the function
 \begin{align}\label{help_31}
 (0,1)\ni u\mapsto \frac{1-\alpha u}{1-u}\ln(u)
 \end{align}
 is strictly increasing.
For this, it is enough to show that its derivative is positive everywhere, i.e.,
 \[
  -\alpha\frac{\ln(u)}{1-u} + \frac{1-\alpha u}{(1-u)^2}\ln(u) + \frac{1-\alpha u}{1-u}\cdot\frac{1}{u}>0, \qquad u\in(0,1),
 \]
 which is equivalent to
 \[
 h(u):=(1-\alpha)\ln(u) + \frac{1}{u} - \alpha -1 +\alpha u>0,\qquad u\in(0,1).
 \]
Since $h$ is continuous, $\lim_{u\downarrow 0} h(u)=\infty$ and $\lim_{u\uparrow 1}h(u)=0$,
 in order to show that $h(u)>0$, $u\in(0,1)$, it is enough to check that $h'(u)<0$, $u\in(0,1)$.
This readily follows, since
 \[
   h'(u) = \frac{1-\alpha}{u} - \frac{1}{u^2} + \alpha <0, \qquad u\in(0,1),
 \]
 holds if and only if $(u-1)(\alpha u + 1)<0$, $u\in(0,1)$, which is trivially satisfied.

\noindent  Now we verify that $\widetilde{\psi}$ is a $Z$-function.
First, we check that $\widetilde{\psi}$ is a $T_1$-function.
Using again that, for all $x\in(0,1)$, the range of the strictly decreasing function $(0,\infty)\ni\beta\mapsto x^\beta$ is $(0,1)$, taking into account \eqref{help_29}, it is enough to verify that there exists a unique solution of the equation
 \[
    g(u):=1 + \frac{1-\alpha u}{1-u}\ln(u) = 0, \qquad u\in(0,1).
 \]
Since $g$ is continuous and strictly increasing (proved earlier, see \eqref{help_31}),
 $\lim_{u\downarrow 0} g(u) = -\infty$ and $\lim_{u\uparrow 1}g(u)=1+(1-\alpha)\lim_{u\uparrow 1}\frac{1/u}{-1} = \alpha >0$,
the Bolzano theorem yields the existence of a unique solution in question. Hence $\widetilde{\psi}$ is a $T_1$-function.
This, together with the fact that $\widetilde{\psi}$ is strictly decreasing in its second variable, implies that $\widetilde{\psi}$ is a $T$-function (see Barczy and P\'ales \cite[Proposition 2.12]{BarPal2}).
Since $\widetilde{\psi}$ is continuous in its second variable as well, we have that it is a Z-function.
It immediately follows that $\psi$ is also a Z-function.
As a consequence, for each $n\in\NN$ and $x_1,\ldots,x_n\in(0,1)$, there is a unique solution
 $\vartheta_{n,\psi}(x_1,\ldots,x_n)$ of the (likelihood)
equation
 $\sum_{i=1}^n \psi(x_i,\beta)=0$, $\beta>0$, that is, of the equation
 \begin{align}\label{help_32}
  1+\ln \left( \sqrt[n]{x_1^\beta\cdots x_n^\beta}\right)
      = (\alpha-1)\frac{1}{n}\sum_{i=1}^n \frac{x_i^\beta\ln(x_i^\beta)}{1-x_i^\beta},\qquad \beta>0.
 \end{align}

In addition, let $\alpha^*\in\RR_{++}$ be also given and let \ $\varphi:(0,1)\times (0,\infty)\to\RR$ be defined by
 \begin{align}\label{beta_3}
   \varphi(x,\beta) := \frac{1}{\beta}\left(1+ \frac{1-\alpha^* x^\beta}{1-x^\beta} \ln(x^\beta)\right),
                   \qquad  x\in(0,1),\quad \beta>0.
 \end{align}
Then, as we have seen, $\varphi$ is a $Z$-function.

\noindent Next, we compare the $\psi$-estimators corresponding to the functions given in \eqref{help_29} and \eqref{beta_3}.
By a direct computation, one can readily see that $\psi(x,\beta)\leq \varphi(x,\beta)$ holds for $x\in(0,1)$ and $\beta>0$ if and only if $\alpha\leq \alpha^*$.
Further, since $\psi$ and $\varphi$ are $Z$-functions, by part (ii) of Remark \ref{Rem_Thm_main},
 we have that if $\psi(x,\beta)\leq \varphi(x,\beta)$, $x\in(0,1)$, $\beta>0$, then
 $\vartheta_{1,\psi}(x)\leq \vartheta_{1,\varphi}(x)$, $x\in(0,1)$.
Consequently, Theorem \ref{Lem_psi_est_eq_5} yields that $\vartheta_{n,\psi}(x_1,\ldots,x_n)\leq \vartheta_{n,\varphi}(x_1,\ldots,x_n)$ holds
 for all $n\in\NN$ and $x_1,\ldots,x_n\in(0,1)$ if and only if $\alpha\leq \alpha^*$.
\proofend
\end{Ex}

The next proposition highlights another application of Theorem \ref{Lem_psi_est_eq_5}.
Namely, given $\alpha\in\RR_{++}$, $n\in\NN$ and $x_1,\ldots,x_n\in(0,1)$, one can approximate the unique solution
 of the equation \eqref{help_32} in terms of the geometric mean of $x_1,\ldots,x_n$.

\begin{Pro}\label{Pro_Beta_approx}
For all $\alpha\in\RR_{++}$, $n\in\NN$ and $x_1,\ldots,x_n\in(0,1)$, let $r(\alpha,n,x_1,\ldots,x_n)$ denote
 the unique solution of the equation \eqref{help_32}.
Then
 \[
   -\frac{\min(\alpha,1)}{\ln\big(\sqrt[n]{x_1\cdots x_n}\big)}
      \leq r(\alpha,n,x_1,\ldots,x_n)
      \leq
   -\frac{\max(\alpha,1)}{\ln\big(\sqrt[n]{x_1\cdots x_n}\big)}.
 \]
In particular, $r(\alpha,n,x_1,\ldots,x_n)$ tends to $-n/\ln(x_1\cdots x_n)$ as $\alpha\to 1$ for all $n\in\NN$ and $x_1,\ldots,x_n\in(0,1)$.
\end{Pro}

\begin{proof}
Note that $r(\alpha,n,x_1,\ldots,x_n)$ is nothing else but $\vartheta_{n,\psi}(x_1,\ldots,x_n)$
 (in the notation $r(\alpha,n,x_1,\ldots,x_n)$, the dependence of $\vartheta_{n,\psi}(x_1,\ldots,x_n)$
 on the parameter $\alpha$ is displayed as well).
For all $\alpha\in\RR_{++}$, $n\in\NN$ and $x_1,\ldots,x_n\in(0,1)$, the unique existence of
 $r(\alpha,n,x_1,\ldots,x_n)$ follows from the discussion in Example \ref{Ex_Bet_type}.

First,  we show that for $\alpha,\beta \in \RR_{++}$ and $x\in(0,1)$, we have
\Eq{ineq_beta}{
   \frac{\min(\alpha,1)}{\beta}+\ln(x)
   \leq
   \frac{1}{\beta}\left(1+ \frac{1-\alpha x^\beta}{1-x^\beta} \ln(x^\beta)\right)
  \leq \frac{\max(\alpha,1)}{\beta}+\ln(x).
}
If $\alpha\in(0,1]$, then $1-\alpha x^\beta\geq 1-x^\beta>0$, therefore,
\Eq{*}{
  \frac{1}{\beta}\left(1+ \frac{1-\alpha x^\beta}{1-x^\beta} \ln(x^\beta)\right)
  \leq \frac{1}{\beta}\left(1+\ln(x^\beta) \right)
  =\frac{1}{\beta}+\ln(x),
}
which gives the right hand side inequality in case of $\alpha\in(0,1]$. If $\alpha>1$, then $1-\alpha x^\beta<1-x^\beta$,
 therefore (also using that $1-x^\beta>0$), we get
\Eq{*}{
  \frac{1}{\beta}\left(1+ \frac{1-\alpha x^\beta}{1-x^\beta} \ln(x^\beta)\right)
  \geq \frac{1}{\beta}\left(1+\ln(x^\beta) \right)
  =\frac{1}{\beta}+\ln(x),
}
which implies the left hand side inequality in case of $\alpha>1$.

If $\alpha\in(0,1]$, then left hand side inequality is equivalent to
\Eq{*}{
  \frac{\alpha}{\beta}+\ln(x)
   \leq \frac{1}{\beta}\left(1+ \frac{1-\alpha x^\beta}{1-x^\beta} \ln(x^\beta)\right).
}
This, with the substitution $u:=x^\beta\in(0,1)$ can be rewritten as
\Eq{*}{
  \alpha+\ln(u)
   \leq 1+ \frac{1-\alpha u}{1-u} \ln(u).
}
After some computation, this inequality reads as follows
\Eq{au}{
  0\leq (1-\alpha)\left(\ln(u)-\frac{u-1}{u}\right),
}
which holds according to the well-known inequality $\ln(u)\geq \frac{u-1}u$, $u>0$, in case of $\alpha\in(0,1]$.

If $\alpha>1$, then right hand side inequality is equivalent to
\Eq{*}{
  \frac{1}{\beta}\left(1+ \frac{1-\alpha x^\beta}{1-x^\beta} \ln(x^\beta)\right)
  \leq \frac{\alpha}{\beta}+\ln(x).
}
By a similar argument, this inequality is the consequence of the reverse of inequality \eq{au}, which holds for all $u\in(0,1)$ in case of $\alpha>1$.

As a next step, motivated by \eqref{ineq_beta}, for all $\alpha\in\RR_{++}$, let $\psi,\,\psi_*,\,\psi^*:(0,1)\times (0,\infty)\to\RR$ be defined by
 \begin{align*}
   &\psi(x,\beta):= \frac{1}{\beta}\left(1+ \frac{1-\alpha x^\beta}{1-x^\beta} \ln(x^\beta)\right), \qquad x\in(0,1),\quad \beta>0,\\
   &\psi_*(x,\beta):= \frac{\min(\alpha,1)}{\beta}+\ln(x), \qquad x\in(0,1),\quad \beta>0,\\
   &\psi^*(x,\beta):= \frac{\max(\alpha,1)}{\beta}+\ln(x), \qquad x\in(0,1),\quad \beta>0.
 \end{align*}
In Example \ref{Ex_Bet_type} we have already checked that $\psi$ was a $Z$-function.
Further, since $\psi_*$ and $\psi^*$ are strictly decreasing and continuous in their second variables,
 we get that $\psi_*$ and $\psi^*$ are $Z$-functions as well (following from Barczy and P\'ales \cite[Proposition 2.12]{BarPal2}).
Using these properties, \eqref{ineq_beta} and
  part (ii) of Remark \ref{Rem_Thm_main}, we get that
 \[
  \vartheta_{1,\psi_*}(x)\leq \vartheta_{1,\psi}(x) \leq \vartheta_{1,\psi^*}(x),\qquad x\in(0,1).
 \]
Consequently, by Theorem \ref{Lem_psi_est_eq_5}, we obtain that
 \begin{align}\label{help_beta_ineq}
   \vartheta_{n,\psi_*}(x_1,\ldots,x_n)
     \leq \vartheta_{n,\psi}(x_1,\ldots,x_n)
     \leq \vartheta_{n,\psi^*}(x_1,\ldots,x_n)
 \end{align}
 for all $n\in\NN$ and $x_1,\ldots,x_n\in(0,1)$.
For all $\beta\in\RR_{++}$, $n\in\NN$ and $x_1,\ldots,x_n\in(0,1)$, we have that
 \begin{align*}
   &\sum_{i=1}^n \psi_*(x_i,\beta) = n\cdot\frac{\min(\alpha,1)}{\beta} + \ln(x_1\cdots x_n), \\
   &\sum_{i=1}^n \psi^*(x_i,\beta) = n\cdot\frac{\max(\alpha,1)}{\beta} + \ln(x_1\cdots x_n),
 \end{align*}
 yielding that
 \begin{align*}
   &\vartheta_{n,\psi_*}(x_1,\ldots,x_n) = -\frac{\min(\alpha,1)}{\ln\big(\sqrt[n]{x_1\cdots x_n}\big)},
   &\vartheta_{n,\psi^*}(x_1,\ldots,x_n) = -\frac{\max(\alpha,1)}{\ln\big(\sqrt[n]{x_1\cdots x_n}\big)}
 \end{align*}
 for all $n\in\NN$ and $x_1,\ldots,x_n\in(0,1)$.
These formulae together with \eqref{help_beta_ineq} yields the statement.
\end{proof}

\begin{Ex}[Gamma distribution]
Let $p,\lambda\in\RR_{++}$ and let $\xi$ be a random variable having Gamma distribution with parameters $p$ and $\lambda$, i.e.,
 $\xi$ has a density function
 \[
     f_\xi(x) := \begin{cases}
                  \dfrac{\lambda^p x^{p-1} \ee^{-\lambda x}}{\Gamma(p)} & \text{if $x>0$,}\\
                  0 & \text{if $x\leq 0$.}
               \end{cases}
 \]

Supposing that $\lambda\in\RR_{++}$ is known,
 we will establish the existence and uniqueness of a solution of the likelihood equation for $p$
 using Theorem 2.10 and Proposition 2.12 in Barczy and P\'ales \cite{BarPal2}.
In the considered case, using the setup given before Example 4.6 in Barczy and P\'ales \cite{BarPal2}, we have
 \ $\Theta := (0,\infty)$ \ and \ $f:\RR\times (0,\infty)\to\RR$,
 \[
     f(x,p) :=\begin{cases}
                  \dfrac{\lambda^p x^{p-1} \ee^{-\lambda x}}{\Gamma(p)} & \text{if $x>0$, $p>0$,}\\
                  0 & \text{otherwise,}
                \end{cases}
 \]
 and consequently, \ $\cX_f=(0,\infty)$.
\ Then \ $\psi:(0,\infty)\times (0,\infty)\to\RR$,
 \begin{align}\label{gamma_1}
   \psi(x,p) = \frac{\partial_2 f(x,p)}{f(x,p)}
             = -\frac{\Gamma'(p)}{\Gamma(p)} + \ln(x) + \ln(\lambda),
             \qquad  x,p\in(0,\infty).
 \end{align}
Note that the function $\frac{\Gamma'}{\Gamma}:(0,\infty)\to (0,\infty)$ is the digamma function,
 which is known to be strictly increasing and strictly concave, and it has range $(-\infty,\infty)$.
As a consequence, $\psi$ is a Z-function, and, for each $n\in\NN$ and $x_1,\ldots,x_n\in(0,\infty)$, the equation
 $\sum_{i=1}^n \psi(x_i,p) =0$, $p>0$, which is equivalent to
 \[
  \frac{\Gamma'(p)}{\Gamma(p)} = \ln(\lambda) + \frac{1}{n}\sum_{i=1}^n \ln(x_i),\qquad p>0,
 \]
has a unique solution $\vartheta_{n,\psi}(x_1,\ldots,x_n)$ for $p>0$.

\noindent Moreover, let $\lambda^*\in\RR_{++}$ and $\varphi:(0,\infty)\times (0,\infty)\to\RR$,
 \begin{align}\label{gamma_2}
   \varphi(x,p) := -\frac{\Gamma'(p)}{\Gamma(p)} + \ln(x) + \ln(\lambda^*),
             \qquad  x,p\in(0,\infty).
 \end{align}
Then, as we have seen it, $\varphi$ is a Z-function.

\noindent In what follows, we compare the $\psi$-estimators corresponding to the functions given in \eqref{gamma_1} and \eqref{gamma_2}.
By a direct computation, one can readily see that $\lambda\leq \lambda^*$ holds if and only if
 $\psi(x,p)\leq \varphi(x,p)$, $x>0$, $p>0$.
Further, since $\psi$ and $\varphi$ are $Z$-functions, by part (ii) of Remark \ref{Rem_Thm_main},
 we have that if $\psi(x,p)\leq \varphi(x,p)$, $x>0$, $p>0$, then
 $\vartheta_{1,\psi}(x)\leq \vartheta_{1,\varphi}(x)$, $x>0$.
Consequently, Theorem \ref{Lem_psi_est_eq_5} yields that
 $\vartheta_{n,\psi}(x_1,\ldots,x_n)\leq \vartheta_{n,\varphi}(x_1,\ldots,x_n)$
 holds for all $n\in\NN$ and $x_1,\ldots,x_n>0$ if and only if $\lambda\leq \lambda^*$.

Now, suppose that $p\in\RR_{++}$ is known.
Similarly as above, we have \ $\Theta := (0,\infty)$ \ and \ $f:\RR\times (0,\infty)\to\RR$,
 \[
     f(x,\lambda) :=\begin{cases}
                      \dfrac{\lambda^p x^{p-1} \ee^{-\lambda x}}{\Gamma(p)} & \text{if $x>0$, $\lambda>0$,}\\
                      0 & \text{otherwise,}
                     \end{cases}
 \]
 and consequently, \ $\cX_f=(0,\infty)$.
Then \ $\psi:(0,\infty)\times (0,\infty)\to\RR$,
 \begin{align*}
   \psi(x,\lambda) = \frac{\partial_2 f(x,\lambda)}{f(x,\lambda)}
             = \frac{p}{\lambda} - x, \qquad  x,\lambda\in(0,\infty).
 \end{align*}
We have that $\psi$ is a Z-function, and for each $n\in\NN$ and $x_1,\ldots,x_n\in(0,\infty)$, the equation
 $\sum_{i=1}^n \psi(x_i,\lambda) =0$, $\lambda>0$, which is equivalent to
 \[
   \frac{pn}{\lambda} = \sum_{i=1}^n x_i,\qquad \lambda>0,
 \]
 has a unique solution $\vartheta_{n,\psi}(x_1,\ldots,x_n):=\frac{pn}{\sum_{i=1}^n x_i}$.
One can easily see that Theorem \ref{Lem_psi_est_eq_5} could also be applied in this case.
\proofend
\end{Ex}

\begin{Ex}[Lomax distribution]
Let $\alpha,\lambda\in\RR_{++}$ and let $\xi$ be a random variable having Lomax (Pareto type II) distribution with parameters $\alpha$ and $\lambda$,
 i.e., $\xi$ has a density function
 \[
     f_\xi(x) := \begin{cases}
                  \dfrac{\alpha}{\lambda}\left(1 + \dfrac{x}{\lambda}\right)^{-(\alpha+1)} & \text{if $x>0$,}\\[2mm]
                  0 & \text{if $x\leq 0$.}
               \end{cases}
 \]

Supposing that $\alpha\in\RR_{++}$ is known, we will establish the existence and uniqueness of a solution
 of the likelihood equation for $\lambda$ using Theorem 2.10 and Proposition 2.12 in Barczy and P\'ales \cite{BarPal2}.
In the considered case, using the setup given before Example 4.6 in Barczy and P\'ales \cite{BarPal2}, we have
 \ $\Theta := (0,\infty)$ \ and \ $f:\RR\times (0,\infty)\to\RR$,
 \[
     f(x,\lambda) :=\begin{cases}
                  \dfrac{\alpha}{\lambda}\left(1 + \dfrac{x}{\lambda}\right)^{-(\alpha+1)} & \text{if $x>0$, $\lambda>0$,}\\[2mm]
                  0 & \text{otherwise,}
                \end{cases}
 \]
 and consequently, \ $\cX_f=(0,\infty)$.
\ Then \ $\psi:(0,\infty)\times (0,\infty)\to\RR$,
 \begin{align}\label{Lomax_1}
  \begin{split}
   \psi(x,\lambda) & = \frac{\partial_2 f(x,\lambda)}{f(x,\lambda)} =  \partial_2 (\ln f)(x,\lambda)
                     = \partial_2\left( \ln(\alpha) - (\alpha+1) \ln\left(1+\frac{x}{\lambda}\right) - \ln(\lambda)\right)\\
                   & = \frac{\alpha x - \lambda}{\lambda(\lambda + x)}
  \end{split}
 \end{align}
 for $x,\lambda\in(0,\infty)$.
Then $\psi$ is a $Z_1$-function with $\vartheta_{1,\psi}(x) = \alpha x$, $x>0$, and $\Theta_\psi=\Theta=(0,\infty)$.
Further, for each $n\in\NN$ and $x_1,\ldots,x_n\in(0,\infty)$, we have
 \begin{align}\label{help_30}
  \sum_{i=1}^n \psi(x_i,\lambda) = \frac{\alpha}{\lambda} \sum_{i=1}^n \frac{x_i}{\lambda+x_i} - \sum_{i=1}^{n}\frac{1}{\lambda+x_i}
                                 = \frac{\alpha n}{\lambda} - (\alpha+1)\sum_{i=1}^n \frac{1}{\lambda+x_i},
                                 \qquad \lambda >0.
 \end{align}
Next, we check that $\psi$ is a Z-function.
For this, using that $\psi$ is continuous in its second variable, by part (vi) of Theorem 2.10 in Barczy and P\'ales \cite{BarPal2},
 it is enough to verify that for all $x,y>0$ with $\alpha x=\vartheta_{1,\psi}(x) < \vartheta_{1,\psi}(y) = \alpha y$,
 the function
 \[
   (\alpha x,\alpha y)\ni \lambda\mapsto
      -\frac{\psi(x,\lambda)}{\psi(y,\lambda)}
      = \frac{(\lambda - \alpha x)(\lambda +y)}{(\alpha y - \lambda)(\lambda +x)}
      = \frac{\lambda^2 + (y-\alpha x)\lambda - \alpha xy}{-\lambda^2 + (\alpha y-x)\lambda + \alpha xy}
      =: h(\lambda)
 \]
 is strictly increasing.
We check that $h'(\lambda)>0$ for all $\lambda>0$, which yields that $h$ is strictly increasing.
An algebraic calculation shows that
 \begin{align*}
  h'(\lambda)
    = \frac{(1+\alpha)(y-x)(\lambda^2 + \alpha xy)}
           {(-\lambda^2 + (\alpha y-x)\lambda + \alpha xy)^2}
    >0, \qquad \lambda \in (\alpha x,\alpha y),
 \end{align*}
 as desired.
As a consequence, using also \eqref{help_30}, for each $n\in\NN$ and $x_1,\ldots,x_n>0$, the likelihood equation for $\lambda$
 takes the form
 \[
   \frac{\alpha}{\alpha+1} = \frac{\lambda}{n}\sum_{i=1}^n \frac{1}{x_i+\lambda},\qquad \lambda>0,
 \]
 which has a unique solution $\vartheta_{n,\psi}(x_1,\ldots,x_n)$.

\noindent Moreover, let $\alpha^*\in\RR_{++}$ be also given, and let $\varphi:(0,\infty)\times (0,\infty)\to\RR$ be defined by
 \begin{align}\label{Lomax_2}
   \varphi(x,\lambda) := \frac{\alpha^* x - \lambda}{\lambda(\lambda + x)},\qquad x,\lambda\in(0,\infty).
 \end{align}
Then, as we have shown it above, $\varphi$ is a $Z$-function.

\noindent Next, we compare the $\psi$-estimators corresponding to the functions given in \eqref{Lomax_1} and \eqref{Lomax_2}.
By a direct computation, one can easily see that
 $\psi(x,\lambda)\leq \varphi(x,\lambda)$ holds for all $x>0$ and $\lambda>0$
 (i.e., the inequality in part (iv) of Theorem \ref{Lem_psi_est_eq_5} is satisfied with $p(\lambda):=1$, $\lambda>0$) if and only if $\alpha \leq \alpha^*$.
Note also that $\vartheta_{1,\psi}(x)\leq \vartheta_{1,\varphi}(x)$ holds for all $x>0$
 if and only if $\alpha \leq \alpha^*$ (following from $\vartheta_{1,\psi}(x) = \alpha x$, $x>0$,
 and $\vartheta_{1,\varphi}(x)=\alpha^*x$, $x>0$).
Consequently, Theorem \ref{Lem_psi_est_eq_5} yields that $\vartheta_{n,\psi}(x_1,\ldots,x_n) \leq \vartheta_{n,\varphi}(x_1,\ldots,x_n)$ holds for all
 $n\in\NN$ and $x_1,\ldots,x_n>0$ if and only if $\alpha\leq \alpha^*$.

Now, suppose that $\lambda\in\RR_{++}$ is known.
Similarly as above, we have \ $\Theta := (0,\infty)$ \ and \ $f:\RR\times (0,\infty)\to\RR$,
 \[
     f(x,\alpha) :=\begin{cases}
                  \dfrac{\alpha}{\lambda}\left(1 + \dfrac{x}{\lambda}\right)^{-(\alpha+1)} & \text{if $x>0$, $\alpha>0$,}\\[2mm]
                  0 & \text{otherwise,}
                \end{cases}
 \]
and consequently, \ $\cX_f=(0,\infty)$.
Then \ $\psi:(0,\infty)\times (0,\infty)\to\RR$,
 \begin{align}\label{Lomax_3}
   \psi(x,\alpha) = \partial_2\left( \ln(\alpha) - (\alpha+1)\ln\left(1+\frac{x}{\lambda}\right) - \ln(\lambda)\right)
                  =  \frac{1}{\alpha} - \ln\left(1+\frac{x}{\lambda}\right), \qquad  x,\alpha\in(0,\infty).
 \end{align}
Then $\psi$ is a $Z_1$-function with $\vartheta_{1,\psi}(x) = 1/\ln\big(1+\frac{x}{\lambda}\big)$, $x>0$, and $\Theta_\psi=\Theta=(0,\infty)$.
This, together with the fact that $\psi$ is strictly decreasing and continuous in its second variable, yields that $\psi$ is a Z-function, see Barczy and P\'ales \cite[Proposition 2.12]{BarPal2}.
Further, for all $n\in\NN$ and $x_1,\ldots,x_n\in(0,\infty)$, the likelihood equation for $\alpha$ takes the form
 \[
   \sum_{i=1}^n \psi(x_i,\alpha) = \frac{n}{\alpha} - \sum_{i=1}^n \ln\left(1+\frac{x_i}{\lambda}\right)=0,\qquad \alpha>0,
 \]
 which has a unique solution
 \[
  \vartheta_{n,\psi}(x_1,\ldots,x_n) = \frac{1}{\frac{1}{n}\sum_{i=1}^n \ln\left(1+\frac{x_i}{\lambda}\right)}.
 \]

In addition, let $\lambda^*\in\RR_{++}$, and \ $\varphi:(0,\infty)\times (0,\infty)\to\RR$,
 \begin{align}\label{Lomax_4}
   \varphi(x,\alpha) :=  \frac{1}{\alpha} - \ln\left(1+\frac{x}{\lambda^*}\right), \qquad  x,\alpha\in(0,\infty).
 \end{align}
Then $\varphi$ is a $Z$-function and $\Theta_\varphi=(0,\infty)$.

To compare the $\psi$-estimators corresponding to the functions given in \eqref{Lomax_3} and \eqref{Lomax_4}, a direct computation shows that $\psi(x,\alpha)\leq \varphi(x,\alpha)$ holds for all $x>0$ and $\alpha>0$
(i.e., the inequality in part (iv) of Theorem \ref{Lem_psi_est_eq_5} is satisfied with $p(\alpha):=1$, $\alpha>0$) if and only if $\lambda \leq \lambda^*$.
Moreover, again by a direct computation, $\vartheta_{n,\psi}(x_1,\ldots,x_n) \leq \vartheta_{n,\varphi}(x_1,\ldots,x_n)$ holds for all $n\in\NN$ and $x_1,\ldots,x_n>0$ if and only if $\lambda\leq \lambda^*$.
In fact, Theorem \ref{Lem_psi_est_eq_5} also directly implies that $\vartheta_{n,\psi}(x_1,\ldots,x_n) \leq \vartheta_{n,\varphi}(x_1,\ldots,x_n)$ holds for all
 $n\in\NN$ and $x_1,\ldots,x_n>0$ if and only if $\lambda\leq \lambda^*$.
\proofend
\end{Ex}

\begin{Ex}[Lognormal distribution]
Let $\mu\in\RR$, $\sigma^2\in\RR_{++}$ and let $\xi$ be a random variable having lognormal distribution with parameters $\mu$ and $\sigma^2$,
 i.e., $\xi$ has a density function
 \[
     f_\xi(x) := \begin{cases}
                  \dfrac{1}{\sigma \sqrt{2\pi}\, x}\exp\left\{ - \dfrac{(\ln(x)-\mu)^2}{2\sigma^2}\right\}  & \text{if $x>0$,}\\[2mm]
                  0 & \text{if $x\leq 0$.}
               \end{cases}
 \]

Supposing that $\sigma^2>0$ is known, we will establish the existence and uniqueness of a solution
 of the likelihood equation for $\mu$ using Proposition 2.12 in Barczy and P\'ales \cite{BarPal2}.
In the considered case, using the setup given before Example 4.6 in Barczy and P\'ales \cite{BarPal2}, we have
 \ $\Theta := \RR$ \ and \ $f:\RR\times \RR\to\RR$,
 \[
   f(x,\mu) := \begin{cases}
                  \dfrac{1}{\sigma \sqrt{2\pi} \,x}\exp\left\{- \dfrac{(\ln(x)-\mu)^2}{2\sigma^2} \right\}  & \text{if $x>0$, $\mu\in\RR$,}\\[2mm]
                  0 & \text{otherwise,}
               \end{cases}
 \]
 and consequently, \ $\cX_f=(0,\infty)$.
Then \ $\psi:(0,\infty)\times\RR\to\RR$,
 \begin{align*}
   \psi(x,\mu) = \partial_2\left( - \frac{(\ln(x)-\mu)^2}{2\sigma^2}  - \ln(x) - \ln(\sigma\sqrt{2\pi}) \right)
                  =  \frac{\ln(x) - \mu}{\sigma^2}, \qquad  x\in(0,\infty),\quad \mu\in\RR.
 \end{align*}
Then $\psi$ is a $Z_1$-function with $\vartheta_{1,\psi}(x) = \ln(x)$, $x>0$, and $\Theta_\psi=\Theta=\RR$.
This, together with the fact that $\psi$ is strictly decreasing and continuous in its second variable, yields that $\psi$ is a Z-function, see Barczy and P\'ales \cite[Proposition 2.12]{BarPal2}.
Further, for all $n\in\NN$ and $x_1,\ldots,x_n\in(0,\infty)$, the likelihood equation for $\mu$ takes the form
 \[
   \sum_{i=1}^n \psi(x_i,\mu) = \frac{1}{\sigma^2}\sum_{i=1}^n \ln(x_i) - \frac{n\mu}{\sigma^2}=0,\qquad \mu\in\RR,
 \]
 which has a unique solution
 \[
  \vartheta_{n,\psi}(x_1,\ldots,x_n) = \frac{1}{n}\sum_{i=1}^n \ln(x_i).
 \]
Note that $\vartheta_{n,\psi}(x_1,\ldots,x_n)$ does not depend on $\sigma^2$.

\noindent Moreover, let $\sigma_*^2>0$ be also given, and \ $\varphi:(0,\infty)\times \RR\to\RR$,
 \begin{align*}
   \varphi(x,\mu) :=  \frac{\ln(x) - \mu}{\sigma_*^2}, \qquad  x\in(0,\infty),\quad \mu\in\RR.
 \end{align*}
Then $\varphi$ is a $Z$-function, $\vartheta_{1,\varphi}(x) = \ln(x) = \vartheta_{1,\psi}(x)$, $x>0$, and
 \[
   \psi(x,\mu) = \frac{\sigma_*^2}{\sigma^2}\,\varphi(x,\mu),\qquad x>0,\quad \mu\in\RR.
 \]
Hence the inequality in part (iv) of Theorem \ref{Lem_psi_est_eq_5} holds with $p(\mu):=\frac{\sigma_*^2}{\sigma^2}$, $\mu\in\RR$.
In fact, we get
 \[
 \vartheta_{n,\psi}(x_1,\ldots,x_n) = \frac{1}{n}\sum_{i=1}^n \ln(x_i) = \vartheta_{n,\varphi}(x_1,\ldots,x_n)
 \]
 for all $n\in\NN$ and $x_1,\ldots,x_n\in(0,\infty)$.
\proofend
\end{Ex}

\begin{Ex}[Laplace distribution]
Let $\mu\in\RR$, $b\in\RR_{++}$ and let $\xi$ be a random variable having Laplace distribution with parameters $\mu$ and $b$,
 i.e., $\xi$ has a density function
 \[
     f_\xi(x) := \frac{1}{2b} \ee^{-\frac{\vert x-\mu\vert}{b}},\qquad x\in\RR.
 \]
Supposing that $\mu\in\RR$ is known, we will establish the existence and uniqueness of a solution
 of the likelihood equation for $b$ using Theorem 2.10 in Barczy and P\'ales \cite{BarPal2}.
In the considered case, using the setup given before Example 4.6 in Barczy and P\'ales \cite{BarPal2}, we have
 \ $\Theta := (0,\infty)$ \ and \ $f:\RR\times (0,\infty)\to\RR$,
 \[
   f(x,b) := \frac{1}{2b} \ee^{-\frac{\vert x-\mu\vert}{b}},\qquad x\in\RR, \quad b>0,
 \]
 and consequently, \ $\cX_f=\RR$.
Further, let us take $X:=\RR\setminus\{\mu\}\subseteq \cX_f$.
Then \ $\psi:(\RR\setminus\{\mu\})\times(0,\infty)\to\RR$,
 \begin{align*}
   \psi(x,b) = \partial_2\left( - \frac{\vert x-\mu\vert}{b}  - \ln(2) - \ln(b) \right)
               = \frac{\vert x-\mu\vert}{b^2} - \frac{1}{b}, \qquad  x\in\RR\setminus\{\mu\},\quad b>0.
 \end{align*}
Then $\psi$ is a $Z_1$-function with $\vartheta_{1,\psi}(x) = \vert x-\mu\vert$, $x\in\RR\setminus\{\mu\}$.
Further, if $x,y\in\RR\setminus\{\mu\}$ are such that $\vert x-\mu\vert = \vartheta_{1,\psi}(x) < \vartheta_{1,\psi}(y) = \vert y-\mu\vert$, then the function
 \[
  (\vert x-\mu\vert, \vert y-\mu\vert)\ni b \mapsto
        - \frac{\psi(x,b)}{\psi(y,b)}
        = - \frac{\vert x-\mu\vert - b}{\vert y-\mu\vert - b}
        = -1 + \frac{\vert y-\mu\vert - \vert x-\mu\vert}{\vert y-\mu\vert - b}
 \]
is strictly increasing. This, together with the fact that $\psi$ is a $Z_1$-function, yields that $\psi$ is a $T$-function, see part (vi) of Theorem 2.10 in Barczy and P\'ales \cite{BarPal2}.
Since $\psi$ is continuous in its second variable as well, we get that $\psi$ is a $Z$-function as well.
By a direct calculation, for all $n\in\NN$ and $x_1,\ldots,x_n\in\RR\setminus\{\mu\}$, the likelihood equation for $b$
 takes the form
 \[
  \sum_{i=1}^n \psi(x_i,b) = -\frac{1}{b^2} \sum_{i=1}^n \vert x_i - \mu \vert - \frac{n}{b}=0,\qquad b>0,
 \]
 which has the unique solution
 \[
 \vartheta_{n,\psi}(x_1,\ldots,x_n) = \frac{1}{n}\sum_{i=1}^n \vert x_i - \mu \vert.
 \]

Furthermore, if $\mu^*\in\RR$ and $\varphi:\RR\setminus\{\mu^*\}\times(0,\infty)\to\RR$,
 \begin{align*}
   \varphi(x,b) := \frac{\vert x-\mu^*\vert}{b^2} - \frac{1}{b}, \qquad  x\in\RR\setminus\{\mu^*\},\quad b>0,
 \end{align*}
then $\varphi$ is a $Z$-function with $\vartheta_{1,\varphi}(x) = \vert x-\mu^*\vert$, $x\in\RR\setminus\{\mu^*\}$.
By a direct computation, one can easily see that $\vartheta_{1,\psi}(x)\leq \vartheta_{1,\varphi}(x)$ holds for all $x\in \RR\setminus\{\mu,\mu^*\}$ if and only if $\mu=\mu^*$.
This also implies that $\vartheta_{n,\psi}(x_1,\ldots,x_n)\leq \vartheta_{1,\varphi}(x_1,\ldots,x_n)$ holds for all $n\in\NN$ and $x_1,\ldots,x_n\in\RR\setminus\{\mu,\mu^*\}$ if and only if $\mu=\mu^*$.
\proofend
\end{Ex}

\section*{Declaration of competing interest}

The authors declare that they have no known competing financial interests or personal relationships
that could have appeared to influence the work reported in this paper.

\section*{Final publication}

This manuscript has been split into two parts and published under the titles \textit{Comparison and equality of generalized $\psi$-estimators} and \textit{Comparison and equality of Bajraktarević-type $\psi$-estimators} in the journals ``Annals of the Institute of Statistical Mathematics'' and ``REVSTAT-Statistical Journal'', respectively.

\addcontentsline{toc}{section}{References}
\bibliographystyle{plain}
\bibliography{psi_estimator_equal_char_bib}

\begin{thebibliography}{10}

\bibitem{BarPal2}
M.~Barczy and {\relax Zs}.~P{\'a}les.
\newblock Existence and uniqueness of weighted generalized $\psi$-estimators.
\newblock arXiv 2211.06026, 2022.

\bibitem{Dar71b}
Z.~Dar{\'o}czy.
\newblock A general inequality for means.
\newblock {\em Aequationes Math.}, 7(1):16--21, 1971.

\bibitem{DarPal82}
Z.~Dar{\'o}czy and {\relax Zs}.~P{\'a}les.
\newblock On comparison of mean values.
\newblock {\em Publ. Math. Debrecen}, 29(1-2):107--115, 1982.

\bibitem{Gon}
L.~Gong.
\newblock Establishing equalities of {OLSE}s and {BLUE}s under seemingly
  unrelated regression models.
\newblock {\em J. Stat. Theory Pract.}, 13(1):Paper No. 5, 10, 2019.

\bibitem{GruPal}
R.~Gr\"{u}nwald and {\relax Zs}.~P\'{a}les.
\newblock On the equality problem of generalized {B}ajraktarevi\'{c} means.
\newblock {\em Aequationes Math.}, 94(4):651--677, 2020.

\bibitem{GruPal2}
R.~Gr\"{u}nwald and {\relax Zs}.~P\'{a}les.
\newblock On the invariance of the arithmetic mean with respect to generalized
  {B}ajraktarevi\'{c} means.
\newblock {\em Acta Math. Hungar.}, 166(2):594--613, 2022.

\bibitem{Gus}
W.~Gustin.
\newblock On the interior of the convex hull of a {E}uclidean set.
\newblock {\em Bull. Amer. Math. Soc.}, 53:299--301, 1947.

\bibitem{Hub64}
P.~J. Huber.
\newblock Robust estimation of a location parameter.
\newblock {\em Ann. Math. Statist.}, 35:73--101, 1964.

\bibitem{Hub67}
P.~J. Huber.
\newblock The behavior of maximum likelihood estimates under nonstandard
  conditions.
\newblock In {\em Proc. {F}ifth {B}erkeley {S}ympos. {M}ath. {S}tatist. and
  {P}robability ({B}erkeley, {C}alif., 1965/66), {V}ol. {I}: {S}tatistics},
  pages 221--233. Univ. California Press, Berkeley, Calif., 1967.

\bibitem{JiaSun}
B.~Jiang and Y.~Sun.
\newblock On the equality of estimators under a general partitioned linear
  model with parameter restrictions.
\newblock {\em Statist. Papers}, 60(1):273--292, 2019.

\bibitem{Kos}
M.~R. Kosorok.
\newblock {\em Introduction to Empirical Processes and Semiparametric
  Inference}.
\newblock Springer Series in Statistics. Springer, New York, 2008.

\bibitem{KraBarFie}
W.~Kr\"{a}mer, R.~Bartels, and D.~G. Fiebig.
\newblock Another twist on the equality of {OLS} and {GLS}.
\newblock {\em Statist. Papers}, 37(3):277--281, 1996.

\bibitem{Lee}
S.~H. Lee.
\newblock Necessary and sufficient conditions for the equality between the two
  best linear unbiased estimators and their applications.
\newblock {\em Korean J. Appl. Statist.}, 6(1):95--103, 1993.

\bibitem{LuSch}
C.~Lu and P.~Schmidt.
\newblock Conditions for the numerical equality of the {OLS}, {GLS} and
  {A}memiya-{C}ragg estimators.
\newblock {\em Econom. Lett.}, 116(3):538--540, 2012.

\bibitem{Mat}
T.~Mathieu.
\newblock Concentration study of {M}-estimators using the influence function.
\newblock {\em Electron. J. Stat.}, 16(1):3695--3750, 2022.

\bibitem{MukDasBasChaBha}
S.~Mukhopadhyay, A.~J. Das, A.~Basu, A.~Chatterjee, and S.~Bhattacharya.
\newblock Does the generalized mean have the potential to control outliers?
\newblock {\em Comm. Statist. Theory Methods}, 50(8):1709--1727, 2021.

\bibitem{Pal88a}
{\relax Zs}.~P{\'a}les.
\newblock General inequalities for quasideviation means.
\newblock {\em Aequationes Math.}, 36(1):32--56, 1988.

\bibitem{PunSty}
S.~Puntanen and G.~P.~H. Styan.
\newblock The equality of the ordinary least squares estimator and the best
  linear unbiased estimator.
\newblock {\em Amer. Statist.}, 43(3):153--164, 1989.
\newblock With comments by Oscar Kempthorne and Shayle R. Searle and a reply by
  the authors.

\bibitem{TiaPun}
Y.~Tian and S.~Puntanen.
\newblock On the equivalence of estimations under a general linear model and
  its transformed models.
\newblock {\em Linear Algebra Appl.}, 430(10):2622--2641, 2009.

\bibitem{TiaWie}
Y.~Tian and D.~P. Wiens.
\newblock On equality and proportionality of ordinary least squares, weighted
  least squares and best linear unbiased estimators in the general linear
  model.
\newblock {\em Statist. Probab. Lett.}, 76(12):1265--1272, 2006.

\bibitem{Vaa}
A.~W. van~der Vaart.
\newblock {\em Asymptotic {S}tatistics}, volume~3 of {\em Cambridge Series in
  Statistical and Probabilistic Mathematics}.
\newblock Cambridge University Press, Cambridge, 1998.

\end{thebibliography}

\end{document}